\documentclass{amsart}
\usepackage[all]{xy}
\usepackage{latexsym}
\usepackage{amsmath}
\usepackage[T1]{fontenc}
\usepackage{amssymb,amscd}
\usepackage{setspace}
\usepackage{mathdots}
\usepackage[mathscr]{eucal}
\usepackage{bbm}
\usepackage{stmaryrd}
\usepackage{color}
\usepackage[colorlinks,linkcolor=blue]{hyperref}
\usepackage{hyperref}
\usepackage{mathrsfs}

\makeatletter
\renewcommand\@makefnmark{\mbox{\textsuperscript{\normalfont(\@thefnmark)}}}
\makeatother
 
\newtheorem{theorem}{Theorem}[section] 
\newtheorem{lemma}[theorem]{Lemma}     
\newtheorem{corollary}[theorem]{Corollary}
\newtheorem{proposition}[theorem]{Proposition}
\newtheorem{remark}[theorem]{Remark}
\newtheorem{definition}[theorem]{Definition}

\numberwithin{equation}{section}

\newcommand{\p}{\mathfrak{p}}
\newcommand{\fm}{\mathfrak{m}}

\newcommand{\cF}{\mathcal{F}}
\newcommand{\cO}{\mathcal{O}}

\newcommand{\cS}{\mathcal{S}}

\newcommand{\F}{\mathbb F}
\newcommand{\EE}{\mathrm{E}}

\newcommand{\N}{\mathbb N}
\newcommand{\Q}{\mathbb Q}

\newcommand{\C}{\mathbb C}
\newcommand{\Z}{\mathbb Z}

\newcommand{\bP}{\mathbb{P}}

\newcommand{\et}{\mathrm{\acute{e}t}}

\newcommand{\sM}{\mathscr{M}}
\newcommand{\sB}{\mathscr{B}}

\newcommand{\ra}{\rightarrow}
\newcommand{\lra}{\longrightarrow}

\newcommand{\GL}{\mathrm{GL}}
\newcommand{\SL}{\mathrm{SL}}
\newcommand{\JL}{\mathrm{JL}}

\newcommand{\bQp}{\overline{\Q}_p}

\newcommand{\Sym}{\mathrm{Sym}}

\newcommand{\brho}{\overline{\rho}}

\newcommand{\ide}{\mathbf{1}}

\newcommand{\xto}[1][]{\xrightarrow{#1}}
\newcommand{\simto}{
\xto[\sim]} 

\providecommand{\ligne}{\textbf{---}}

\newcommand{\matr}[4]{\begin{pmatrix}{#1}&{#2}\\ {#3}&{#4}\end{pmatrix}}
\newcommand{\smatr}[4]{\bigl(\begin{smallmatrix} {#1}& {#2}\\ {#3}&{#4}\end{smallmatrix}\bigl)}

\def\into{\hookrightarrow}

\DeclareMathOperator{\End}{{\mathrm{End}}}
\DeclareMathOperator{\Ext}{{\mathrm{Ext}}}

\DeclareMathOperator{\Gal}{{\mathrm{Gal}}}
\DeclareMathOperator{\gr}{{\mathrm{gr}}}
\DeclareMathOperator{\Hom}{{\mathrm{Hom}}}

\DeclareMathOperator{\Ind}{{\mathrm{Ind}}}

\DeclareMathOperator{\JH}{{\mathrm{JH}}}

\DeclareMathOperator{\Ker}{{\mathrm{Ker}}}

\DeclareMathOperator{\Mod}{\mathrm{Mod}}

\DeclareMathOperator{\soc}{{\mathrm{soc}}}
\DeclareMathOperator{\Sp}{{\mathrm{Sp}}}

\def\To#1{\buildrel\hbox{\tiny{$#1$}}\over\longrightarrow}

\newcommand{\sta}{${}^*$}

\newcommand{\quash}[1]{}

\begin{document}

\title{On some $p$-adic and mod $p$ representations of quaternion algebra over $\Q_p$}

\author{Yongquan HU \and Haoran WANG }

\maketitle

\begin{abstract}
Let $D$ be the non-split quaternion algebra over $\Q_p$. We prove  that a class of admissible unitary Banach space representations of $D^{\times}$ of global origin are topologically of finite length. 
\end{abstract}

\setcounter{tocdepth}{1}
\tableofcontents

\section{Introduction}

Let $p$ be a prime number, $G=\GL_2(\Q_p)$ and $D$ be the non-split quaternion algebra over $\Q_p$. Let $E$ be a finite extension of $\Q_p$ with ring of integers $\cO$ and residue field $\F$; they will serve as rings of coefficients.  

In \cite{Scholze} Scholze constructed
a cohomological covariant $\delta$-functor $\{\cS^i,i \geq 0\}$ from the category of admissible smooth representations of
$G$ over  $\cO$-torsion modules to admissible smooth representations of $D^{\times}$ carrying a continuous and commuting
action of $G_{\Q_p}:=\Gal(\bQp/\Q_p)$. 
The functor extends to admissible unitary Banach space representations $\Pi$ (see \cite{Paskunas-JL}), by setting
\[\cS^i(\Pi):=\big(\varprojlim_n\cS^i(\Theta/\varpi^n)\big)_{\rm tf}\otimes_{\cO}E\]
where $\Theta$ is some open bounded $G$-invariant lattice in $\Pi$, and  the subscript ${\rm tf}$ means taking the maximal Hausdorff torsion-free quotient.

\medskip

It is hoped for that Scholze's functor realizes both $p$-adic local Langlands and Jacquet-Langlands correspondence. Let $\rho:G_{\Q_p}\ra \GL_2(E)$ be a continuous $p$-adic representation, and $\Pi(\rho)$ be the associated unitary admissible Banach space representation of $G$ by \cite{Co}.  
The natural candidate for a $p$-adic Jacquet-Langlands correspondence is to take $\cS^1(\Pi(\rho))$.  More concretely, since there is also an action of $G_{\Q_p}$ on $\cS^1(\Pi(\rho))$,  commuting with $D^{\times}$,  it seems more reasonable to factor out the action of $G_{\Q_p}$ by putting (for a suitable normalization)
\[\JL(\rho):=\Hom_{G_{\Q_p}}\big(\rho(-1),\cS^1(\Pi(\rho))\big).\]
One of the central questions in this version of $p$-adic  Jacquet-Langlands correspondence is as follows: \medskip

\textbf{Question}. Is $\JL(\rho)$ topologically of finite length ? 
\medskip

The difficulties to attack  this question are (at least) two-fold. Firstly, a useful way to control the length of a unitary Banach space representation is by bounding the length of its residual representation (e.g.~this method works well in the case of $G$). However, it  does not work here because  $\JL(\rho)$ is   \emph{residually of infinite length} (cf.~ \cite[Thm.~7.8]{Scholze}).  Secondly, although $\JL(\rho)$ often has $\delta$-dimension $1$ (see \cite{Paskunas-JL}, \cite{HW-JL1}),  it is hard to exclude the possibility that $\JL(\rho)$ admits infinitely many irreducible subquotients of $\delta$-dimension $0$ (equivalently,  finite dimensional as  $E$-vector spaces); see \cite[Rem.~1.3]{DPS22} for a related discussion.

Assuming that the difference of the Hodge-Tate-Sen weights of $\rho$ is not a non-zero integer,  Dospinescu, Pa\v{s}k\=unas and Schraen (\cite{DPS22}) recently  proved that $\JL(\rho)$ is indeed topologically of finite length. The point is that this assumption forces that $\JL(\rho)$ does not admit any irreducible subquotients which are finite dimensional over $E$.  In this paper, we give an affirmative answer to the question   for  those $\rho$ having a  global origin. 
\medskip

Let us fix the global setup. Let $B_0$ be a definite quaternion algebra over\footnote{In the main body of the paper, we work with a totally real field $F$ such that $F_{\p}\cong\Q_p$ for a place above $p$.} $\Q$,  split at $p$.     Let $U=\prod_vU_v$ be a compact open subgroup  of $(B_0\otimes_{\Q}\mathbb{A}_{\Q}^f)^{\times}$; we assume that $U$ is sufficiently small. 
Write $U^{p}=\prod_{v\neq p}U_v$.
Let $S(U^p,\cO)$  be the space of continuous functions 
\[f:B_0^{\times}\backslash (B_0\otimes_{\Q}\mathbb{A}_{\Q}^f)^{\times}/U^p\ra \cO.\]  
Given a continuous character $\psi:(\mathbb{A}_{\Q}^f)^{\times}/{\Q}^{\times}\ra \cO^{\times}$, we let $S_{\psi}(U^p,\cO)$ be the subspace of $S(U^p,\cO)$ consisting of functions on which $(\mathbb{A}_{\Q}^{f})^{\times}$ acts by  $\psi$.  
The group $(B_0\otimes_{\Q}\Q_{p})^{\times}\cong \GL_2(\Q_p)$ acts continuously on $S_{\psi}(U^p,\cO)$ by right translation. There is a certain Hecke algebra $\mathbb{T}(U^{p})$ acting continuously on $S_{\psi}(U^{p},\cO)$.

Fix a continuous absolutely irreducible odd representation 
$\overline{r}: G_{\Q}\ra \GL_2(\F)$, 
unramified outside all the ``bad'' places. 
It corresponds to a maximal ideal $\fm$ of $\mathbb{T}(U^p)$ and we assume that the localization $S_{\psi}(U^{p},\cO)_{\fm}$ is non-zero, i.e. $\overline{r}$ is modular. It is well-known that there is an associated $2$-dimensional continuous Galois representation
\[r_{\fm}: G_{\Q}\ra \GL_2(\mathbb{T}(U^{p})_{\fm})\]
deforming $\overline{r}$. As a consequence, for any $y\in \mathrm{Spm}(\mathbb{T}(U^{p})_{\fm}[1/p])$, it specializes to  a $p$-adic representation $r_{y}:G_{\Q}\ra \GL_2(\kappa(y))$, where $\kappa(y)$ is the residue field of $\mathbb{T}(U^{p})_{\fm}[1/p]$ at $y$.

Our main result is the following.
\begin{theorem}\label{thm:intro-main}
Assume that $\overline{r}(1)|_{G_{\Q_p}}$ is generic (in the sense of Definition \ref{def:generic}).  Let $y\in \mathrm{Spm}(\mathbb{T}(U^p)_{\fm}[1/p])$  such that  $\rho_y:=r_y(1)|_{G_{\Q_p}}$ is irreducible. Then $\JL(\rho_y)$ is topologically of finite length, and the quotient $\JL(\rho_y)/\JL(\rho_y)^{\rm lalg}$ is topologically irreducible. 
\end{theorem}

\begin{remark}
Under the assumptions of Theorem \ref{thm:intro-main}, if moreover the difference of the  Hodge-Tate-Sen weights of $\rho$ is not a non-zero integer, then $\JL(\rho)^{\rm lalg}=0$  and $\JL(\rho)$ itself is topologically irreducible.  This strengthens the aforementioned result of \cite{DPS22} and confirms an expectation  mentioned in \cite[Rem.~1.4]{DPS22}. 
\end{remark}

The proof of Theorem \ref{thm:intro-main} is based on a general  criterion for an admissible unitary Banach space representation of $D^{\times}$  to be  topologically of finite length. To explain it we need to introduce more notation. 

Let $\cO_D$ be the ring of integers in $D$ with a fixed uniformizer $\varpi_D$. Let $U_D^1=1+\varpi_D\cO_D$  which is the Sylow pro-$p$ subgroup of $\cO_D^{\times}$, and  $Z_D^1$ be the centre of $U_D^1$. The Iwasawa algebra $\Lambda:=\F[\![U_D^1/Z_D^1]\!]$ is an Auslander regular ring; let $\fm_D$ be its maximal ideal. The graded ring $\gr_{\fm_D}(\Lambda)$  has a nice structure and is studied in \cite[\S2]{HW-JL1} (motivated by the work \cite{BHHMS1} in the case of $\GL_2(\Q_{p^f})$). In particular, it contains a distinguished two-sided ideal, denoted by $J$, such that 
$\gr_{\fm_D}(\Lambda)/J\cong\F[y,z]/(yz).$ 
Consider the category $\mathcal{C}$   of admissible smooth $\F$-representations $\pi$ of $D^{\times}$ with a central character such that the graded module $\gr_{\fm_D}(\pi^{\vee})$ is annihilated by $J^n$ for some $n\geq 1$.   
For $\pi\in\mathcal{C}$, we may define its \emph{multiplicity}, denoted by $\mu(\pi)$, which measures the size of $\pi$ and such that $\mu(\pi)=0$ if and only if $\dim_{\F}\pi<\infty$, see \S\ref{sec:lowerbound} for details. 

Let $\Pi$ be an admissible unitary Banach space representation of $D^{\times}$ over $E$ which admits a central character. Denote by $\Pi^{\rm lalg}$ the subspace of locally algebraic vectors in $\Pi$.

\begin{theorem}\label{thm:intro-fl}
Assume that 
\begin{itemize}
\item[(a)] $\Pi^{\rm lalg}$ is finite dimensional over $E$,
\item[(b)] for some open bounded $D^{\times}$-lattice $\Theta$ in $\Pi$, $\Theta/\varpi\Theta\in\mathcal{C}$  and $\mu(\Theta/\varpi\Theta)\leq 4$.
\end{itemize}
Then  $\Pi$ is topologically of finite length.  
If moreover
\begin{itemize}
\item[(c)] $\mu(\Theta/\varpi\Theta)=4$ and $\Theta/\varpi\Theta$  does not admit any  quotient of finite dimension over $\F$,
\end{itemize}
then $\Pi/\Pi^{\rm lalg}$ is topologically irreducible.
\end{theorem}

We first explain how to deduce Theorem \ref{thm:intro-main} from   Theorem \ref{thm:intro-fl}, which requires to verify the conditions (a), (b), (c)  for $\Pi=\JL(\rho_y)$.  Since $\rho_y$ has a global origin, the condition (a)  is satisfied thanks to the local-global compatibility with the classical local Jacquet-Langlands correspondence.  Under the genericity assumption on $\brho$, the condition (c) is proved in our previous work \cite{HW-JL1}.  Again by \cite{HW-JL1}, the condition (b) is a consequence of the following multiplicity one result.  

\begin{theorem}\label{thm:intro-one}
Let   $\brho:G_{\Q_p}\ra \GL_2(\F)$ be a continuous representation and $\pi(\brho)$ be the admissible smooth representation of $G$ attached to $\brho$ by the mod $p$ local Langlands correspondence.
Assume that $\brho$ is generic (in the sense of Definition \ref{def:generic}). Then 
\[\cS^1(\pi(\brho))\cong \brho(-1)\otimes \JL(\brho)\] for some admissible smooth representation $\JL(\brho)$ of $D^{\times}$ whose $\cO_D^{\times}$-socle is multiplicity free. 
\end{theorem}

In fact, we will prove  for a slightly broader class of $\brho$. 
The proof of Theorem \ref{thm:intro-one} is based on the work of Colmez-Dospinescu-Niziol \cite{CDN22}. Andrea Dotto recently proved 
a  relative version of Theorem \ref{thm:intro-one} in a more general setting, see  \cite[Thm.~1.2]{Dotto22}. \medskip 
 
Let us now sketch  the proof of Theorem \ref{thm:intro-fl}. We may assume $\Pi$ is  infinite dimensional over $E$. The condition (a) allows us to construct a closed subrepresentation $\Pi_1\subset \Pi$, which is infinite dimensional and is of finite length. Letting $\Theta_1=\Pi_1\cap\Theta$, we obtain an embedding \[ \Theta_1/\varpi\Theta_1\hookrightarrow  \Theta/\varpi\Theta.\]
The following result, which is the technical heart of the paper, shows that $\mu(\Theta_1/\varpi\Theta_1)\geq 4$,  hence $\mu(\Theta_1/\varpi\Theta_1)=\mu(\Theta/\varpi\Theta)$ by (b).  This implies  that  $\Pi/\Pi_1$ is finite dimensional. If moreover (c) holds, then $\Theta_1/\varpi\Theta_1=\Theta/\varpi\Theta$ and consequently $\Pi_1=\Pi$. 

\begin{theorem}
\label{thm:intro-mu}
For any $\pi\in\mathcal{C}$ which is infinite dimensional, we have $\mu(\pi)\geq 4$.  
\end{theorem}

Due to the action of $\varpi_D$, it is easy to see that $\mu(\pi)\in 2\Z_{>0}$ (as $\pi $ is infinite dimensional). However, a lower bound by $4$  is somewhat surprising, if we compare with the case of $G=\GL_2(\Q_p)$.  In fact,  the graded ring of $\F[\![I_1/Z_1]\!]$ has a similar structure  as $\Lambda$, where $I_1$  is the pro-$p$ Iwahori subgroup of $G$ with centre $Z_1$,  and it is easy to show that principal series of $G$ do have multiplicity $2$ (cf. \cite[Prop.~3.3.3.4]{BHHMS2}).  
The key in  the proof of Theorem \ref{thm:intro-mu} is to show that there exist no smooth  representations of $\cO_D^{\times}/Z_D^1$  (or  $U_D^1/Z_D^1$) having multiplicity $1$; again such representations of $I_1/Z_1$ exist, cf.~\cite[Prop.~7.2]{Pa10}.
\bigskip

The paper is organized as follows. In Section \ref{sec:modp}, we study mod $p$   representation theory of $\cO_D^{\times}$; this   is the most technical part of the   paper.  Theorem \ref{thm:intro-mu} and the criterion Theorem \ref{thm:intro-fl} are proved in Section \ref{sec:criterion}. In Section \ref{sec:multione}, we prove  Theorem \ref{thm:intro-one} on the multiplicity one property of $\JL(\brho)$. In Section \ref{sec:main}, we put ourselves in global situation and prove the main result Theorem \ref{thm:intro-main}.   In the appendix Section \ref{sec:app}, we recall some results from \cite{Bjork} about Gabber filtrations.

\subsection{Notation}
We fix a prime number $p\geq 5$. Let $E$ be a finite extension of $\Q_p$, with ring of integers $\cO$ and residue field $\F$. Fix a uniformizer $\varpi$ of $E$. We will assume that $E$ and $\F$ are sufficiently large.

If $F$ is a field, let $G_F:=\Gal(\overline{F}/F)$ denote its absolute Galois group. Let $\varepsilon$ denote the $p$-adic cyclotomic character of $G_F$, and $\omega$ the mod $p$ cyclotomic character. 

If $R$ is a ring and $M$ is a left $R$-module, we denote by $\soc_RM$ (resp. $\mathrm{cosoc}_RM$) the socle (resp. cosocle) of $M$. When $R$ is a group algebra, for example $R=\F[G]$, we simply write $\soc_GM$ instead of $\soc_{\F[G]}M$.  If $M$ has finite length, we denote by $\JH(M)$ the multi-set of Jordan-H\"older factors of $M$. 

\subsection{Acknowledgements}
We thank Yiwen Ding,  Vytautas Pa\v{s}k\=unas, Zicheng Qian and Yichao Tian for several interesting  discussions during the preparation of the paper, and we thank  Gabriel Dospinescu and Weizhe Zheng for answering our questions.  

Y.~H. is partially supported by  National Key R$\&$D Program of China 2020YFA0712600; CAS Project for Young Scientists in Basic Research, Grant No.~YSBR-033;  National Natural Science Foundation of China Grants 11971028 and 12288201; National Center for Mathematics and Interdisciplinary Sciences and Hua Loo-Keng Key Laboratory of Mathematics, Chinese Academy of Sciences. H.~W. is partially supported by National Natural Science Foundation of China Grants 11901331 and 11971028.

 \section{Mod $p$ representations of $D^{\times}$}\label{sec:modp}

Let $D$ be the unique non-split quaternion algebra over $\Q_p$, with ring of integers $\cO_D$. In this section, we prove several results on finite dimensional mod $p$ representations of $\cO_D^{\times}$.

Fix a  uniformizer $\varpi_D$ of $D^{\times}$ satisfying $\varpi_D^2=p$. Then 
\[D\cong\Q_{p^2}\langle\varpi_D\rangle/(\varpi_D^2=p,\varpi_D^{-1}a\varpi_D=\sigma(a), a\in\Q_{p^2}),\]
where $\sigma$ is the Frobenius element on $\Q_{p^2}$.  
 For $n\geq 1$, let  
\[U_D^n:=1+\varpi_D^n\cO_D.\] 
Each $U_D^n$ is a normal subgroup of $\cO_D^{\times}$. Moreover,  $U_D^1$ is the (unique) Sylow pro-$p$ subgroup of $\cO_D^{\times}$ and $\cO_D^{\times}/U_D^1\cong \F_{q}^{\times}$ where $q:=p^2$. Let $[\cdot]$ denote the Teichm\"uller lifting $\F_q\hookrightarrow \Z_q\hookrightarrow \cO_D$, and $H<\cO_D^{\times}$ denote the image of $\F_q^{\times}$.  

Let $Z_D$ denote the centre of  $D^{\times}$ and $Z_D^1:=Z_D\cap U_D^1$. 
We will only consider  mod $p$  representations of $D^{\times}$ (or its subgroups)  with central character. Since $Z_D^1$ is pro-$p$, it acts trivially on such representations. Let $\fm_D$ denote the maximal ideal of the  Iwasawa algebra \[\Lambda:=\F[\![U_D^1/Z_D^1]\!].\] If $\pi$ is a smooth $\F$-representation of $\cO_D^{\times}/Z_D^1$, it is clear that $\soc_{\cO_D^{\times}}\pi$ is nothing but  the subspace of $U_D^1$-invariants $\pi^{U_D^1}=\pi[\fm_D]$, where \[\pi[\fm_D]:=\{v\in \pi: xv=0,\ \forall x\in \fm_D\}. \]

Fix once for all an embedding $\F_q\hookrightarrow \F$. Throughout we assume $p\geq 5$. 

\subsection{Extensions of characters}

 Let $\alpha:\cO_{D}^{\times}\ra\F^{\times}$ denote the character defined by 
\[\alpha(x):=\overline{x}^{p-1},\]
where $\overline{x}\in \F$ denotes the reduction of $x$ via $\cO_D/\varpi_D\simto \F_q$ and the fixed embedding $\F_q\hookrightarrow \F$.  
Since $\F_{q}^{\times}$ is a cyclic group of order $q-1$, $\alpha$ is of order $p+1$. In particular, we deduce that $\alpha\neq \alpha^{-1}$ and $\alpha\neq \alpha^{-2}$ (as $p>2$).

\begin{proposition}\label{prop-Ext1-U1}
Let $\psi,\chi:\cO_D^{\times}\ra\F^{\times}$ be two smooth characters and $i\in\{1,2\}$. Then
  $\Ext^i_{\cO^{\times}_D/Z_D^1}(\psi,\chi)$ is non-zero if and only if $\psi=\chi\alpha$ or $\psi=\chi\alpha^{-1}$. Moreover,
\[\dim_{\F}\Ext^i_{\cO_D^{\times}/Z_D^1}(\chi\alpha,\chi)=\dim_{\F}\Ext^i_{\cO_D^{\times}/Z_D^1}(\chi\alpha^{-1},\chi)=1.\]
\end{proposition}
\begin{proof}
See \cite[Prop.~2.13, Prop.~2.14]{HW-JL1}.
\end{proof}

\begin{definition}
(i) Let $\kappa:\cO_D^{\times}/Z_D^1\ra \F$ be the function defined by $\kappa(g)=\kappa_1(g{[\overline{g}]}^{-1})$ where 
\[\kappa_1(1+\varpi_D x):=\overline{x}.\]

(ii) Let $\kappa':\cO_D^{\times}/Z_D^1\ra \F$ be the function defined by $\kappa'(g)=\kappa_1'(g{[\overline{g}]}^{-1})$ where 
\[ \kappa_1'(1+\varpi_Dx):=\overline{x}^p.\]
\end{definition}

It is direct to check that 
\[\kappa(gg')=\kappa(g)+\alpha(g)\kappa(g'),\ \ \kappa'(gg')=\kappa'(g)+\alpha^{-1}(g)\kappa'(g')\]
for $g,g'\in\cO_D^{\times}/Z_D^1$. 
In other words, $\kappa$ (resp.~$\kappa'$) defines an element in $H^1(\cO_D^{\times}/Z_D^1,\alpha)$ (resp. $H^1(\cO_D^{\times}/Z_D^1,\alpha^{-1})$).

Let $\chi:\cO_D^{\times}\ra\F^{\times}$ be a smooth character. We denote by $E^-(\chi)$ the unique non-split extension
\begin{equation}\label{equation-R-chi}0\ra \chi\ra E^-(\chi)\ra \chi\alpha^{-1}\ra0.\end{equation}
Explicitly, we may choose a basis $\{v_0,v_1\}$ of $E^-(\chi)$ such that  for any $g\in\cO_D^{\times}$:
\begin{equation}\label{equation-cocycle-R-chi}\left\{ {\begin{array}{ll}
gv_0=\chi(g)v_0\\  gv_1=\chi\alpha^{-1}(g)\cdot(v_1+\kappa(g)v_0).\\
\end{array}}\right.\end{equation}
We similarly define $E^{+}(\chi)\in \Ext^1_{\cO_D^{\times}/Z_D^1}(\chi\alpha,\chi)$ using the cocycle $\kappa'$.

\textbf{Convention}: \emph{in the rest of this section, we will omit the subscript $\cO_D^{\times}/Z_D^1$ when we consider $\Ext^i_{\cO_D^{\times}/Z_D^1}$.}
\medskip

We refer to \cite[Def.~2.1.3]{Dotto} for the definition of uniserial representations.
 
\begin{lemma}\label{lemma-Rchi-n}
For any $0\leq n\leq p-1$, there exists a unique (up to isomorphism)   smooth $\F$-representation of $\cO_D^{\times}/Z_D^1$, denoted by $E^-(\chi,n)$, which is uniserial of dimension $n+1$, namely $\dim_{\F}E^-(\chi,n)[\fm_D^i]=i$ for any $1\leq i\leq n+1$, and such that the graded pieces are $\chi,\chi\alpha^{-1},\dots,\chi\alpha^{-n}$.
\end{lemma}
\begin{proof}
Clearly we may take $E^-(\chi,0)=\chi$ and $E^-(\chi,1)=E^-(\chi)$. 
By induction, once $E^-(\chi, i)$ is constructed, it suffices to show that $\dim_{\F}\Ext^1(\chi\alpha^{-(i+1)},E^-(\chi,i))=1$.
We have a short exact sequence (for $i\geq 1$)
\[0\ra E^-(\chi,i-1)\ra E^-(\chi,i)\ra\chi\alpha^{-i}\ra0\]
which implies 
\[\Ext^1(\chi',E^-(\chi,i-1))\ra\Ext^1(\chi',E^-(\chi,i))\ra\Ext^1(\chi',\chi\alpha^{-i})\ra\Ext^2(\chi',E^-(\chi,i-1))\]
where $\chi':=\chi\alpha^{-(i+1)}$.
As $2\leq i+1\leq p-1$,   Proposition \ref{prop-Ext1-U1} implies that the first and last term vanish while the third term has dimension $1$, from which the result follows.
\end{proof}

\begin{remark}
In Proposition \ref{prop:finiteness} below, we will see that there don't exist uniserial $(p+1)$-dimensional representations as in Lemma \ref{lemma-Rchi-n}.  
\end{remark}
 
We have the following analogous result of Lemma \ref{lemma-Rchi-n}.

\begin{lemma}\label{lemma-Rchi+n}
For any $0\leq n\leq p-1$, there exists a unique (up to isomorphism)   smooth $\F$-representation of $\cO_D^{\times}/Z_D^1$, denoted by $E^+(\chi,n)$, which is uniserial of dimension $n+1$, namely $\dim_{\F}E^+(\chi,n)[\fm_D^i]=i$ for any $1\leq i\leq n+1$, and such that the graded pieces are $\chi,\chi\alpha,\dots,\chi\alpha^{n}$.
\end{lemma} 
Alternatively,  $E^+(\chi,n)$ can be obtained as  the conjugation of $E^-(\chi^{\sigma},n)$ by $\varpi_D$, where $\chi^{\sigma}$ denotes the character $\chi(\varpi_D^{-1}\cdot \varpi_D)$.
 
\subsection{Induced representations of $\cO_D^{\times}$}

Let $\chi:\cO_D^{\times}\ra \F^{\times}$ be a smooth character.  We view $\chi$ as a character of $U_D^{2}H$ and study the structure of $\Ind_{U_D^2H}^{\cO_D^{\times}}\chi$. Since $U_D^2$ is a normal subgroup of $\cO_D^{\times}$ and $\chi|_{U_D^2}$ is trivial, $U_D^2$ acts trivially on $\Ind_{U_D^2H}^{\cO_D^{\times}}\chi$.  Note that the coset $U_D^{2}H\backslash\cO_D^{\times}$ has a set of representatives
\[\big\{1+\varpi_D[\lambda],\ \ \lambda\in\F_q\big\}.\] 
Hence, by fixing a basis $v$ of $\chi$, we get a basis of $\Ind_{U_D^2H}^{\cO_D^{\times}}\chi$ as follows:
\[\big\{\big[1+\varpi_D[\lambda], v\big],\ \ \lambda\in\F_q\big\}.\]
Here, for $g\in\cO_D^{\times}$, $[g,v]$ denotes the unique function in $\Ind_{U_D^2H}^{\cO_D^{\times}}\chi$ supported on $(U_D^2H)g^{-1}$ and such that $[g,v](g^{-1})=v$. 
For $0\leq k\leq q-1$, define
\begin{equation}\label{eq:f-kv}f_{k,v}=\sum_{\lambda\in\F_q}\lambda^k\big[1+\varpi_D[\lambda],v\big]\in \Ind_{U_D^2H}^{\cO_D^{\times}}\chi.\end{equation}
A standard argument shows that these $f_{k,v}$ also form a basis of $\Ind_{U_D^2H}^{\cO_D^{\times}}\chi$. Moreover, $f_{k,v}$ is an $H$-eigenvector of character $\chi\alpha^{-k}$. Indeed, for $\mu\in\F_q^{\times}$ we have \[[\mu]\cdot(1+\varpi_D[\lambda])=[\mu]+\varpi_D[\mu^p\lambda]=(1+\varpi_D[\mu^{p-1}\lambda])\cdot [\mu]\]
and so 
\[[\mu]\cdot f_{k,v}=\sum_{\lambda\in\F_q}\lambda^k\big[1+\varpi_D[\mu^{p-1}\lambda],\chi(\mu)v\big]=\chi(\mu)\mu^{(1-p)k}f_{k,v}=(\chi\alpha^{-k})([\mu])f_{k,v}.\] 
Remark that $f_{p,v}$ has $H$-character $\chi\alpha$ as $\alpha^{p+1}=1$.

For $\mu,\lambda\in\F_q$, we have  (in $\Z_{q}\hookrightarrow \cO_D$)
\[[\mu]+[\lambda]\equiv [\mu+\lambda]-p [P(\mu,\lambda)]\ \mod p^2\]
where $P(X,Y)$ denotes the Witt polynomial (cf.~\cite[\S II.6]{Se79})
\[P(X,Y):=\sum_{s=1}^{p-1}\frac{\binom{p}{s}}{p}X^{p^{-1}(p-s)}Y^{p^{-1}s}.\]

For convenience, we introduce the following notation:  if $\mu_1,\mu_2,\mu_3\in\F_q$, let 
\begin{equation}
\label{def:A}A(\mu_1,\mu_2,\mu_3):=1+\varpi_D[\mu_1]+p[\mu_2]+p\varpi_D[\mu_3]\in\cO_D^{\times}.\end{equation}

\begin{lemma}\label{lem:Witt}
Let $\mu_1,\mu_2,\mu_3,\lambda\in\F_q$, then 
\[A(\mu_1,\mu_2,\mu_3)\cdot (1+\varpi_D[\lambda])=\big(1+\varpi_D[\mu_1+\lambda]\big)\big(1+p[\mu_2+\mu_1^p\lambda]+p\varpi_D[X]+p^2X'\big)\]
where $X'\in\cO_D$ and
\[X=\mu_3+\mu_2^p\lambda-(\mu_1+\lambda)(\mu_2+\mu_1^p\lambda)-P(\mu_1,\lambda).\]
\end{lemma}
\begin{proof}
It is a direct computation:  
\[\begin{array}{rll}&&A(\mu_1,\mu_2,\mu_3)\cdot (1+\varpi_D[\lambda])\\
&=&1+\varpi_D([\mu_1]+[\lambda])+p([\mu_2]+[\mu_1^p\lambda])+p\varpi_D([\mu_3]+[\mu_2^p\lambda])+p^2[\mu_3^p\lambda]\\
&\equiv& 1+\varpi_D[\mu_1+\lambda] +p[\mu_2+\mu_1^p\lambda]+p\varpi_D[\mu_3+\mu_2^p\lambda-P(\mu_1,\lambda)] \mod p^2\cO_D.\end{array}\] 
The result follows by further factoring out the term $(1+\varpi_D[\mu_1+\lambda])$.
\end{proof}

For $0\leq k\leq q-1$, let $k=k_0+pk_1$ be its $p$-adic expansion. For $0\leq k,k'\leq q-1$, we say $k'\preceq k$ if $k_0'\leq k_0$ and $k_1'\leq k_1$.  Next result describes the  submodule structure of $\Ind_{U_D^2H}^{\cO_D^{\times}}\chi$.

\begin{proposition}\label{prop:Ind-v}
The subrepresentation of $\Ind_{U_D^2H}^{\cO_D^{\times}}\chi$ generated by $f_{k,v}$  is spanned over $\F$ by $\{f_{k',v},  k'\preceq k\}$. In particular, the $\cO_D^{\times}$-socle of $\Ind_{U_D^2H}^{\cO_D^{\times}}\chi$ is one-dimensional, spanned by $f_{0,v}$.
\end{proposition}

\begin{proof}
Since $f_{k,v}$ is an $H$-eigenvector and $U_D^2$ acts trivially on $\Ind_{U_D^2H}^{\cO_D^{\times}}\chi$, it suffices to consider the action of $U_D^1/U_D^2$. Using Lemma \ref{lem:Witt} (with $\mu_2=\mu_3=0$) we have
\[A(\mu_1,0,0)\cdot f_{k,v}=\sum_{\lambda\in\F_q}\lambda^k\big[(1+\varpi_D[\mu_1+\lambda]),v\big]=\sum_{\lambda\in\F_q}(\lambda-\mu_1)^{k}\big[1+\varpi_D[\lambda],v\big].\]
Expanding $(\lambda-\mu_1)^{k}=(\lambda-\mu_1)^{k_0}(\lambda^p-\mu_1^p)^{k_1}$, we get
\begin{equation}\label{eq:sum-A}A(\mu_1,0,0)\cdot f_{k,v}=\sum_{k'\preceq k}\binom{k_0}{k_0'}\binom{k_1}{k_1'}(-\mu_1)^{k-k'}f_{k',v}\end{equation}
which implies that $\langle \cO_D^{\times}\cdot f_{k,v}\rangle$ is contained in 
$\oplus_{k'\preceq k}\F f_{k',v}$. Taking  summation $\sum_{\mu_1\in\F_q}\mu_1^i$ of the above equality, for $i=q-1-(k-k')$,  shows that each $f_{k',v}$ (for   $k'\preceq k$) also belongs to $\langle \cO_D^{\times}\cdot f_{k,v}\rangle$.  

To see the last assertion, let $f\in \big(\Ind_{U_D^2H}^{\cO_D^{\times}}\chi\big)^{U_D^1}$ be non-zero and write $f=\sum_{k=0}^{q-1}a_kf_{k,v}$. Choose an index $k$ such that $k$ is maximal  among the set $\{k': a_{k'}\neq0\}$ with respect to $\preceq$. We need to prove $k=0$. Using \eqref{eq:sum-A} one checks that
\[\sum_{\mu_1\in\F_q}\mu_1^{q-1-k}A(\mu_1,0,0)\cdot f_{k',v}=\left\{\begin{array}{cll}(-1)^{k+1}f_{0,v} &k'=k\\
0&k'\neq k\end{array}\right.\]
which implies  $\sum_{\mu_1\in\F_q}\mu_1^{q-1-k}A(\mu_1,0,0) \cdot f\neq0$. If $k\neq0$, then $\sum_{\mu_1\in\F_q}\mu_1^{q-1-k}=0$ and \[\sum_{\mu_1\in\F_q}\mu_1^{q-1-k}A(\mu_1,0,0)\in \fm_D,\] a contradiction.  
 \end{proof}
\begin{remark}
Proposition \ref{prop:Ind-v} shows that the structure of $\Ind_{U_D^2H}^{\cO_D^{\times}}\chi$ looks like that of $\Ind_{B(\F_q)}^{\GL_2(\F_q)}\eta$ where $\eta$ is a smooth character of the (upper) Borel subgroup $B(\F_q)$, cf.~\cite[\S2]{BP}.
\end{remark}

\begin{corollary}\label{cor:Ind-v-sub}
(i) The subspace of  $\Ind_{U_D^2H}^{\cO_D^{\times}}\chi$ spanned by $\{f_{i,v}, 0\leq i\leq p-1\}$ is $\cO_D^{\times}$-stable, and isomorphic to $E^-(\chi,p-1)$. As a consequence, $U_D^2$ acts trivially on $E^-(\chi,p-1)$.

(ii) The space $\Hom_{\cO_D^{\times}}\big(E^-(\chi,p-1),\Ind_{U_D^2H}^{\cO_D^{\times}}\chi\big)$ is one-dimensional, and any non-zero element in it has image equal to $\oplus_{i=0}^{p-1}\F f_{i,v}$. 
\end{corollary}
\begin{proof}
(i) It is an easy consequence of   Proposition \ref{prop:Ind-v} and the uniqueness of $E^-(\chi,p-1)$  in Lemma \ref{lemma-Rchi-n}.

(ii) Since the $\cO_D^{\times}$-socle of $\Ind_{U_D^{2}H}^{\cO_D^{\times}}\chi$ is isomorphic to $\chi$ and since $\chi$ occurs in $E^-(\chi,p-1)$ with multiplicity one,  the space $\Hom_{\cO_D^{\times}}\big(E^-(\chi,p-1),\Ind_{U_D^2H}^{\cO_D^{\times}}\chi\big)$ has dimension at most $1$. The result then follows from (i). 
\end{proof}
\medskip

\begin{definition}
Let $\kappa_2,\kappa_2',\epsilon_2:U_D^2/Z_D^1\ra\F$ be the functions defined by (where $x\in\cO_D$)
\[\kappa_2(1+p[\lambda]+p\varpi_Dx):=\overline{x},\]
\[\kappa_2'(1+p[\lambda]+p\varpi_Dx):=\overline{x}^p,\]
\[\ \ \ \epsilon_2(1+p x):=\overline{x}^p-\overline{x}.\]
\end{definition}
It is direct to check that $\kappa_2,\kappa_2',\epsilon_2$ are elements in $\Hom(U_D^2/Z_D^1,\F)$. 
\begin{lemma}\label{lem:U2-F}
We have $\Hom(U_D^2/Z_D^1,\F)=\F \kappa_2 \oplus \F\kappa_2'\oplus\F\epsilon_2$.
\end{lemma}
\begin{proof}
Since $U_D^2/Z_D^1$ is a uniform pro-$p$ group of dimension $3$, the space $\Hom(U_D^2/Z_D^1,\F)$ has dimension $3$ over $\F$. It is then direct to check that $\kappa_2,\kappa_2',\epsilon_2$ form a basis. 
\end{proof}

\begin{definition}\label{def:W}
Consider the following two-dimensional representation $W$ of $U_D^2H$:
\begin{enumerate}
\item[$\bullet$] $W=\F v\oplus \F w$, where $v$, $w$ are $H$-eigenvectors of character $\chi$, $\chi\alpha$, respectively; 
\item[$\bullet$] for $g\in U_D^2$, 
\begin{equation}\label{eq:W-action}
gv=v,\ \ gw=w+\kappa_2'(g)v.
\end{equation}
\end{enumerate}
\end{definition}
It is direct to check that $W$ is well-defined. Actually, taking into account the $H$-action,  Lemma \ref{lem:U2-F} implies that $\Ext^1_{U_D^2H/Z_D^1}(\chi\alpha,\chi)$  is one-dimensional with $W$ being a basis. 

Let $V:=\Ind_{U_D^2H}^{\cO_D^{\times}} W$. A basis of $V$ is given by $\{f_k, F_k, 0\leq k\leq q-1\}$, where (cf.~\eqref{eq:f-kv})
\[f_k:=f_{k,v},\ \ F_k:=f_{k,w}. \]

\begin{lemma}\label{lem:V-fix}
We have $V^{U_D^2}=\Ind_{U_D^2H}^{\cO_D^{\times}}(\F v)$.
\end{lemma}

\begin{proof}
We already observed that $\Ind_{U_D^2H}^{\cO_D^{\times}}(\F v)$ is fixed by $U_D^2$ as $\F v$ is. To see the inclusion $V^{U_D^2}\subseteq\Ind_{U_D^2H}^{\cO_D^{\times}}(\F v)$, by Proposition \ref{prop:Ind-v} it suffices to check that $F_0\notin V^{U_D^2}$. Taking $\mu_1=\mu_2=0$ in Lemma \ref{lem:Witt}  we get $X=\mu_3$, so by \eqref{eq:W-action}
\[A(0,0,\mu_3)F_0=\sum_{\lambda\in\F_q}\big[(1+\varpi_D[\lambda]),(w+\mu_3^pv)\big]=F_0+\mu_3^pf_0\]
which implies the result.
\end{proof}

\begin{proposition}\label{prop:p+1}
$V$ does not admit any subrepresentation (resp.~any quotient) of dimension $p+1$ which is uniserial with successive graded pieces $\chi, \chi\alpha^{-1},\dots,  \chi\alpha^2,\chi\alpha$. 
\end{proposition}

\begin{proof}
It is direct to check that $W^{\vee}\cong W\otimes (\chi^{-2}\alpha^{-1})$ and  $V^{\vee}\cong V\otimes (\chi^{-2}\alpha^{-1}) $. Thus it suffices to prove the result for subrepresentations.

Suppose not, and let $V_1\subset V$ be a subrepresentation of dimension $p+1$ as in the proposition. Let $V_2=V_1[\fm_D^{p}]$.  Then $V_2$ is isomorphic to $E^-(\chi,p-1)$  and, by Corollary \ref{cor:Ind-v-sub} and Lemma \ref{lem:V-fix}, is precisely the subrepresentation of $\Ind_{U_D^2H}^{\cO_D^{\times}}(\F v)$ spanned by $\{f_0,\dots,f_{p-1}\}$. Using Proposition \ref{prop:Ind-v}, the $\cO_D^{\times}$-socle of $\Ind_{U_D^2H}^{\cO_D^{\times}}(\F v)/V_2$ is one-dimensional and spanned by $f_{p}$. From the exact sequence 
\[0\ra \Ind_{U_D^2H}^{\cO_D^{\times}}(\F v)/V_2\ra V/V_2\ra  \Ind_{U_D^2H}^{\cO_D^{\times}}(\F w)\ra0,\]
we deduce that  the $\cO_D^{\times}$-socle of $V/V_2$ is contained in $ \F f_{p}\oplus  \F F_0$. Hence, there exist $a,b\in \F$  such that (as vector spaces)  
\[V_1=V_2\oplus \F(af_{p}+b F_0). \]
Moreover, we must have $b\neq0$, otherwise $V_1=\oplus_{i=0}^{p}\F f_i$ and  would \emph{not} be uniserial by Proposition \ref{prop:Ind-v} again. 

Now, if $\mu_2\in\F_q$, we have $A(0,\mu_2,0)\cdot f_p=f_p$ and by Lemma \ref{lem:Witt}   
\[A(0,\mu_2,0)\cdot F_0=\sum_{\lambda\in\F_q}\big[(1+\varpi_D[\lambda]),w+(\mu_2^p\lambda-\mu_2\lambda)^pv\big]=F_0+(\mu_2-\mu_2^p)f_p,\]
so that 
\[A(0,\mu_2,0)\cdot(af_p+bF_0)-(af_p+bF_0)=b(\mu_2-\mu_2^p)f_p.\]
By taking $\mu_2\in\F_q\setminus \F_p$, we deduce that $f_p$, hence also  $F_0$,  belongs to $V_1$. But this would imply $\dim_{\F} V_1\geq p+2$,   a contradiction.
\end{proof}

\begin{remark}
Using Lemma \ref{lem:Witt}, one can check that  the subrepresentation $\langle \cO_D^{\times}.F_0\rangle$ of $V$ is equal to $(\oplus_{i=0}^{p-2}\F f_{i,v})\oplus \F f_{p,v}\oplus \F (f_{p-1,v}+f_{2p,v})\oplus \F F_0$.
\end{remark}

\subsection{Uniserial representations}\label{sec:uniserial}

 We define two elements in the group algebra $\F[U^1_D]$ which will play an important role in the rest.
\begin{definition}\label{defn-X+X-}
Define
\[Y:=\sum_{\lambda\in\F_{q}^{\times}} \lambda^{-1}(1+\varpi_D[\lambda])\in \F[U_D^1]\] 
\[ 
Z:=\sum_{\lambda\in\F_q^{\times}}\lambda^{-p}(1+\varpi_D[\lambda])\in\F[U_D^1]\]
where the terms $1+\varpi_D[\lambda]$ are viewed as group elements in $U_D^1$.
\end{definition}
 
 We also view $Y, Z$ as elements in the Iwasawa algebra $\Lambda:=\F[\![U_D^1/Z_D^1]\!]$ (after taking images). 
Then the maximal ideal $\fm_D$ is generated by $Y$ and $Z$ by \cite[Lem.~2.7]{HW-JL1}. Let \[\gr_{\fm_D}(\Lambda):=\bigoplus_{n\geq 0}\fm_D^n/\fm_D^{n+1}\] be the graded ring of $\Lambda$ with respect to the $\fm_D$-adic filtration.  Since we will only consider the $\fm_D$-adic filtration on $\Lambda$,  the subscript $\fm_D$ in $\gr_{\fm_D}(\Lambda)$ will be often omitted.

Comparing with the case of $\GL_2(\Z_p)$, we should consider $Y$ (resp.~$Z$) as a remplacement of $\sum_{\lambda\in\F_p}\lambda^{-1}\smatr{1}{[\lambda]}01$ (resp.~$\sum_{\lambda\in\F_p}\lambda^{-1}\smatr10{[\lambda]}1$). The reason is as follows.

\begin{lemma}\label{lemma-character-X+X-}
Let $\pi$ be a smooth representation of $\cO_D^{\times}/Z_D^1$. If $v\in\pi$ is an eigenvector of $H$ of character $\chi$, then  $Yv$ (resp.~$Zv$) is  an eigenvector of $H$ of character $\chi\alpha$ (resp.~$\chi\alpha^{-1}$).
\end{lemma}
\begin{proof}
This is a direct check.
\end{proof}

Let $y$ (resp.~$z$) denote  the image of $Y$ (resp.~$Z$) in $\gr(\Lambda)$.
It is proved in \cite[\S2]{HW-JL1}  that  $\gr(\Lambda)$  is isomorphic to the universal enveloping algebra of the  Lie algebra $\mathfrak{g}$, where $\mathfrak{g}=\F y\oplus \F z\oplus \F h$, with the relations 
\begin{equation}\label{eq:g-Lie}[y,z]=h,\ \ [y,h]=[z,h]=0\end{equation}
and $\deg y=\deg z=1$. As a consequence, there is an inclusion $\F[y]\subseteq \gr(\Lambda)$.\medskip

\begin{lemma}\label{lem:subring}
There is a natural inclusion $\F[\![Y]\!]\subseteq \Lambda$.
\end{lemma} 

\begin{proof}
The  natural ring morphism $\F[Y]\ra \Lambda$ is injective because it is injective when passing to graded rings, as is recalled above.  It is clear that $\fm_D^n\cap \F[Y]=(Y^n)$ for any $n\geq 1$. Since $\Lambda$ is $\fm_D$-adically complete, we deduce the result by taking completion. \end{proof}

If $\pi$ is a smooth representation of $\cO_D^{\times}/Z_D^1$, write \[\pi[Y]:=\pi[Y=0]=\big\{v\in\pi: Yv=0\big\}.\]
The following result is an analog of \cite[Prop.~5.9]{Pa10}. 
\begin{lemma}\label{lem-pi-injective}
Let $\pi$ be a smooth representation of $\cO_D^{\times}/Z_D^1$ of dimension $d\geq 2$ (possibly infinite). Assume  that 
\[ \dim_{\F} \pi[Y]=1\]
and $H$ acts on it by $\chi$. Then for any $1\leq n\leq d$, $\dim_{\F}\pi[\fm_D^n]=n$  and there is an exact sequence
\begin{equation}\label{equation-pi-filtration}0\ra \pi[\fm_D^n]\ra \pi[\fm_D^{n+1}]\ra \chi\alpha^{-n}\ra0.\end{equation}
\end{lemma}

\begin{proof}
First note that since $\pi[\fm_D]=\pi^{U_D^1}$ is always non-zero, and $\pi[\fm_D]\subset \pi[Y]$, the assumption $\dim_{\F}\pi[Y]=1$ implies that $\pi[\fm_D]=\pi[Y]$; in particular, $\pi[Y]$ is stable under $\cO_{D}^{\times}$.
We claim that $(\pi/\pi[Y])[Y]$ is still one-dimensional. This will imply $\pi[\fm_D^2]=\pi[Y^2]$ and we may conclude inductively using Lemma \ref{lemma-character-X+X-}.

The Pontryagin dual  $\pi^{\vee}$  is a finitely generated $\Lambda$-module.  The assumption $\dim_{\F}\pi[Y]=1$ implies that $\pi^{\vee}$ is  cyclic when viewed as an $\F[\![Y]\!]$-module via Lemma \ref{lem:subring}. The claim is equivalent to saying that the submodule $(Y)\pi^{\vee}$ is again a cyclic $\F[\![Y]\!]$-module which is obvious.  
 \end{proof}

\begin{corollary}\label{coro-pi-injective}
Let $\pi$ be a smooth representation of $\cO_D^{\times}/Z_D^1$ such that $\dim_{\F}\pi[Y]=1$. Then $\dim_{\F}\pi'[Y]=1$  for any subquotient $\pi'$ of $\pi$.
\end{corollary}
\begin{proof}
It directly follows from the proof of Lemma \ref{lem-pi-injective}.
\end{proof}

The following result shows a key difference between the theory of representations of $\cO_D^{\times}$ (or $U_D^1$) and of $I_1$ (pro-$p$ Iwahori subgroup of $\GL_2(\Z_p)$),  cf.~\cite[Prop.~7.2]{Pa10}.
\begin{proposition}\label{prop:finiteness}
Let $\pi$ be a smooth representation of $\cO_D^{\times}/Z_D^1$ such that
$\dim_{\F}\pi[Y]=1$.
Then $\pi$ has dimension $\leq p$. Moreover, $\pi$ is isomorphic to $E^-(\chi, n)$ with $\chi=\pi[Y]$ and $n=\dim_{\F}\pi-1$.  \end{proposition}

\begin{proof}
Assume $\dim_{\F}\pi\geq p+1$ for a contradiction.  By Lemma \ref{lem-pi-injective}, $\pi$ contains a unique subrepresentation $\pi'$ which is of dimension $p+1$. Thus it suffices to show that there is no representation $\pi$ of dimension $p+1$ and satisfying $\dim_{\F}\pi[Y]=1$. Suppose $\dim_{\F}\pi=p+1$ (for a contradiction). The $\cO_D^{\times}$-socle of $\pi$ is identified with $\pi[Y]$, which is one-dimensional isomorphic to $\chi$, say. Then the graded pieces of the socle filtration of $\pi$ are $\chi,\chi\alpha^{-1},\dots,\chi\alpha^{-p}=\chi\alpha$. In particular, $\mathrm{cosoc}(\pi)\cong\chi\alpha$.

Let $\pi_1\subset \pi$ be the unique subrepresentation of dimension $p$ (cf.~Lemma \ref{lem-pi-injective}). Corollary \ref{cor:Ind-v-sub} implies that $\pi_1$ is isomorphic to $E^-(\chi,p-1)$ and is fixed by $U_D^2$.   Choose $w\in \pi\setminus\pi_1$ on which  $H$ acts via $\chi\alpha$; this is always possible as $\mathrm{cosoc}(\pi)\cong \chi\alpha$. We claim that $w$ is not fixed by $U_D^2$. Otherwise,  by Frobenius reciprocity the embedding $\F w\cong \chi\alpha\hookrightarrow \pi|_{U_D^2H}$ would induce an $\cO_D^{\times}$-equivariant morphism 
\[\Ind_{U_D^2H}^{\cO_D^{\times}}\chi\alpha\ra \pi\]
which must be surjective because $\pi$ is generated by $w$ as a representation of $\cO_D^{\times}$. But it follows from Proposition \ref{prop:p+1} that $\Ind_{U_D^2H}^{\cO_D^{\times}}\chi$ does not admit $\pi$ as a quotient. 

On the other hand, the quotient $\pi/\pi[Y]$ is uniserial of dimension $p$,  on which $U_D^2$ acts trivially by Corollary \ref{cor:Ind-v-sub}. We deduce that \[gw-w\in \pi[Y],\ \ \forall g\in U_D^2\]
equivalently, \[W:=(\pi[Y])\oplus \F w\subset \pi\] is a subspace stable under $U_D^2H$. Comparing the characters of $H$, we see that the structure of $W$ is exactly as in Definition \ref{def:W} (non-split because $w$ is not fixed by $U_D^2$ as seen above).  By Frobenius reciprocity, this induces a surjection $\Ind_{U_D^2H}^{\cO_D^{\times}}W\twoheadrightarrow \pi$, which is impossible by Proposition \ref{prop:p+1} again.
\end{proof}

\subsection{Annihilators}

Recall that   the graded ring $\gr(\Lambda)$ is isomorphic to the universal enveloping algebra (over $\F$) of the Lie algebra $\mathfrak{g}=\F y\oplus \F z\oplus \F h$, see \S\ref{sec:uniserial}.

\begin{proposition}\label{prop:ann}
If $\mathfrak{a}$ denotes the left ideal of $\gr(\Lambda)$ generated by $y^2$ and $z$, then $\gr(\Lambda)/\mathfrak{a}$ is $3$-dimensional over $\F$, spanned by $1,y,zy$. Similarly, the quotient $\gr(\Lambda)/(y,z^2)$ is  $3$-dimensional.
\end{proposition}

\begin{proof}
It is enough to show that $z^2y, yzy\in\mathfrak{a}$. First, since $[y,h]=0$ by \eqref{eq:g-Lie},   we  get  \[2yzy=y^2z+zy^2\in \mathfrak{a},\] so $yzy\in\mathfrak{a}$ (as $p\neq 2$). Similarly, the equality $[z,h]=0$ gives
$z^2y=(2zy-yz)z\in\mathfrak{a}$.
\end{proof}

The group $H$ acts on $\gr(\Lambda)$ in the following way. For $g\in H$,
\begin{equation}\label{eq:H-action}gy=\alpha(g)y,\ \ gz=\alpha^{-1}(g)z,\ \ g h=h.\end{equation}
We say that $N$ is a $\gr(\Lambda)$-module with \emph{compatible $H$-action} if $H$ acts on $N$ such that \[g(rv)=(gr)(gv),\ \ \   g\in H, r\in\gr(\Lambda), v\in N.\] Typical examples are $\gr_{\fm_D}(\pi^{\vee})$ for  smooth representations $\pi$ of $\cO_D^{\times}/Z_D^1$.

Using \eqref{eq:H-action}, it is easy to see that as $H$-modules
\begin{equation}\label{eq:gr2}
\gr^1_{\fm_D}(\Lambda)=\F y\oplus \F z \cong \alpha\oplus \alpha^{-1}, \ \ \gr^2_{\fm_D}(\Lambda)=\F y^2\oplus \F yz\oplus \F zy \oplus \F z^2\cong \alpha^2\oplus \ide \oplus \ide\oplus \alpha^{-2}.
\end{equation} 
Note that the assumption $p\geq 5$ ensures that $\alpha^{2}\neq \alpha^{-2}.$ We deduce the following easy result, which will be repeatedly used in next subsections.  
\begin{lemma}\label{lem:annihilator}
Let $\pi$ be a smooth representation of $\cO_D^{\times}/Z_D^{1}$. Assume that $\pi[\fm_D]$ is $1$-dimensional isomorphic to $\chi$. Then $\gr_{\fm_D}(\pi^{\vee})$ is a cyclic $\gr(\Lambda)$-module. Moreover,
\begin{itemize}
\item[(i)] if $\chi\alpha$ (resp.~$\chi\alpha^{-1}$) does not occur in $\pi[\fm_D^2]/\pi[\fm_D]$, then $z$ (resp.~$y$) annihilates $\gr_{\fm_D}(\pi^{\vee})$;
\item[(ii)] if $\chi\alpha^2$ (resp.~$\chi\alpha^{-2}$) does not occur in $\pi[\fm_D^3]/\pi[\fm_D^2]$, then $z^2$ (resp.~$y^2$) annihilates $\gr_{\fm_D}(\pi^{\vee})$;
\item[(iii)] if $\chi$ does not occur in $\pi[\fm_D^3]/\pi[\fm_D^2]$, then $(yz,zy)$ annihilates $\gr_{\fm_D}(\pi^{\vee})$. 
\end{itemize}

\end{lemma} 

Next we give a criterion for $\pi$ to satisfy $\dim_{\F}\pi[Y]=1$. 

\begin{lemma}\label{lem:criterion-piY}
Let $\pi$ be a smooth representation of $\cO_D^{\times}/Z_D^{1}$. Assume that $\pi[\fm_D^3]$ is $3$-dimensional and isomorphic to $E^-(\chi,2)$ for some character $\chi$. Then $\dim_{\F}\pi[Y]=1$. 
\end{lemma}

\begin{proof}
The assumption implies that $\pi[\fm_D]\cong \chi$, $\pi[\fm_D^2]/\pi[\fm_D]\cong \chi\alpha^{-1}$ and $\pi[\fm_D^3]/\pi[\fm_D^2]\cong\chi\alpha^{-2}$. In particular, $\gr_{\fm_D}(\pi^{\vee})$ is a cyclic $\gr(\Lambda)$-module. Moreover, it is annihilated by $(z,zy)$ by Lemma \ref{lem:annihilator}, hence is a quotient of $\F[y]$. This implies the result.
\end{proof}

\subsection{Extensions with $E^-(\chi)$}

\begin{proposition}\label{prop-Ext1-psi-R-}
Let $\psi:\cO_D^{\times}\ra\F^{\times}$ be a smooth character. Then $\Ext^1(\psi,E^-(\chi))\neq 0$ if and only if  $\psi\in\{\chi,\chi\alpha,\chi\alpha^{-2}\}$. Moreover, if this is the case, we have $\dim_{\F}\Ext^1(\psi,E^-(\chi))=1$.
\end{proposition}
\begin{proof}
Since $p>2$, the three characters  $\chi,\chi\alpha,\chi\alpha^{-2}$ are all distinct. 
Using Proposition \ref{prop-Ext1-U1} and the long exact sequence
\begin{equation}\label{equation-Ext1-U1} \Hom(\psi,\chi\alpha^{-1})\hookrightarrow\Ext^1(\psi,\chi)\ra\Ext^1(\psi,E^-(\chi))\ra\Ext^1(\psi,\chi\alpha^{-1})\ra \Ext^2(\psi,\chi)\end{equation}
we easily deduce that  $\Ext^1(\psi,E^-(\chi))=0$ if $\psi\notin\{\chi,\chi\alpha,\chi\alpha^{-2}\}$.

 If $\psi=\chi\alpha$, then $\Hom(\psi,\chi\alpha^{-1})=\Ext^1(\psi,\chi\alpha^{-1})=0$ by  Proposition \ref{prop-Ext1-U1}, and the result easily follows.
The case $\psi=\chi$ is treated similarly, using $\Ext^2(\chi,\chi)=0$ by Proposition \ref{prop-Ext1-U1}.
 
If $\psi=\chi\alpha^{-2}$, the result follows from Lemma \ref{lemma-Rchi-n}.
\end{proof}
 
\begin{proposition}\label{prop:Ext1-R+R-}
Let $\psi,\chi:\cO_D^{\times}\ra \F^{\times}$ be smooth characters. 

\begin{enumerate}
\item[(i)] If $\psi\in\{\chi,\chi\alpha^{-2}\}$, then $\Ext^1(E^+(\psi),E^-(\chi))=0$. 
\item[(ii)] Assume that the natural morphism  \[\Ext^1(E^+(\psi),E^-(\chi)\big)\ra \Ext^1\big(\psi,E^-(\chi))\] induced  by $\psi\hookrightarrow E^+(\psi)$ is non-zero. Then $\psi=\chi\alpha$ and \[\dim_{\F}\Ext^1(E^+(\chi\alpha),E^-(\chi))=1.\]  
\end{enumerate}
\end{proposition}

\begin{proof}
 If $\Ext^1(E^+(\psi),E^-(\chi))\neq0$  then  $\psi\in \{\chi,\chi\alpha,\chi\alpha^{-2}\}$ or $\psi\alpha\in\{\chi,\chi\alpha,\chi\alpha^{-2}\}$ by Proposition \ref{prop-Ext1-psi-R-}, equivalently
\[\psi\in\{\chi,\chi\alpha,\chi\alpha^{-1},\chi\alpha^{-2},\chi\alpha^{-3}\}.\]

(i) Assume first $\psi=\chi$. The short exact sequence $0\ra \chi\ra E^+(\chi)\ra \chi\alpha\ra0$ induces a long exact sequence
\[0\ra \Hom(\chi,E^-(\chi))\overset{\partial}{\ra} \Ext^1(\chi\alpha,E^-(\chi))\ra \Ext^1(E^+(\chi),E^-(\chi))\overset{\beta}{\ra}\Ext^1(\chi,E^-(\chi)).\]
For the reason of dimensions, Proposition \ref{prop-Ext1-psi-R-} implies that $\partial$ is an isomorphism, so it suffices to prove $\beta=0$. 
Let $c$ be  an  extension class in $\Ext^1(E^+(\chi),E^-(\chi))$ such that $\beta(c)$ is non-zero. Then $c[\fm_D]$ is one-dimensional (isomorphic to $\chi$), and so $\gr_{\fm_D}(c^{\vee})$ is a cyclic $\gr(\Lambda)$-module with compatible $H$-action.  Moreover, it is easy to see that $c$ is uniserial of the form (with $\chi$ being the socle)
\[\chi\ \ligne\ \chi\alpha^{-1}\ \ligne \ \chi\ \ligne\ \chi\alpha \] and so  $\gr_{\fm_D}(c^{\vee})$ is annihilated by the left ideal $(z,y^2)$ by Lemma \ref{lem:annihilator}. But then    Proposition \ref{prop:ann} implies that $c$ has dimension $\leq 3$, a contradiction.

Now assume $\psi=\chi\alpha^{-2}$, so that $\psi\alpha=\chi\alpha^{-1}$. Noting that $\Ext^1(\chi\alpha^{-1},E^-(\chi))=0$ by Proposition \ref{prop-Ext1-psi-R-}, we get an embedding
\[ \Ext^1(E^+(\psi),E^-(\chi))\hookrightarrow \Ext^1(\psi,E^-(\chi)).\]
Let $c$ be  a non-zero  extension class in $\Ext^1(E^+(\psi),E^-(\chi))$.   Then $c$ is uniserial of the form 
\[\chi\ \ligne\ \chi\alpha^{-1}\ \ligne \ \chi\alpha^{-2}\ \ligne\ \chi\alpha^{-1}. \]
So its dual has the form
$\chi^{\vee}\alpha\ \ligne\ \chi^{\vee}\alpha^{2}\ \ligne \ \chi^{\vee}\alpha\ \ligne\ \chi^{\vee}$, 
and  we  conclude as in the above case.

(ii) By Proposition \ref{prop-Ext1-psi-R-}, the assumption implies that $\psi\in\{\chi,\chi\alpha,\chi\alpha^{-2}\}$, so we must have $\psi=\chi\alpha$ by (i). The last assertion is an easy exercise. 
\end{proof}
\begin{remark}
With the notation in Proposition \ref{prop:Ext1-R+R-}, if $\psi=\chi\alpha^{-1}$ or $\psi=\chi\alpha^{-3}$ then one easily shows that $\Ext^1(E^+(\psi),E^-(\chi))$ is $1$-dimensional. 
\end{remark}
   
\subsection{Extensions with $E^-(\chi,n)$}

 In this subsection, we generalize Proposition \ref{prop-Ext1-psi-R-} and Proposition \ref{prop:Ext1-R+R-} to the case $n\geq 2$. 
  
\begin{proposition}\label{prop-H1=dim3}
Let $2\leq n\leq p-1$ and $\psi,\chi:\cO_D^{\times}\ra\F^{\times}$ be smooth characters. Then $\Ext^1(\psi,E^-(\chi,n))\neq0$ if and only if  
 \[\left\{\begin{array}{ll}\psi\in\{\chi,\chi\alpha,\chi\alpha^{-n-1}\}&\mathrm{if}\  n<p-1\\
 \psi\in\{\chi,\chi\alpha\}&\mathrm{if}\ n=p-1.
\end{array}\right.\]
Moreover, when $\Ext^1(\psi,E^-(\chi,n))\neq0$, the following statements hold. 

\begin{enumerate}
\item[(i)] The natural morphism  (induced by $E^-(\chi)\hookrightarrow E^-(\chi,n)$)
\begin{equation}\label{eq:isom1=n}\Ext^1\big(\psi,E^-(\chi)\big)\ra \Ext^1\big(\psi,E^-(\chi,n)\big)\end{equation}
is an isomorphism unless  $\psi=\chi\alpha^{-n-1}$ (and $n<p-1$).
\item[(ii)] In each case, $\dim_{\F}\Ext^1(\psi,E^-(\chi,n))=1$.
\end{enumerate}
\end{proposition}

\begin{proof} 
To simplify the notation, write $\pi=E^-(\chi,n)$. Let $\pi'$ be a non-zero extension class in $\Ext^1(\psi,\pi)$, i.e.
\[0\ra \pi\ra \pi'\ra \psi\ra0.\]
 Then $\pi'[\fm_D]=\pi[\fm_D]\cong \chi$. As a consequence, $\gr_{\fm_D}(\pi'^{\vee})$ is a cyclic $\gr(\Lambda)$-module. 

Since $\pi[\fm_D^2]/\pi[\fm_D]$ is one-dimensional, we see that $\pi'[\fm_D^2]/\pi'[\fm_D]$ has dimension  $\leq 2$. 
\begin{enumerate}
\item[(a)] If $\pi'[\fm_D^2]/\pi'[\fm_D]$ has dimension $2$ then 
\[\pi'[\fm_D^2]/\pi'[\fm_D]\cong (\pi[\fm_D^2]/\pi[\fm_D])\oplus \psi=\chi\alpha^{-1}\oplus\psi,\]
and $\pi'$ lies in the image of the natural morphism
\begin{equation}\label{eq:image1}\Ext^1(\psi,\chi)\ra \Ext^1(\psi,\pi).\end{equation}
By Proposition \ref{prop-Ext1-psi-R-} we have $\psi\in\{\chi\alpha,\chi\alpha^{-1}\}$. But  $\pi'[\fm_D^2]/\pi'[\fm_D]$ is multiplicity free  by \eqref{eq:gr2}, so $\psi\neq \chi\alpha^{-1}$. Therefore we have $\psi=\chi\alpha$.   

\item[(b)] If $\pi'[\fm_D^2]/\pi'[\fm_D]$ is one-dimensional,  then it is equal to $\pi[\fm_D^2]/\pi[\fm_D]$ and isomorphic to $\chi\alpha^{-1}$. Similarly, $\pi'[\fm_D^3]/\pi'[\fm_D^2]$ has dimension $\leq 2$, and if the equality holds, then 
\[\pi'[\fm_D^3]/\pi'[\fm_D^2]\cong (\pi[\fm_D^3]/\pi[\fm_D^2])\oplus \psi= \chi\alpha^{-2}\oplus\psi\]
and $\pi'$ lies in the image of the natural morphism
\[\Ext^1(\psi,E^-(\chi))\ra \Ext^1(\psi,\pi)\]
but not in the image of \eqref{eq:image1}. Combined with Proposition \ref{prop-Ext1-psi-R-} (and its proof), this forces $\psi\in\{\chi\alpha^{-2},\chi\}$. But $\chi\alpha^{-2}$ occurs in $\pi'[\fm_D^3]/\pi'[\fm_D^2]$ with multiplicity one by \eqref{eq:gr2}, so  we must have  $\psi\cong\chi$. 

Now, assume $\pi'[\fm_D^3]/\pi'[\fm_D^2]$ has dimension $1$, equivalently,   $\pi'[\fm_D^3]=\pi[\fm_D^3]$ which is isomorphic to $E^-(\chi,2)$.  By Lemma \ref{lem:criterion-piY}, we deduce that   $\dim_{\F}\pi'[Y]=1$, so by Proposition \ref{prop:finiteness} $\pi'$ is isomorphic to $E^-(\chi,n+1)$ and $\psi=\chi\alpha^{-(n+1)}$. If $n=p-1$, this can not happen by Proposition \ref{prop:finiteness} again.
\end{enumerate}
Thus, the first statement of the proposition is proved. 
Moreover, if $\psi\in\{\chi\alpha,\chi\}$,  it follows from the above proof that \eqref{eq:isom1=n} is  an isomorphism, proving (i) and also (ii) in these cases.

We are left to prove (ii) for $\psi=\chi\alpha^{-n-1}$ and $n<p-1$, which follows from the proof of   Lemma \ref{lemma-Rchi-n}. 
\end{proof}

\begin{proposition}\label{prop-Ext1-psi-n}
Let $2\leq n\leq p-1$ and let $\chi,\psi:\cO_D^{\times}\ra\F^{\times}$ be smooth characters such that $\psi\neq \chi\alpha^{-(n+2)}$ when $n=p-1$. Assume that the natural morphism
\[\beta:\Ext^1(E^+(\psi),E^-(\chi,n))\ra \Ext^1(\psi,E^-(\chi,n))\]
is non-zero. 
 Then $\psi=\chi\alpha$.
\end{proposition}
\begin{proof}
We  have a short exact sequence
\begin{equation}\label{eq:seq-Rn}
0\ra E^-(\chi)\ra E^-(\chi,n)\ra  Q\ra0 \end{equation}
where $Q:=E^-(\chi\alpha^{-2},n-2)$.

The assumption implies in particular that 
$\Ext^1(\psi,E^-(\chi,n))\neq0$, so \[\psi\in \{\chi,\chi\alpha,\chi\alpha^{-n-1}\}\] by Proposition \ref{prop-H1=dim3}. We need to prove that the case $\psi=\chi$ or $\psi=\chi\alpha^{-n-1}$ can not happen.

\begin{enumerate}
\item[(a)] Assume  $\psi=\chi$. Using \eqref{eq:seq-Rn} we get an exact sequence 
\[\Ext^1(E^+(\chi),E^-(\chi))\ra\Ext^1(E^+(\chi),E^-(\chi,n))\ra\Ext^1(E^+(\chi),Q).\]
Proposition \ref{prop:Ext1-R+R-} implies that $\Ext^1(E^+(\chi),E^-(\chi))=0$. To prove the result it suffices to prove that $\Ext^1(E^+(\chi),Q)=0$. If $\dim_{\F} Q=1$, then $Q\cong \chi\alpha^{-2}$ and the claim is obvious because $\chi,\chi\alpha\notin\{\chi\alpha^{-1},\chi\alpha^{-3}\}$ (as $p\geq 5$), see Proposition \ref{prop-Ext1-U1}. If $\dim_{\F}  Q\geq 2$, then by Proposition \ref{prop-Ext1-psi-R-} and Proposition \ref{prop-H1=dim3}
\[\chi\in \{\chi\alpha^{-2},\chi\alpha^{-1},\chi\alpha^{-n-1}\} \ \mathrm{or}\ \chi\alpha\in\{\chi\alpha^{-2},\chi\alpha^{-1},\chi\alpha^{-n-1}\}.\]
This is only possible when  $\chi\alpha=\chi\alpha^{-n-1}$ and $n=p-1$, but we have excluded this case.

\item[(b)] Assume $\psi=\chi\alpha^{-n-1}$.  Let $c\in \Ext^1\big(E^+(\psi),E^-(\chi,n)\big)$ be a non-zero class with $\beta(c)\neq0$. Then by Proposition \ref{prop-H1=dim3} and its proof, $\beta(c)$ is a uniserial representation isomorphic to $E^-(\chi,n+1)$. In particular, $c[\fm_D^3]=E^-(\chi,n+1)[\fm_D^3]$ has dimension $3$. Using Lemma \ref{lem:criterion-piY} and  Proposition \ref{prop:finiteness}, we deduce that $\dim_{\F}c[Y]=1$ and $c\cong E^-(\chi,n+2)$. But this implies that $\psi=\chi\alpha^{-n-1}$ \emph{and} $\psi\alpha=\chi\alpha^{-n-2}$, which is impossible.   \qedhere
\end{enumerate}
\end{proof}

We will need the following auxiliary lemma. Recall that, for a character $\chi:\cO_D^{\times}\ra \F^{\times}$, $\chi^{\sigma}$ denotes the character $\chi(\varpi_D^{-1}\cdot\varpi_D)$.

\begin{lemma}\label{lem:split1}
Let $\chi:\cO_D^{\times}\ra \F^{\times}$ be a smooth character  and $1\leq n\leq p-2$. Let $V$ be a smooth representation of $\cO_D^{\times}/Z_D^1$ fitting in an extension
\[0\ra \chi\oplus E^+(\chi^{\sigma},n)\ra V\ra \chi^{\sigma}\alpha^{n+1}\ra0.\]
Assume that $V[\fm^{n+1}_{D}]=\chi\oplus E^+(\chi^{\sigma},n)$ and $\chi^{\sigma}\alpha^{n+1}\neq \chi\alpha^{-1}$.
Then $V\cong\chi\oplus E^+(\chi^{\sigma},n+1)$. 
\end{lemma}

\begin{proof}
First note that the assumption on $V[\fm_D^{n+1}]$ implies that $V$ is non-split and, moreover, that the induced extension   $0\ra E^+(\chi^{\sigma},n)\ra V/\chi\ra\chi^{\sigma}\alpha^{n+1} \ra0$ is non-split (this needs  $n\geq 1$). As a consequence, $V/\chi$ is isomorphic to $E^+(\chi^{\sigma},n+1)$ by Lemma \ref{lemma-Rchi+n}, namely
\begin{equation}\label{eq:V/chi}
0\ra \chi\ra V\ra E^+(\chi^{\sigma},n+1)\ra0.\end{equation}

On the other hand, the quotient  $V/E^+(\chi^{\sigma},n)$ gives an extension class $c\in\Ext^1(\chi^{\sigma}\alpha^{n+1},\chi)$. 
\begin{itemize}
\item[$\bullet$] If $c=0$, then it is easy to see that \eqref{eq:V/chi}  splits, and  the result follows. 
\item[$\bullet$] If $c\neq0$, then \[\chi^{\sigma}\alpha^{n+1}\in\{\chi\alpha, \chi\alpha^{-1}\}\] by Proposition \ref{prop-Ext1-U1}. Since $\chi^{\sigma}\alpha^{n+1}\neq \chi\alpha^{-1}$ by assumption, we must have $\chi^{\sigma}\alpha^{n+1}=\chi\alpha$, equivalently $\chi^{\sigma}\alpha^{n}=\chi$. Then the cokernel of $E^+(\chi^{\sigma},n-1)\hookrightarrow E^+(\chi^{\sigma},n)\hookrightarrow V$ satisfies the following exact sequence
\[0\ra \chi\oplus \chi\ra V/E^+(\chi^{\sigma},n-1)\ra \chi\alpha\ra0.\]
Using the fact that $\Ext^1(\chi\alpha,\chi)$ is $1$-dimensional, we see that $V/E^+(\chi^{\sigma},n-1)$  admits a one-dimensional quotient isomorphic to $\chi$. Hence, $V$ also admits a one-dimensional quotient isomorphic to $\chi$. 
Applying $\Hom(-,\chi)$ to \eqref{eq:V/chi} and noting that $\chi^{\sigma}\alpha^{n+1}=\chi\alpha\neq \chi$, we obtain an injection 
\[0\ra \Hom(V,\chi)\ra \Hom(\chi,\chi)\]
which is actually an isomorphism for the reason of dimensions, hence the sequence \eqref{eq:V/chi} splits.
\end{itemize}
This finishes the proof.
\end{proof}

\subsection{A finiteness result}
Recall that $H$ acts on  $\gr(\Lambda)$ via \eqref{eq:H-action}. For a character $\chi$ of $H$, we write $\chi\otimes \gr(\Lambda)$ for the module $\gr(\Lambda)$ twisted by $\chi$.

\begin{theorem}\label{thm:nonexist1}
Let $\chi:\cO_D^{\times}\ra \F^{\times}$ be a smooth character. Assume that $\chi\neq \chi^{\sigma}$ and $\chi\neq \chi^{\sigma}\alpha^p$. Then there does not exist any smooth $D^{\times}/Z_D^1$-representation $\pi$ such that 
\begin{equation}\label{eq:nonexist1}\gr_{\fm_D}(\pi^{\vee})\cong \big(\chi^{\vee}\otimes \F[y]/(y^{p+2})\big)\oplus \big((\chi^{\sigma})^{\vee}\otimes \F[z]/(z^{p+2})\big).\end{equation}
\end{theorem}
\begin{remark}\label{rem:nongeneric}
We can make the condition on $\chi$ in Theorem \ref{thm:nonexist1} more concrete:   $\chi\neq \chi^{\sigma}$ if and only if $\chi\neq \xi^{s(p+1)}$ for any $s\geq 0$; $\chi\neq \chi^{\sigma}\alpha^p$ if and only if $\chi\neq \xi^{1+s(p+1)}$ for any $s\geq 0$. Here $\xi$ denotes the character $\cO_D^{\times}\twoheadrightarrow \F_q^{\times}\hookrightarrow \F^{\times}$ via  the fixed embedding $\F_q\hookrightarrow \F$. 
\end{remark}

\begin{proof}
Suppose that such a representation as in the statement exists. We prove inductively   the following

\textbf{Claim}. \emph{For any $0\leq n\leq p$, there exists an $\cO_D^{\times}/Z_D^1$-subrepresentation $V_n$ of $\pi|_{\cO_D^{\times}}$ of dimension $n+1$, such that $V_n[Y]$ is one-dimensional and isomorphic to $\chi$.}
\medskip
 
This will finish the proof because $V_{p}$ does not exist  by Proposition \ref{prop:finiteness}. 
\medskip

The claim is clear if $n=0$. For $n=1$, it follows from \eqref{eq:nonexist1} and the following general fact: if $M$ is a  finitely generated $\Lambda$-module with compatible $H$-action annihilated by $\fm_D^2$, then $\gr_{\fm_D}(M)$ is naturally isomorphic to $M$, 
because this is the case for $\Lambda/\fm_D^2$.

 Assume that $V_n$ has been constructed for some  $1\leq n<p$. Then \[V_n\cong E^-(\chi,n)\] by Proposition \ref{prop:finiteness}.  The inclusion $V_n\hookrightarrow \pi$ induces an inclusion $\varpi_DV_n\hookrightarrow \pi$. Since $V_n$ and $\varpi_DV_n$ have  non-isomorphic $\cO_D^{\times}$-socles (as $\chi\neq\chi^{\sigma}$), we obtain an inclusion $V_n\oplus \varpi_DV_n\hookrightarrow \pi$, whose image is identified  with $\pi[\fm_D^{n+1}]$ by \eqref{eq:nonexist1}. 
Set \[Q_n:=\pi/(V_n\oplus \varpi_D V_n).\] 
By the assumption on $\gr_{\fm_D}(\pi^{\vee})$, one checks that 
\begin{equation}\label{eq:nonexist1-Q}\gr_{\fm_D}(Q_n^{\vee})\cong 
\big(\psi^{\vee}\otimes\F[y]/(y^{p+1-n})\big)\oplus \big((\psi^{\sigma})^{\vee}\otimes\F[z]/(z^{p+1-n})\big)\end{equation}
where $\psi:=\chi\alpha^{-(n+1)}$ (the case $\psi=\psi^{\sigma}$ is allowed). In particular, $\soc_{\cO_D^{\times}} Q_n\cong \psi\oplus \psi^{\sigma}$. Pulling back $\psi^{\sigma}$ we obtain a subrepresentation of $\pi|_{\cO_D^{\times}}$ which is an extension of $\psi^{\sigma}$ by $V_n\oplus \varpi_DV_n$, hence induces an extension class in $\Ext^1(\psi^{\sigma},V_n)$ which we denote by $c$. 
We have the following possibilities.
\begin{itemize}
\item[$\bullet$] Assume  $c=0$. Then $\pi$  contains a subrepresentation $V_{n+1}'$ fitting in an extension \[0\ra \varpi_DV_n\ra V_{n+1}'\ra \psi^{\sigma}\ra0.\]
Moreover, $V_{n+1}'$ is uniserial by checking  $V_{n+1}'[\fm_D^{n+1}]
$ which is contained in $\pi[\fm_D^{n+1}]$. We take $V_{n+1}:=\varpi_D V_{n+1}'$ and need to show that $\dim_{\F}V_{n+1}[Y]=1$. Since $\varpi_D^{-1}Y\varpi_D=Z$ (see Definition \ref{defn-X+X-}), equivalently we need to show  
\[\dim_{\F}V_{n+1}'[Z]=1.\]
This is clear if $n<p-1$, because  $V_{n+1}'$ is isomorphic to $E^+(\chi^{\sigma},n+1)$ by the conjugate version of Proposition \ref{prop-H1=dim3} (and its proof). If $n=p-1$, then Proposition \ref{prop-H1=dim3} shows that $V_{n+1}'$ can not be uniserial, i.e. this case can not happen. 

\item[$\bullet$] Assume $c$ is nonzero and $n<p-1$. Then by Proposition \ref{prop-H1=dim3} we have 
$\psi^{\sigma}\in \{\chi,\chi\alpha,\chi\alpha^{-(n+1)}\}$. Using \eqref{eq:nonexist1-Q} and the general fact recalled above, we see that $E^+(\psi^{\sigma})$ embeds in $Q_n$, and induces a natural extension class in $\Ext^1(E^+(\psi^{\sigma}),V_n)$ whose image  in $\Ext^1(\psi^{\sigma},V_n)$ is just $c$, hence is non-zero. 
By Proposition \ref{prop-Ext1-psi-n}, we must have \begin{equation}\label{eq:psi-nonexist}\psi^{\sigma}(=\chi^{\sigma}\alpha^{n+1})=\chi\alpha.\end{equation} Note that $\dim_{\F}\Ext^1(\psi^{\sigma},V_n)=1$ by Proposition \ref{prop-H1=dim3}, and it is easy to see that the natural morphism
\[\Ext^1(\psi^{\sigma},\chi)\ra \Ext^1(\psi^{\sigma},V_n)\]
is an isomorphism for the reason of dimensions. Thus, $c$ comes from the pushout of the unique non-split extension in $\Ext^1(\psi^{\sigma},\chi)$.  We deduce that $\pi|_{\cO_D^{\times}}$ contains a subrepresentation which is an extension of $\psi^{\sigma} (=\chi^{\sigma}\alpha^{n+1})$ by $\chi\oplus \varpi_DV_n$. Since $\varpi_DV_n\cong E^+(\chi^{\sigma},n)$ and $\psi^{\sigma}=\chi\alpha\neq \chi\alpha^{-1}$ by \eqref{eq:psi-nonexist}, we may apply Lemma \ref{lem:split1} to obtain a subrepresentation $V_{n+1}'$ of $\pi$ which is isomorphic to $E^+(\chi^{\sigma},n+1)$. The claim follows by taking $V_{n+1}:=\varpi_DV_{n+1}'$.

\item[$\bullet$] Assume $c$ is nonzero and $n=p-1$. Then by Proposition \ref{prop-H1=dim3} we have 
$\psi^{\sigma}\in \{\chi,\chi\alpha\}$ and again  $E^+(\psi^{\sigma})$ embeds in $Q_n$. In this case, we can not apply Proposition \ref{prop-Ext1-psi-n} to obtain $\psi^{\sigma}=\chi\alpha$. However, since $\psi:=\chi\alpha^{-(n+1)}=\chi\alpha$ (as $n=p-1$), we get $\psi^{\sigma}=\chi^{\sigma}\alpha^p$ and so the assumption on $\chi$ excludes the case $\psi^{\sigma}=\chi$. Therefore,  $\psi^{\sigma}=\chi\alpha$ (in particular $\psi^{\sigma}=\psi$) and we may conclude as in the last case. \qedhere
\end{itemize}
\end{proof}

\begin{theorem}\label{thm:nonexist2}
Let $\chi:\cO_D^{\times}\ra \F^{\times}$ be a smooth character.  Then there does not exist any smooth $D^{\times}$-representation $\pi$ with a central character such that 
\begin{equation}\label{eq:nonexist2}\gr_{\fm_D}(\pi^{\vee})\cong \big(\chi^{\vee}\otimes \F[y]/(y^{p+3})\big)\oplus \big((\chi^{\sigma})^{\vee}\otimes \F[z]/(z^{p+3})\big).\end{equation}
\end{theorem}

\begin{proof}
We use the notation introduced in Remark \ref{rem:nongeneric}. If $\chi\notin\{\xi^{s(p+1)}, \xi^{1+s(p+1)}\}$ for any $s\geq0$, the result follows directly from Theorem \ref{thm:nonexist1} together with Remark \ref{rem:nongeneric}. 

Assume $\chi=\xi^{s(p+1)}$ or $\chi=\xi^{1+s(p+1)}$ for some $s\geq 0$. Note that $\pi[\fm_D]$ is a $D^{\times}$-subrepresentation of $\pi$. Put $\pi'=\pi/\pi[\fm_D]$. Then $\pi'^{\vee}$ is identified with $\fm_D\pi^{\vee}$, so the $\fm_D$-adic filtration on $\pi^{\vee}$ is exactly the induced filtration from $\pi^{\vee}$. We deduce that \[\gr_{\fm_D}(\pi'^{\vee})\cong   \big(\fm_D\gr_{\fm_D}(\pi^{\vee})\big)(1)\cong \big(\chi'^{\vee}\otimes \F[y]/(y^{p+2})\big)\oplus \big((\chi'^{\sigma})^{\vee}\otimes \F[z]/(z^{p+2})\big)\]
where $\chi':=\chi\alpha^{-1}$. It is direct to check that $\chi'$ is \emph{not} of the form $\xi^{s'(p+1)}$ or $\xi^{1+s'(p+1)}$, so 
we  conclude by  Theorem \ref{thm:nonexist1}.  
\end{proof}

\section{A finiteness criterion} \label{sec:criterion}
We keep the notation in  Section \ref{sec:modp}. In this section, we prove a criterion for an admissible unitary Banach space representation of $D^{\times}$  to be topologically of finite length. 

\subsection{A class of Cohen-Macaulay modules} \label{sec:CA}
In this subsection, we consider a special class of graded modules over the graded ring $\F[y,z]/(yz)$. We refer to \cite[\S1.5]{BH-CM} for basic facts about graded rings and modules.

Let $A=\F[y,z]/(yz)$, viewed as a graded ring by setting $\deg y=\deg z=1$. Then $A$ is a \sta local ring with \sta maximal ideal $\fm_A=(y,z)$, see \cite[Exam.~1.5.14(b)]{BH-CM}. Let $\mathcal{M}(A)$ denote the category of graded $A$-modules, with morphisms being graded morphisms of degree $0$. 
Let $\mathcal{M}^{\rm fg}(A)$ denote the subcategory of finitely generated graded $A$-modules. 

If $N\in \mathcal{M}(A)$, let $N(a)$ be the shifted graded module defined by $N(a)_n:=N_{n+a}$. 
If $N\neq0$ is finitely generated, then we can always shift $N$ so that $N_{-1}=0$ and $N_0\neq0$. 

The ring $A$ admits two minimal graded prime ideals, namely $\p_0:=(y)$ and $\p_1:=(z)$. For $N\in \mathcal{M}^{\rm fg}(A)$ we define $\mu(N)$ as 
\[\mu(N):=m_{\p_0}(N)+m_{\p_1}(N)=\mathrm{length}_{A_{\p_0}}(N_{\p_0})+\mathrm{length}_{A_{\p_1}}(N_{\p_1}).\]

We can define the notions of \sta dimension and \sta depth for $N\in\mathcal{M}^{\rm fg}(A)$,  see \cite[\S1.5]{BH-CM}, so we have the notion of Cohen-Macaulay graded 
$A$-modules. For the modules of interest to us, these notions just coincide with the usual ones. For example, $A$ itself is Cohen-Macaulay as a \sta local ring.  

We can also define\footnote{This need to enlarge the set of morphisms, see \cite[\S1.5]{BH-CM}.} ${}\mathrm{{}^*Ext}^i_A(N,N')$ for two objects $N,N'$ in $\mathcal{M}(A)$,  and it is known that  
\begin{equation}\label{eq:*Ext}
\mathrm{{}^*Ext}^i_{A}(N,N')=\Ext^i_A(N,N')\end{equation}
for $N\in\mathcal{M}^{\rm fg}(A)$, where $\Ext^i_A$ denotes the usual $\Ext$ groups for $A$-modules.
We first determine $\Ext^1_{A}(N,N')$ for $N,N'\in \{\F, \F[y],\F[z]\}$, where $\F$ (resp.~$\F[y]$, $\F[z]$) is viewed as a graded $A$-module via $A/\fm_A\cong \F$ (resp.~$A/\p_1\cong \F[y]$, $A/\p_0\cong \F[z]$).

\begin{lemma}\label{lem:Ext-A}
For $N,N'\in\{\F,\F[y],\F[z]\}$, $\Ext^1_A(N,N')$ is given by the following table. \medskip 

{\begin{center} \label{eq:table}
\begin{tabular}{ |c||c|c|c| }
\hline
$\Ext^1_A(N,N')$& $N=\F$ &  $N=\F[y]$ &  $N=\F[z]$    \\
 \hline\hline
$N'=\F$&$\F\oplus \F$ & $0$&$0$  \\
\hline
$N'=\F[y]$&$\F$&$0$&$\F$ \\
\hline
$N'=\F[z]$&$\F$&$\F$&$0$ \\
\hline
\end{tabular}
\end{center}}
\end{lemma}

\begin{proof}
This is an easy exercise. For example, $\F$ admits a free $A$-module resolution 
\[\cdots \lra A\oplus A\overset{\smatr{z}{0}{0}{y}}{\lra} A\oplus A\overset{\binom{y}{z}}{\lra} A \ra \F\ra0.\]
Applying $\mathrm{Hom}_A(-,\F)$ induces a complex which computes $\Ext^i_A(\F,\F)$:
\[0\ra\F\overset{(y\ z)}{\lra} \F\oplus \F\overset{\smatr{z}00y}{\lra} \F\oplus \F\ra \cdots\]
which clearly implies that $\dim_{\F}\Ext^1_A(\F,\F)=2$. 
\end{proof}

We deduce from Lemma \ref{lem:Ext-A} the dimension of $\mathrm{{}^*Ext}^1_A(N,N')$ for $N,N'\in\{\F,\F[y],\F[z]\}$ via \eqref{eq:*Ext}. The following  lemma makes (some of) the groups $\mathrm{{}^*Ext}^1_A(N,N')$  more concrete and can be directly checked. 

\begin{lemma}\label{lem:*Ext-A}
(i) A basis of $\mathrm{{}^*Ext}^1_A(\F,\F[y])$ is represented by the extension
\[0\ra \F[y](-1)\overset{y}{\lra} \F[y]\ra \F\ra0.\] 

(ii) A basis of $\mathrm{{}^*Ext}_A^1(\F[z],\F[y])$ is represented by the extension
\[0\ra \F[y](-1)\overset{y}{\lra} A\ra \F[z]\ra0.\]

(iii) Analogous statements hold for $\mathrm{{}^*Ext}^1_A(\F,\F[z])$ and $\mathrm{{}^*Ext}^1_A(\F[y],\F[z])$.
\end{lemma}

\begin{proposition}\label{prop:primefilt}
Let $N\in \mathcal{M}^{\rm fg}(A)$ be a Cohen-Macaulay module of   dimension $1$. 
Then $N$ has a finite filtration by graded submodules
\[(0)=N_0\subset \cdots \subset N_{t-1}\subset N_t=N\]
such that, for each $i\in\{0,\dots,t-1\}$, $N_{i+1}/N_{i}\cong A/\mathfrak{q}_i$ (up to shift) for $\mathfrak{q}_i\in\{\p_0,\p_1\}$.
\end{proposition} 
\begin{proof} 
First observe that the only graded prime ideals of $A$ are $\p_0,\p_1, \fm_A$. Since $N$ is Cohen-Macaulay of dimension $1$, we have $\mathrm{Ass}(N)\subset\{\p_0,\p_1\}$. 

We will prove the result by induction on $\mu(N)$. Choose $\mathfrak{q}_0\in \mathrm{Ass}(N)$. By \cite[Lem.~1.5.6]{BH-CM}, $\mathfrak{q}_0$ is the annihilator of some homogeneous element, so we obtain an embedding $N_1:=(A/\mathfrak{q}_0)(r)\hookrightarrow N$ for some $r\in\Z$.  If the quotient $N/N_1$ is Cohen-Macaulay of dimension $1$, then we may conclude by induction as $\mu(N/N_1)=\mu(N)-1$. Otherwise, $N/N_1$ contains a graded submodule isomorphic to $\F$ (up to shift), say $Q_1$, and pulling back we obtain a submodule $N_1'\subset N$ satisfying
\[0\ra N_1\ra N_1'\ra \F\ra0.\]
This sequence does not split because $\fm_A\notin \mathrm{Ann}(N)$. But by Lemma \ref{lem:*Ext-A}, $N_1'$ is  isomorphic to $(A/\mathfrak{q}_0)(r+1)$, so we may replace $N_1$ by $N_1'$ and continue the argument. It is easy to see that after a finite number of steps we arrive at some graded submodule $N_1$ such that either $N=N_1$, or  $N/N_1$ is still Cohen-Macaulay of dimension $1$, and we conclude by induction.
\end{proof}

\begin{proposition}\label{prop:CM-class}
Let $N\in\mathcal{M}^{\rm fg}(A)$.  Assume that $N$ is Cohen-Macaulay of  dimension $1$ and $\mu(N)\leq 2$. 
\begin{itemize}
\item[(i)] If $\mu(N)=1$, then $N$ is isomorphic to $\F[y](r)$ or $\F[z](r)$ for some $r\in\Z$.
\item[(ii)] If $\mu(N)=2$ and assume $m_{\p_0}(N)=m_{\p_1}(N)=1$, then $N$ is isomorphic to $A(r)$ or $\F[y](r)\oplus \F[z](s)$ for some $r,s\in\Z$. 
\end{itemize} 
\end{proposition}
\begin{proof}
(i) It directly follows from Proposition \ref{prop:primefilt}. 

(ii) By Proposition \ref{prop:primefilt} we may find $N_1\subsetneq N$ such that, up to twist, $N_1\cong A/\mathfrak{q}_0$ and $N/N_1\cong A/\mathfrak{q}_1$ for $\mathfrak{q}_i\in\{\p_0,\p_1\}$. Moreover, the assumption implies that $\{\mathfrak{q}_0,\mathfrak{q}_1\}=\{\p_0,\p_1\}$. Using Lemma \ref{lem:*Ext-A}, we see that $N\cong A(r)$ for some $r$ if the corresponding sequence $0\ra N_1\ra N\ra N/N_1\ra0$ does not split, and $N\cong \F[y](r)\oplus \F[z](s)$ otherwise. 
\end{proof}

\subsection{Pure modules over $\gr_{\fm_D}(\Lambda)$}
Recall that $\Lambda:=\F[\![U_D^1/Z_D^1]\!]$ with maximal ideal $\fm_D$.  
Let $J$ denote the two-sided ideal $(yz,zy)$ of $\gr(\Lambda)$, cf.~\S\ref{sec:uniserial}. 
If $N$ is a finitely generated graded $\gr(\Lambda)$-module annihilated by $J^n$ for some $n\geq 1$ and $\mathfrak{q}$ is a minimal graded prime ideal of $\gr(\Lambda)/J$, we  
  define the multiplicity of $N$ at $\mathfrak{q}$ to be
\begin{equation}
m_{\mathfrak{q}}(N)=\sum_{i=0}^nm_{\mathfrak{q}}(J^iN/J^{i+1}N).
\end{equation}
where $m_{\mathfrak{q}}(-)$ denotes the length of $(-)_{\mathfrak{q}}$ over $(\gr(\Lambda)/J)_{\mathfrak{q}}$. 
By the same proof as \cite[Lem.~3.1.4.3]{BHHMS2}, $m_{\mathfrak{q}}(-)$ is additive with respect to short exact sequences.

Note that $\gr(\Lambda)/J\cong \F[y,z]/(yz)$ is just the ring $A$ considered in \S\ref{sec:CA}; it has two minimal graded prime ideals $\p_0$ and $\p_1$.   For $N$ as above, we set
\[\mu(N)=m_{\p_0}(N)+m_{\p_1}(N).\]

The following lemma is obvious; we refer to \S\ref{sec:Auslander} for the notion of grade.
\begin{lemma}\label{lem:grade-mu}
If $N$ is a finitely generated graded 
$\gr(\Lambda)$-module annihilated by $J^n$ for some $n\geq 1$, then $N$ has grade $\geq 2$. Moreover, it has grade $2$ (resp.~$3$) if and only if $\mu(N)\neq 0$ (resp.~$\mu(N)=0$.)
\end{lemma}
Let $\widetilde{H}$ denote the subgroup of $D^{\times}$: \[\widetilde{H}:=H\rtimes  \varpi_D^{\Z}.\]  
It is easy to see that   $\widetilde{H}\cong D^{\times}/U_D^1$, so $\widetilde{H}$ naturally acts  on $\Lambda$ and $\gr(\Lambda)$.  We say $N$ is a $\gr(\Lambda)$-module with \emph{compatible $\widetilde{H}$-action} if 
\begin{equation}\label{eq:compatible}\tilde{h}(rv)=(\tilde{h}r)(\tilde{h}v),\ \ \ \tilde{h}\in\widetilde{H}, r\in\gr(\Lambda), v\in N.\end{equation}
\begin{lemma}\label{lem:m1=m2}
Let $N$ be a finitely generated graded $\gr(\Lambda)$-module annihilated by $J^n$ for some $n\geq 1$. Assume that $N$ carries  a compatible $\widetilde{H}$-action.  Then 
$m_{\p_0}(N)=m_{\p_1}(N).$  
In particular, $\mu(N)\in 2\Z_{\geq 0}$.
\end{lemma}

\begin{proof}
The action of $\varpi_D$  induces an involution on $\gr(\Lambda)$ which sends $\p_0$ to $\p_1$ and vice versa. Since $N$ carries a compatible $\widetilde{H}$-action, the action of $\varpi_D$  induces an isomorphism $N_{\p_0}\simto N_{\p_1}$, which implies the result. 
\end{proof}

\begin{lemma}\label{lem:CM-mod}
Let $N$ be a pure  (cf.~Definition \ref{def:CM}), finitely generated graded $\gr(\Lambda)$-module annihilated by $J^n$ for some $n\geq 1$. Assume that $N$ carries a compatible $\widetilde{H}$-action and $\mu(N)=2$. Then there exists a character $\chi:H\ra \F^{\times}$ and $r\in\Z$ such that $N(r)\cong (\chi\otimes \F[y])\oplus (\chi^{\sigma}\otimes \F[z])$, or $N(r)\cong \chi\otimes \F[y,z]/(yz)$ in which case $\chi=\chi^{\sigma}$. \end{lemma}

\begin{proof}
We first prove that $N$ is annihilated by $J$. Note that $J=(h,yz)$ where   $h:=yz-zy$, and $(h,yz)$ form a sequence of central elements in $\gr(\Lambda)$.  Consider the exact sequence 
\[0\ra N[h]\ra N\overset{ h}{\ra} N\ra N/hN\ra0.\]
If  $\mathrm{Im}(h)=hN\neq0$, then $\mu(hN)\neq 0$ as $N$ is pure of grade $2$, see Lemma \ref{lem:grade-mu}.  Since $hN$ is  stable under the action of $\widetilde{H}$,  Lemma \ref{lem:m1=m2} implies that $\mu(hN)=2$ (as $\mu(N)=2$), and so \[\mu(N[h])=\mu(N)-\mu(hN)=0.\] This means that $N[h]$ has grade $3$, hence is zero because $N$ is pure. In other words, $h$ is injective on $N$. However, by assumption $N$ is annihilated by  $J^n$, hence also annihilated by  $h^n$, a contradiction.  
Therefore, $\mathrm{Im}(h)=0$ and $N$ is annihilated by $h$. In a similar way one shows that $N$ is annihilated by $yz$, hence by $J$. 

By Lemma \ref{lem:pure=CM}, being pure of grade $2$ over $\gr(\Lambda)$ is the same as being Cohen-Macaulay of grade $2$. Together with the above discussion, we deduce that $N$ is a Cohen-Macaulay module over $\gr(\Lambda)/J\cong A$.  Taking into account the action of $\widetilde{H}$ on $N$ which preserves the degree, the result follows from Proposition \ref{prop:CM-class}.
\end{proof}
 
\begin{corollary}\label{cor:CM-gen}
Let $N$ be as in Lemma \ref{lem:CM-mod}.   Assume moreover that $N_{-1}=0$ and $N_0\neq0$. Then   $N$ is generated by $N_0$ over $\gr(\Lambda)$.
\end{corollary}
\begin{proof}
A direct consequence of Lemma \ref{lem:CM-mod}.
\end{proof}
\subsection{The $\fm_D$-adic filtration vs Gabber filtration} Let $M$ be a  pure finitely generated  $\Lambda$-module. In this subsection we address   the question if $\gr_{\fm_D}(M)$ is still pure as a $\gr(\Lambda)$-module, in other words, if the $\fm_D$-adic filtration on $M$ is a Gabber filtration (cf.~Appendix \S\ref{sec:app}).

Let $F$ be a good filtration on $M$. If $\gr_F(M)$  is annihilated by some power of $J$, then for any good filtration $F'$ on $M$, $\gr_{F'}(M)$ is also annihilated by some power of $J$ and 
$\mu(\gr_F(M))=\mu(\gr_{F'}(M))$, see \cite[Lem.~3.3.4.3]{BHHMS2}.
By abuse of notation, we often write $\mu(M)$ instead of $\mu(\gr_F(M))$ for any choice of good filtration on $M$. 
\begin{proposition}\label{prop:aut-CM}
Assume that 
\begin{itemize}
\item[(i)] $M$ carries  a compatible $\widetilde{H}$-action in a sense similar to \eqref{eq:compatible};
\item[(ii)] $\gr_{\fm_D}(M)$ is annihilated by some power of $J$ and $\mu(\gr_{\fm_D}(M))=2$. 
\end{itemize}  
Then $\gr_{\fm_D}(M)$ is a pure $\gr(\Lambda)$-module and is annihilated by $J$.
\end{proposition}

\begin{proof}
There is an equivalence relation on the family of good filtrations on $M$, see Definition \ref{def:bjork}. 
Consider the equivalence class which contains the $\fm_D$-adic filtration. By Theorem \ref{thm:bjork}, it contains a unique filtration $F$ such that $\gr_F(M)$ is pure. Fix $\tilde{h}\in\widetilde{H}$. Define another filtration $F'$ by $F'^nM:=\tilde{h}\cdot F^nM$. Then $F'$ is also equivalent to the $\fm_D$-adic filtration, because the $\fm_D$-adic filtration is invariant under the action of  $\tilde{h}$. Thus we must have $F=F'$ by Theorem \ref{thm:bjork}, namely $F$ is invariant under the action of $\widetilde{H}$. 
As a consequence, $\gr_F(M)$ carries a compatible action of $\widetilde{H}.$ Since $\mu(\gr_F(M))=\mu(\gr_{\fm_D}(M))=2$, Lemma \ref{lem:CM-mod} implies that $\gr_F(M)$ is annihilated by $J$, so it suffices to prove that $F$ coincides with the $\fm_D$-adic filtration. 

By Theorem \ref{thm:bjork}, we   have $\fm_D^nM\subseteq F^nM$  which induces a \emph{non-zero} morphism of graded modules $f:\gr_{\fm_D}(M)\ra \gr_F(M)$. 
 Since $\gr_{\fm_D}(M)$ is generated by $\gr^0_{\fm_D}(M)$ over $\gr(\Lambda)$,  the map  \[f^0:\gr_{\fm_D}^0(M)\ra \gr_F^0(M)\] is also nonzero. In particular, $\gr_F^0(M)$ is nonzero, i.e. $F^1M\subsetneq F^0M$.\footnote{We only consider descending filtrations; this is different from the convention of \cite{LiO}.}  On the other hand,  
for $n=0$ we have \[M=\fm_D^0M\subseteq F^0M\subseteq M,\] thus $F^0M=M$ and the map $f^0$ is surjective.  In all, we see that $\gr_F(M)$ satisfies the assumption of Corollary \ref{cor:CM-gen}, so $\gr_F(M)$ is generated by $\gr_F^0(M)$ over $\gr(\Lambda)$. As a consequence, since $f^0$ is surjective,  $f$ is also surjective by  Nakayama's lemma (for graded modules over a \sta local ring).   Hence, the inclusion $\fm_D^nM\subseteq F^nM$ is an equality for any $n\geq 0$. 
\end{proof}

\subsection{A lower bound for multiplicity}\label{sec:lowerbound}
Consider the category $\mathcal{C}$ of admissible smooth $\F$-representations $\pi$ of $D^{\times}$ with a central character, such that for some  (equivalently any) good   filtration $F$ on $\pi^{\vee}$  the graded module $\gr_{F}(\pi^{\vee})$ is annihilated by a finite power of $J$.   It is clear that $\mathcal{C}$ is an abelian category and is stable under subquotients and extensions.

For $\pi\in\mathcal{C}$, we set \[\mu(\pi):=\mu(\gr_{\fm_D}(\pi^{\vee}))\]
and call it the \emph{multiplicity} of $\pi$.

\begin{theorem}\label{thm:mu=4}
If $\pi\in\mathcal{C}$  such that   $\mu(\pi)\neq 0$, 
then $\mu(\pi)\geq 4$.
\end{theorem}
\begin{proof}
 Write $M=\pi^{\vee}$. By Remark \ref{rem:CMquotient}, $M$ always has a pure quotient which has the same multiplicity as $M$. Thus  we may assume $M$ is pure. 

We assume $\mu(\pi)<4$ for a contradiction. By Lemma \ref{lem:m1=m2}, we must have $\mu(\pi)=2$ as $\mu(\pi)\neq 0$ by assumption. By Proposition \ref{prop:aut-CM}, $\gr_{\fm_D}(\pi^{\vee})$ is automatically a pure $\gr(\Lambda)$-module and is annihilated by $J$. 
We have the following possibilities by Lemma \ref{lem:CM-mod}. 
\begin{itemize}
\item[(a)] $\gr_{\fm_D}(\pi^{\vee})$ has the form $(\chi\otimes \F[y])\oplus (\chi^{\sigma}\otimes\F[z])$  for some smooth character $\chi$ of $\cO_D^{\times}$. Then $\pi[\fm_D^{p+3}]$ would be a representation of $D^{\times}$ satisfying \eqref{eq:nonexist2}, a contradiction with Theorem \ref{thm:nonexist2}.
\item[(b)] $\gr_{\fm_D}(\pi^{\vee})$ has the form $\chi^{\vee}\otimes \F[y,z]/(yz)$
for some character $\chi$ satisfying $\chi=\chi^{\sigma}$. 
 Then $\pi[\fm_D]\cong \chi$ is one-dimensional and is stable under $D^{\times}$. Set $\pi':=\pi/(\pi[\fm_D])$ which is again a representation of $D^{\times}$, and $\pi'[\fm_D]\cong \chi\alpha\oplus \chi\alpha^{-1}$ is two-dimensional.  It is clear that  $\pi'^{\vee}$ is identified with $\fm_D\pi^{\vee}$, so the $\fm_D$-adic filtration on $\pi'^{\vee}$ coincides (up to a shift) with the induced one from  the $\fm_D$-adic filtration on $\pi^{\vee}$.   We deduce that $\gr_{\fm_D}(\pi'^{\vee})$ is as in the case (a), again a contradiction with Theorem \ref{thm:nonexist2}.\qedhere
\end{itemize}
\end{proof}
 
\subsection{A finiteness criterion}

Let $\Pi$ be an admissible unitary Banach space representation of $D^{\times}$  over $E$, with a central character. 
\begin{lemma}\label{lem:lattice-mu}
Assume that $\Theta/\varpi\Theta$ lies in $\mathcal{C}$ for some open bounded $D^{\times}$-invariant lattice $\Theta$ in $\Pi$. Then the same is true for any open bounded $D^{\times}$-invariant lattice in $\Pi$. Moreover, the quantity $\mu(\Theta/\varpi\Theta)$  does not depend on the choice of $\Theta$.
\end{lemma}

\begin{proof}
Let   $\Theta'$ be another open bounded $D^{\times}$-lattice in $\Pi$.  Then $\Theta$ and $\Theta'$ are commensurable, so we may assume 
$\varpi^n\Theta'\subset \Theta\subset \Theta'$
for some  $n\geq 1$.  
As a consequence, $\Theta'/\Theta$ is killed by $\varpi^n$. The exact sequence
$0\ra \Theta\ra \Theta'\ra \Theta'/\Theta\ra0$ 
 induces, by tensoring with $\cO/\varpi^n$, an exact sequence 
\begin{equation}\label{eq:lattice-mu}0\ra \Theta'/\Theta\ra \Theta/\varpi^n\Theta\ra \Theta'/\varpi^n\Theta'\ra \Theta'/\Theta\ra0.\end{equation}
Recall that the category $\mathcal{C}$ is stable under subquotients and extensions. The assumption implies that $\Theta/\varpi^n\Theta\in \mathcal{C}$, hence  $\Theta'/\Theta\in \mathcal{C}$ too,  and finally $\Theta'/\varpi^n\Theta'\in\mathcal{C}$. This implies the first assertion.
Moreover, we  deduce from \eqref{eq:lattice-mu} that $\mu(\Theta/\varpi^n\Theta)=\mu(\Theta'/\varpi^n\Theta')$ by the additivity of $\mu(\cdot)$, from which the second assertion follows. 
\end{proof}

By abuse of notation, we say $\Pi\in \widehat{\mathcal{C}}$ if $\Pi$ satisfies the condition of Lemma \ref{lem:lattice-mu}. In this case, we  simply write  $\mu(\Pi)$ for $\mu(\Theta/\varpi\Theta)$ by choosing an open bounded $D^{\times}$-invariant lattice $\Theta$. 
Note that $\mu(\cdot)$ is additive with respect to short exact sequences of admissible unitary Banach space representations of $D^{\times}$ (provided that all the terms make sense). 

\begin{corollary}\label{cor:mu=4}
Assume $\Pi\in\widehat{\mathcal{C}}$ and $\mu(\Pi)\neq0$. Then $\mu(\Pi)\geq 4$. 
\end{corollary}

\begin{proof}
By the above discussion, it is a direct consequence of Theorem \ref{thm:mu=4}.
\end{proof}
Denote by $\Pi^{\rm lalg}$ the subspace of locally algebraic vectors in $\Pi$.

\begin{theorem}\label{thm:finitelength}
Keep the above notation.  Assume that 
\begin{itemize}
\item[(a)] $\Pi^{\rm lalg}$ is finite dimensional over $E$.
\item[(b)]   $\Pi\in \widehat{\mathcal{C}}$ and $\mu(\Pi)\leq 4$.
\end{itemize}
Then  $\Pi$ is topologically of finite length.  
\end{theorem} 
\begin{proof}
We may assume that $\Pi$ itself is infinite dimensional over $E$. As a consequence, we have $\mu(\Pi)>0$.  

By \cite[Prop.~6.13]{Paskunas-JL},  any unitary representation of $D^{\times}$ on a finite dimensional $E$-vector space (with a central character) is locally algebraic. Hence, the quotient $\Pi/\Pi^{\rm lalg}$ contains an irreducible closed subrepresentation of infinite dimension, say $\Pi'$; see the proof of  \cite[Thm.~6.13]{Paskunas-JL}. Let $\Pi_1\subseteq \Pi$  be the pullback of $\Pi'$, i.e. \[0\ra \Pi^{\rm lalg}\ra \Pi_1\ra \Pi'\ra0.\]
Then $\mu(\Pi_1)=\mu(\Pi')>0$.  It suffices to prove that $\Pi/\Pi_1$ is finite dimensional over $E$, equivalently $\mu(\Pi)=\mu(\Pi_1)$. But this is clear,  because $\mu(\Pi)\leq 4$ by Condition (b) and $\mu(\Pi_1)\geq 4$  by Corollary \ref{cor:mu=4}. 
\end{proof}

\begin{remark}
The assumption $\mu(\Pi)\leq 4$ in Theorem \ref{thm:finitelength} is crucial, otherwise it could happen that $\mu(\Pi/\Pi_1)>0$ and one could not exclude the possibility that $ (\Pi/\Pi_1)^{\rm lalg}$ is infinite dimensional (cf.~\cite[Rem.~1.3]{DPS22}).
\end{remark}

\begin{corollary}\label{cor:finitelength}
Keep the assumptions in Theorem \ref{thm:finitelength}. Assume moreover  $\mu(\Pi)=4$ and that there exists an open bounded $D^{\times}$-invariant  lattice $\Theta$ in $\Pi$ such that $(\Theta/\varpi\Theta)^{\vee}$ is Cohen-Macaulay as a finitely generated $\Lambda$-module. Then $\Pi/\Pi^{\rm lalg}$ is topologically irreducible. 
In particular, if $\Pi^{\rm lalg}=0$ then $\Pi$ is topologically irreducible.
\end{corollary}
\begin{proof}
Let $\Pi_1$ be as in the proof of Theorem \ref{thm:finitelength}, and $\Theta_1:=\Pi_1\cap \Theta$. Then \[\Theta_1/\varpi\Theta_1\hookrightarrow\Theta/\varpi\Theta\]  whose cokernel is finite dimensional over $\F$, because $\mu(\Theta_1/\varpi\Theta_1)=\mu(\Theta/\varpi\Theta)=4$.  The assumption then implies that  $\Theta_1/\varpi\Theta_1=\Theta/\varpi\Theta$, hence $\Pi_1=\Pi$.  
\end{proof}

\section{Scholze's functor and multiplicity one result}
 \label{sec:multione}

Let $G=\GL_2(\Q_p)$, $K=\GL_2(\Z_p)$ and $B$ be the (upper) Borel subgroup of $G$. For each $n\geq 1$, let
\[K_n:=1+p^n M_2(\Z_p)\]
be the $n$-th principal congruence subgroup.
Let $Z_G$ be the centre of $G$ and $Z_1:=Z_G\cap K_1$.

Let $\Mod_G^{\rm sm}(\cO)$ denote the category of smooth representations of $G$ on $\cO$-torsion modules, and $\Mod_G^{\rm adm}(\cO)$ (resp.~$\Mod_G^{\rm l,adm}(\cO)$) denote the full subcategory of admissible (resp.~locally admissible) representations. If $\zeta:Z_G\ra \cO^{\times}$ is a continuous character, we add 
the subscript $\zeta$ in the above categories to indicate the full subcategory consisting of representations on which $Z_G$ acts by the character $\zeta$. 

The Pontryagin duality $M\mapsto M^{\vee}=\Hom_{\cO}^{\rm cont}(M,E/\cO)$ induces an anti-equivalence between the category 
of discrete $\cO$-modules and the category of compact $ \cO$-modules. 
Let $\Mod_{G}^{\rm pro}(\cO)$ be the category anti-equivalent to   $\Mod_G^{\rm sm}(\cO)$ under this duality. Let  
 $\mathfrak{C}_G(\cO)$ (resp.~$\mathfrak{C}_{G,\zeta}(\cO)$) be the category which is anti-equivalent to $\Mod_{G}^{\rm l,adm}(\cO)$ (resp.~$\Mod_{G,\zeta}^{\rm l,adm}(\cO)$). 
Note that $Z_G$ acts on objects in $\mathfrak{C}_{G,\zeta}(\cO)$ by $\zeta^{-1}$.

For each of the above categories, we replace $\cO$ by $\F$ to indicate the corresponding full subcategory consisting of objects which are killed by $\varpi$. 
\medskip

Finally, we keep the notation on the non-split quaternion algebra $D$ introduced in \S\ref{sec:modp}.

\subsection{Blocks of $G$}\label{ss:block}

Let  $\zeta:\Q_p^{\times}\ra \cO^{\times}$ be a continuous character.   
Recall that a \emph{block} in $\Mod_{G,\zeta}^{\rm sm}(\F)$ is an equivalence class of (absolutely) irreducible objects in $\Mod_{G,\zeta}^{\rm sm}(\F)$, where $\tau\sim \pi$ if and only if there exists a sequence of irreducible representations $\tau=\tau_0$, $\tau_1,\dots,\tau_n=\pi$ such that $\Ext^1_G(\tau_i,\tau_{i+1})\neq 0$ or $\Ext^1_G(\tau_{i+1},\tau_i)\neq 0$ for each $i$.

Let $\pi\in \Mod_{G,\zeta}^{\rm sm}(\F)$ be absolutely  irreducible and let $\sB$ be the block in which $\pi$ lies.  For $p\geq 5$, one of the following holds (cf. \cite[Prop.~5.42]{Pa13}):
\begin{enumerate}
\item[(I)] if $\pi$ is supersingular, then $\sB=\{\pi\}$;
\item[(II)] if $\pi= \Ind_{B}^G\chi_1\otimes\chi_2\omega^{-1}$ with $\chi_1\chi_2^{-1}\neq \ide,\omega^{\pm1}$, then
\[\sB=\big\{\Ind_B^G\chi_1\otimes\chi_2\omega^{-1},\ \Ind_B^G\chi_2\otimes\chi_1\omega^{-1}\big\};\]
\item[(III)] if $\pi=\Ind_B^G\chi\otimes\chi\omega^{-1}$, then $\sB=\{\pi\}$;
\item[(IV)] otherwise, $\sB=\{\chi\circ\det,\Sp\otimes\chi\circ\det, (\Ind_B^G\alpha)\otimes\chi\circ\det\}$, where $\alpha:=\omega\otimes\omega^{-1}$. 
\end{enumerate}

By \cite[Prop.~5.34]{Pa13}, the category $\Mod_{G,\zeta}^{\rm l,adm}(\cO)$ decomposes into a direct sum of subcategories 
\[\Mod_{G,\zeta}^{\rm l,adm}(\cO)=\bigoplus_{\sB}\Mod_{G,\zeta}^{\rm l,adm}(\cO)^{\sB}\]
where the direct sum is taken over all the blocks $\sB$ and the objects of $\Mod_{G,\zeta}^{\rm l,adm}(\cO)^{\sB}$ are representations with all the irreducible subquotients lying in $\sB$. Correspondingly, we  have  a decomposition of categories 
\[\mathfrak{C}_{G,\zeta}(\cO)=\prod_{\sB}\mathfrak{C}_{G,\zeta}(\cO)^{\sB}\]
where $\mathfrak{C}_{G,\zeta}(\cO)^{\sB}$ denotes the dual category of $\Mod_{G,\zeta}^{\rm l,adm}(\cO)^{\sB}$.  
 \medskip

Let $\brho:G_{\Q_p}\ra\GL_2(\F)$ be a continuous representation and $\pi(\brho)$ be the smooth representation of $G$ associated with $\brho$ by the mod $p$ Langlands
correspondence for $G$, cf. \cite{Br03}, \cite{Co}.   We normalize the correspondence in such a way that the central character of $\pi(\brho)$ is equal to $\det(\brho)\omega^{-1}$.  Using the explicit description of $\pi(\brho)$ (see for example \cite[\S4.2]{HW-JL1} for some cases), it is easy to check that\footnote{Here, $\brho^{\rm ss}$ means the semisimplification of $\brho$ and $\JH(\pi(\brho))$ is counted up to multiplicity.}   \[\brho^{\rm ss}\longmapsto \JH(\pi(\brho))\] establishes a bijection between the set of two-dimensional  semisimple $\F$-representations of $G_{\Q_p} $ and the set of blocks  of $G$.

From now on, we fix a representation $\brho:G_{\Q_p}\ra\GL_2(\F)$ and let $\sB$ be the block corresponding to $\brho^{\rm ss}$.  Let $R_{\sB}^{\rm ps,\delta}:=R^{\rm ps,\delta}_{\mathrm{tr}(\brho)}$ denote the universal pseudo-deformation ring of $\brho$ with fixed determinant $\delta:=\zeta\varepsilon$, where $\zeta:\Q_p^{\times}\ra \cO^{\times}$ is a (fixed) continuous character lifting the central character of $\sB$. We have the following important result (recall that $p\geq 5$).
\begin{theorem}[\cite{Pa13}] \label{thm:pas}
Colmez's functor induces an isomorphism between $R^{\rm ps,\delta}_{\sB}$ and the Bernstein centre of $\mathfrak{C}_{G,\zeta}(\cO)^{\sB}$.    In particular, $R^{\rm ps,\delta}_{\sB}$ naturally acts on any object in $\mathfrak{C}_{G,\zeta}(\cO)^{\sB}$ and any morphism in $\mathfrak{C}_{G,\zeta}(\cO)^{\sB}$ is $R^{\rm ps,\delta}_{\sB}$-equivariant.
\end{theorem}

For $\pi\in\sB$, let $P_{\pi}$ denote a projective envelope of $\pi^{\vee }$ in the category $\mathfrak{C}_{G,\zeta}(\cO)$. As a consequence of Theorem  \ref{thm:pas}, $R^{\rm ps,\delta}_{\sB}$ acts on $P_{\pi}$. 
Actually, when $\sB$ is not of type (IV), $P_{\pi}$ is a flat module over $R^{\rm ps,\delta}_{\sB}$ by \cite[Cor.~3.12]{Pa13}. 

\begin{lemma}\label{lem:P/m}
(i) If $\sB$ is of type (I), then $\pi$ is the unique object in $\sB$ and \[\big(P_{\pi}\otimes_{R^{\rm ps,\delta}_{\sB}}\F\big)^{\vee}\cong\pi.\]

(ii) If $\sB$ is of type (II), then 
there exists a non-split exact sequence 
\[0\ra \pi\ra \big(P_{\pi}\otimes_{R^{\rm ps,\delta}_{\sB}}\F\big)^{\vee}\ra \pi'\ra0\]
where $\pi'$ is the unique representation such that $\sB=\{\pi,\pi'\}$. 
 \end{lemma}
\begin{proof}
The results follow from \cite[Prop.~6.1]{Pa13} and \cite[Prop.~8.3]{Pa13}, respectively.  Remark that $\dim_{\F}\Ext^1_{G,\zeta}(\pi',\pi)=1$  in (ii). 
\end{proof}

\begin{proposition}\label{prop:x}
Assume that $\sB$ is of type (I) or (II) and let $\pi\in\sB$. There exists $x\in R_{\sB}^{\rm ps, \delta}$ such that  $P_{\pi}/x$ is isomorphic to a projective envelope of $(\soc_K\pi)^{\vee}$ in $\Mod_{K,\zeta}^{\rm pro}(\cO)$.
\end{proposition}
\begin{proof}
It is proved in \cite[Thm.~5.2]{Pa15}.  
\end{proof}

 We introduce a genericity condition on $\brho$, following \cite[\S4.3]{HW-JL1}. Let $I_{\Q_p}$ be the inertia subgroup of $G_{\Q_p}$ and $\omega_2$ be Serre's   fundamental character of niveau $2$.

\begin{definition}\label{def:generic}
(i) If $\brho$ is absolutely irreducible and $\brho|_{I_{\Q_p}}\sim \smatr{\omega_2^{r+1}}00{\omega_2^{p(r+1)}}$ (up to twist) with $0\leq r\leq p-1$, we say $\brho$ is \emph{generic} if  $2\leq r\leq p-3$.   

If $\brho$ is reducible  and $\brho|_{I_{\Q_p}}\sim \smatr{\omega^{r+1}}{*}0{1}$ (up to twist) with $0\leq r\leq p-2$, we say $\brho$ is \emph{generic} if $1\leq r\leq p-3$ and if $\brho^{\rm ss}\nsim \omega\oplus \ide$ up to twist. 

(ii) Let $\sB$ be a block of type (I) or (II). We say $\sB$ is \emph{generic} if there exists a generic representation $\brho$, such that $\sB$ corresponds to $\brho^{\rm ss}$.  
\end{definition}

\subsection{Smooth duals}
Let $\pi\in \Mod_{G}^{\rm adm}(\F)$. The Pontryagain dual $\pi^{\vee}$ is a finitely generated module over the Iwasawa algebra $\F[\![K_1]\!]$. It is explained in \cite[\S3]{Ko-dual} that for each $i\geq 0$ the module \[\Ext^i_{\F[\![K_1]\!]}\big(\pi^{\vee},\F[\![K_1]\!]\big)\] carries naturally a compatible action of  $G$ and gives an admissible smooth representation of $G$ by taking dual again. By \cite[Cor.~3.15]{Ko-dual} and \cite[Cor.~5.2]{Ko-dual}, this construction coincides with  $\mathrm{S}^{4-i}_G(\pi)$ 
(as $\dim K_1=4$), where $\mathrm{S}^j_G(-)$ denotes the $j$-th smooth duality functor for $G=\GL_2(\Q_p)$ defined in \cite[Def.~3.12]{Ko-dual}.  

Following \cite{Ko-dual}, we say that $\pi$ is \emph{Cohen-Macaulay} if $\mathrm{S}^j_G(\pi)=0$ for all but one degree $j_0$, and call $j_0$ the \emph{$\delta$-dimension} of $\pi$.   It is easy to see that $\pi$ is Cohen-Macaulay if and only if $\pi^{\vee}$ is Cohen-Macaulay as an $\F[\![K_1]\!]$-module in the sense of Definition \ref{def:CM}.

\begin{theorem}\label{thm:dual-block}
Let $\sB$ be a block of $G$ and $\pi\in\sB$. Then $\mathrm{S}^j_G(\pi)\in \Mod_{G,\zeta^{-1}}^{\rm adm}(\F)^{\sB'}$ where \[\sB':=\sB\otimes \zeta^{-1}\circ\det.\]  
\end{theorem}

\begin{proof}
This follows from Proposition 5.4, Proposition 5.7 and Theorem 5.13 of \cite{Ko-dual}. See also \cite[Thm.~1.1]{Witthaus} when $\sB$ is of type (I).
\end{proof}

We  generalize Theorem \ref{thm:dual-block} to any object  in $\Mod_{G,\zeta}^{\rm adm}(\F)^{\sB}$. For this we first treat those  objects which are injective when restricted to $K/Z_1$.

\begin{lemma}\label{lem:inj-dual}
Let $0\neq \Omega\in \Mod_{G,\zeta}^{\rm adm}(\F)$  such that $\Omega|_{K}$ is  injective as a representation of $K/Z_1$. 
Then $\Omega$ is Cohen-Macaulay of $\delta$-dimension $3$ and $\mathrm{S}^3_G(\Omega)|_K$ is again  injective  as a representation of $K/Z_1$. 
\end{lemma}

\begin{proof}
The restriction of $\Omega$ to $K_1/Z_1$ is again injective, so it is isomorphic to a direct summand of $ C^0(K_1/Z_1,\F)^{\oplus n}$ for some $n\geq 1$, where $C^0(K_1/Z_1,\F)$ is the space of continuous $\F$-valued functions on $K_1/Z_1$. We claim that $C^0(K_1/Z_1,\F)$ is Cohen-Macaulay of $\delta$-dimension $3$. Dually we need to prove that $\F[\![K_1/Z_1]\!]$ has grade $1$ as an $\F[\![K_1]\!]$-module (cf.~\S\ref{sec:app} for the notion of grade). Since $Z_1\cong 1+p\Z_p$ is pro-cyclic (as $p>2$) and noting that $\F[\![K_1]\!]$ is a domain,  we obtain a short exact sequence 
\begin{equation}\label{eq:gamma-K1}0\ra \F[\![K_1]\!]\overset{\gamma-1}{\lra} \F[\![K_1]\!]\ra \F[\![K_1/Z_1]\!]\ra0\end{equation}
by choosing a topological generator $\gamma\in 1+p\Z_p$. The claim easily follows from this. Moreover, we deduce an isomorphism \begin{equation}\label{eq:K/Z1}\Ext^1_{\F[\![K_1]\!]}\big(\F[\![K_1/Z_1]\!],\F[\![K_1]\!]\big)\cong \F[\![K_1/Z_1]\!].  \end{equation}

The claim implies that $\Omega$ is Cohen-Macaulay of $\delta$-dimension $3$, and  $\mathrm{S}^3_G(\Omega)$ is injective when restricted to $K_1/Z_1$ by using \eqref{eq:K/Z1}. By \cite[Lem.~8.8(ii)]{ST03}, $\mathrm{S}^3_G(\Omega)$ is also injective as a representation of $K/Z_1$.
\end{proof}

\begin{proposition}\label{prop:sm-dual}
Let $\sB$ be a block of $G$. Let $0\neq \Omega\in \Mod_{G,\zeta}^{\rm adm}(\F)^{\sB}$  such that $\Omega|_{K}$ is  injective as a representation of $K/Z_1$. 
 Then   $\mathrm{S}^3_G(\Omega)\in \Mod_{G,\zeta^{-1}}^{\rm adm}(\F)^{\sB'} $ where $\sB':=\sB\otimes\zeta^{-1}\circ\det$.
\end{proposition}

We need some preparations to prove Proposition \ref{prop:sm-dual}. 
 
\begin{lemma}\label{lem:criterion-inj}
Let $\Omega\in \Mod^{\rm adm}_{G,\zeta}(\F)$. Then $\Omega$ is   injective as a representation of $K/Z_1$   if and only if $\soc_K\Omega\subsetneq \soc_K\Omega'$ for any extension\footnote{Here we only consider extensions with central character.} $\Omega'\in\Ext^1_G(\pi',\Omega)$  and any irreducible $\pi'\in \Mod_{G,\zeta}^{\rm sm}(\F)$.
\end{lemma}
\begin{proof}
The condition is clearly necessary. For the sufficiency, we note that by a construction of Breuil and Pa\v{s}k\=unas \cite[\S9]{BP}, there exists  $\widetilde{\Omega}\in\Mod_{G,\zeta}^{\rm adm}(\F)$ which is injective as a representation of $K/Z_1$,  together with a $G$-equivariant embedding $\Omega\hookrightarrow \widetilde{\Omega}$ such that  $\soc_K\Omega=\soc_K\widetilde{\Omega}$. If $\Omega$ is not injective, i.e. $\Omega\subsetneq\widetilde{\Omega}$, then by choosing an irreducible subrepresentation $\pi'$ of $\widetilde{\Omega}/\Omega$ (always possible as $\widetilde{\Omega}/\Omega$ is admissible), we obtain an extension class $\Omega'\in\Ext^1_G(\pi',\Omega)$ with $\soc_K\Omega=\soc_K\Omega'$,  contradicting with the assumption.
\end{proof}

Let $\Omega\in \Mod_{G,\zeta}^{\rm adm}(\F)$. Assume $\Omega\neq 0$ and let $\pi$ be an irreducible subrepresentation of $\Omega$. By Zorn's lemma, there exists a maximal subrepresentation $\Omega_{\pi}\subset\Omega$  such that $\soc_G\Omega_{\pi}=\pi$. By construction, we have
\begin{equation}\label{eq:soc-Omega}\soc_G\Omega=\soc_G\Omega_{\pi}\oplus \soc_G(\Omega/\Omega_{\pi}).\end{equation}

\begin{lemma}\label{lem:sub=inj}
With the above notation, if $\Omega$ is   injective as a representation of $K/Z_1$, then so are $\Omega_{\pi}$ and $\Omega/\Omega_{\pi}$. 
\end{lemma}

\begin{proof}
It suffices to prove the injectivity of  $\Omega_{\pi}$. To apply the criterion in Lemma \ref{lem:criterion-inj}, take an irreducible  $\pi'\in\Mod_{G,\zeta}^{\rm sm}(\F)$.   By \eqref{eq:soc-Omega} we have an injection
\[0\ra \Ext^1_G(\pi',\Omega_{\pi})\ra \Ext^1_G(\pi',\Omega).\]
We may assume $\Ext^1_{G}(\pi',\Omega_{\pi})\neq0$;  
choose such a non-zero extension class, say $\Omega_{\pi}'$. It induces 
a commutative diagram \[\xymatrix{0\ar[r]&\Omega_{\pi}\ar[r]\ar@{^(->}[d]&\Omega_{\pi}'\ar[r]\ar@{^(->}[d]&\pi'\ar[r]\ar^{=}[d]&0\\
0\ar[r]&\Omega\ar[r]&\Omega'\ar[r]&\pi'\ar[r]&0}\]
with $\Omega'$ being the pushout of $\Omega_{\pi}'$ and $\Omega$.   We have $\Omega/\Omega_{\pi}\cong \Omega'/\Omega_{\pi}'$ by the snake lemma.
 
Assume for a contradiction that $\soc_K\Omega_{\pi}=\soc_K\Omega_{\pi}'$. Then 
\[\begin{array}{rll}\dim_{\F}\soc_K\Omega'&\leq &\dim_{\F}\soc_K\Omega_{\pi}'+\dim_{\F}\soc_K(\Omega'/\Omega_{\pi}')\\
&=&\dim_{\F}\soc_K\Omega_{\pi}+\dim_{\F}\soc_K(\Omega/\Omega_{\pi})\\
&=&\dim_{\F}\soc_K\Omega
\end{array}\]
where the last equality follows from \eqref{eq:soc-Omega}. In particular, we get 
$\soc_K\Omega=\soc_{K}\Omega'$, 
which contradicts with Lemma \ref{lem:criterion-inj} as $\Omega$ is assumed to be injective. 
\end{proof}
 
\begin{proof}[Proof of Proposition \ref{prop:sm-dual}]
The proof is motivated by that of \cite[Thm.~1.1]{Witthaus}. First, we may assume that $\pi:=\soc_G\Omega$ is irreducible by Lemma \ref{lem:sub=inj} and Lemma \ref{lem:inj-dual}.  Then $\Omega^{\vee}$ is a quotient of $P_{\pi}$, where $P_{\pi} $ is as in \S\ref{ss:block}.   
 By \cite[Prop.~5.20]{HuJEMS}, there exists a commutative local subring  $A\subset \End_{G}(\Omega^{\vee})$, which is formally smooth of dimension $2$ over $\F$,  such that $\Omega^{\vee}$ is flat over $A$ and $\Omega^{\vee}\otimes_A\F$ has finite length. Write $A=\F[\![x_1,x_2]\!]$. Then $x_1,x_2$ form a regular sequence for $\Omega^{\vee}$. Lemma \ref{lem:inj-dual} implies that $
\Omega^{\vee}$ is a Cohen-Macaulay $\F[\![K_1]\!]$-module of grade $1$. Using the short exact sequence
\[0\ra \Omega^{\vee}\overset{x_1}{\ra} \Omega^{\vee}\ra \Omega^{\vee}/x_1\Omega^{\vee}\ra0,\]
one  shows that 
$\Omega^{\vee}/x_1\Omega^{\vee}$ is Cohen-Macaulay of grade $2$, see the proof of \cite[Lem.~A.15]{Gee-Newton}.  As a consequence,  
we obtain a short exact sequence
\begin{equation}\label{eq:x1}0\ra\EE^1(\Omega^{\vee})\overset{x_1}{\ra} \EE^1(\Omega^{\vee})\ra \EE^2(\Omega^{\vee}/x_1\Omega^{\vee})\ra0,\end{equation}
where   we have  written $\EE^i(-):=\Ext^i_{\F[\![K]\!]}(-,\F[\![K]\!])$.

We need to prove  $\EE^1(\Omega^{\vee})\in\mathfrak{C}_{G,\zeta^{-1}}(\F)^{\sB'}$. Using   \eqref{eq:x1},  it suffices to prove  $ \EE^2(\Omega^{\vee}/x_1\Omega^{\vee})\in\mathfrak{C}_{G,\zeta^{-1}}(\F)^{\sB'}$.  
Similarly, we are reduced to proving the statement for  $\EE^3\big(\Omega^{\vee}/(x_1,x_2)\big)$. Now $\Omega^{\vee}/(x_1,x_2)$ has finite length, so the result follows from Theorem \ref{thm:dual-block}.
\end{proof}

\begin{corollary}\label{cor:sm-dual}
Let $\sB$ be a block of $G$ and $\pi\in \Mod_{G,\zeta}^{\rm adm}(\F)^{\sB}$. Then  $\mathrm{S}^i_G(\pi)\in \Mod_{G,\zeta^{-1}}^{\rm adm}(\F)^{\sB'} $ for any $i\geq 0$,  where $\sB'=\sB\otimes\zeta^{-1}\circ\det$.
\end{corollary}
\begin{proof}
Recall that $\F[\![K/Z_1]\!]$ has finite global dimension by \cite[Thm.~2.36]{Ven} (as $K/Z_1$ has no $p$-torsion). Thus, we may find a minimal injective resolution  of finite length of $\pi|_K$  in the category $\Mod_{K,\zeta}^{\rm sm}(\F)$:
\[0\ra \pi\ra I^{\bullet}.\]   
Using a construction of Breuil and Pa\v{s}k\=unas \cite[\S9]{BP}, we may equip a compatible $G$-action on $I^{\bullet}$ so that the above resolution is $G$-equivariant. Moreover, the minimality of the resolution implies that $I^j\in \Mod_{G,\zeta}^{\rm adm}(\F)^{\sB}$ for each $j$. Applying the smooth duality functors $\mathrm{S}^i_G$ to the resolution, we obtain an $E_1$-spectral sequence computing $\mathrm{S}^i_G(\pi)$ in terms of $\mathrm{S}^i_G(I^j)$. Together with Lemma \ref{lem:inj-dual} and Proposition \ref{prop:sm-dual}, this implies the result.  
\end{proof}

\subsection{Scholze's functor}\label{sec:Scholze}

Recall that (in this special setting), to any  admissible smooth $\cO/\varpi^n$-representation $\pi$ of $G,$ Scholze \cite{Scholze} associates a Weil-equivariant sheaf $\cF_{\pi}$ on the \'etale site of the adic space $\bP^{1}_{\C_p}.$  For $i\geq 0,$ denote
\[ \cS^i(\pi):=H^i_{\mathrm{\acute{e}t}}(\bP^{1}_{\C_p}, \cF_{\pi}).\] It is proved in \cite[Thm.~1.1]{Scholze} that $\cS^i(\pi)$ carries a continuous action of $ G_{\Q_p} \times D^{\times}$ and is admissible smooth as a representation of $D^{\times}$.

\medskip

The following theorem collects some properties of $\cS^i$.

\begin{theorem}\label{thm:Scholze}
Let $\pi$  be an admissible smooth representation of $G$ over $\F.$

\begin{itemize}
\item[(i)] For $i> 2,$ $\cS^i(\pi) = 0.$   If $\pi|_K$ is an injective representation of $K$, then $\cS^2(\pi) = 0.$ 

\item[(ii)] The inclusion $\pi^{\SL_2(\Q_p)}\into \pi$ induces an isomorphism $\cS^0(\pi^{\SL_2(\Q_p)}) \simeq \cS^0(\pi).$ In particular, if $\pi^{\SL_2(\Q_p)} = 0$ then $\cS^0(\pi) = 0.$

\item[(iii)] If $\pi$ carries a central character, then the centre of $D^{\times}$ acts on $\cS^i(\pi)$ by the same character.  

\item[(iv)] If $\pi =\Ind_B^G\chi$, where $\chi:B\to \F^\times$ is a smooth character, then $\cS^2(\pi) = 0.$ 

\item[(v)] If  $\pi$ is supersingular and is generic, in the sense that the block $\{\pi\}$ is generic (cf.~Definition \ref{def:generic}), then $\cS^2(\pi) = 0.$ 
\end{itemize}
\end{theorem}

\begin{proof}
(i) and (ii) are special cases of \cite[Thm.~3.2,~Prop.~4.7]{Scholze}. (iii) is proved in \cite[Lem.~7.3]{DPS22}. (iv) is proved in  \cite[Thm.~4.6]{Ludwig}.  (v) is \cite[Thm.~1.2]{HW-JL1}.
\end{proof}

\begin{proposition}\label{prop:S1-injective}
Let $\sB$ be a block which is not of type (IV). Let $\Omega\in \Mod_{G,\zeta}^{\rm adm}(\F)^{\sB}$ and assume that $\Omega$ is injective as a representation of $K/Z_1$. Then  $\cS^1(\Omega)$ is   injective as a  reprersentation of $\cO_D^{\times}/Z_D^1$.
\end{proposition}

\begin{proof}
This is a consequence of \cite[Prop.~4.5]{H-M}.  Lemma \ref{lem:inj-dual}  implies that $\Omega$ is Cohen-Macaulay of $\delta$-dimension $3$, and $\cS^0(\Omega)=\cS^2(\Omega)=0$ by  (i), (ii) of Theorem \ref{thm:Scholze} (using that $\Omega^{\SL_2(\Q_p)}=0$ as $\sB$ is not of type (IV)), so the conditions of \cite[Prop.~4.5]{H-M} are satisfied with $c=1$. We deduce an isomorphism 
\begin{equation}\label{eq:H-M}\mathrm{S}_{D^{\times}}^i(\cS^1(\Omega))\cong \cS^{i-2}(\mathrm{S}^3_G(\Omega))\end{equation}
for any $i\geq 0$, where $\mathrm{S}^i_{D^{\times}}(-)$ denotes the $i$-th smooth duality functor for the group $D^{\times}$ (see \cite[Def.~3.12]{Ko-dual}). 

We prove that $\mathrm{S}_{D^{\times}}^i(\cS^1(\Omega))=0$ except when $i=3$. On the one hand, since  $\cS^1(\Omega)$ carries a central character by  Theorem \ref{thm:Scholze}(iii),  it has $\delta$-dimension $\leq 3$ and so  
\[\mathrm{S}_{D^{\times}}^4(\cS^1(\Omega))=0.\]
On the other hand, 
by Proposition \ref{prop:sm-dual} and the assumption on $\sB$, we have   $\cS^0(\mathrm{S}^3_G(\Omega))=0$, thus via \eqref{eq:H-M}
\[\mathrm{S}^i_{D^{\times}}(\cS^1(\Omega))=0\]
for any $i\leq 2$. 
 
The above discussion implies that $\cS^1(\Omega)$ is Cohen-Macaulay of $\delta$-dimension $3$ (as a representation of $D^{\times}$). Recall that $\cS^1(\Omega)$ carries a central character. Using an analog of \eqref{eq:gamma-K1} with respect to $\F[\![U_D^1]\!]$ and $\F[\![U_D^1/Z_D^1]\!]$, one shows that $\cS^1(\Omega)$ is again Cohen-Macaulay of $\delta$-dimension $3$ when viewed as a representation of $U_D^1/Z_D^1$. Since $\F[\![U_D^1/Z_D^1]\!]$ is an Auslander regular ring of global dimension $3$,  the Auslander-Buchsbaum formula (cf. \cite[Thm.~6.2]{Ven}) implies that $\cS^1(\Omega)$ is injective as a representation of $U_D^1/Z_D^1$, hence also injective as a representation of $\cO_D^{\times}/Z_D^1$ by \cite[Lem.~8.8(ii)]{ST03}. 
\end{proof}

\medskip

Next we recall a result of \cite{CDN22} which relates $\cS^1(\pi)$ with   the cohomology of the infinite-level perfectoid space associated to the Drinfeld tower. We introduce some notation. 
\begin{itemize}
\item[$\bullet$] $\{\sM_{n,\breve{\Q}_p}, n\geq 0\}$ denotes the Drinfeld tower  defined over $\breve{\Q}_p$, which is equipped with commuting actions of $W_{\Q_p}$ (the Weil group of $\Q_p$), $G$ and $D^{\times}$; $\sM_{n,\breve{\Q}_p}$ is finite \'etale over $\sM_{0,\breve{\Q}_p}$ with Galois group isomorphic to $\cO_D^{\times }/(1 + \varpi^n_D \cO_D)$.  
\item[$\bullet$] $\sM_{n}$ denotes the base change of $\sM_{n,\breve{\Q}_p}$ to $\C_p$. 
\item[$\bullet$] $\sM_{\infty} \sim \varprojlim_n\sM_n$ is the infinite level perfectoid space associated to the Drinfeld tower, which is isomorphic to the infinite level Lubin-Tate tower by Faltings' isomorphism (\cite[Thm.~E]{Scholze-Weinstein}). 
\item[$\bullet$]  $\sM_{n}^{p}$ (resp.~$\sM_{\infty}^p$) is the quotient of $\sM_{n}$ (resp.~$\sM_{\infty}$) by $\smatr{p}00p\in G$. There is a rigid analytic space $\sM_{n,\Q_p}^{p}$ over $\Q_p$ such that $\sM_{n}^{p} \cong \sM_{n,\Q_p}^{p}\times_{\Q_p} \C_p.$
\end{itemize}
\begin{theorem}\label{thm:CDN-S1}
Let $\pi$ be an admissible smooth representation of $G$ on an $\cO/\varpi^n$-module. Assume that  $\pi$ does not admit any subquotient   lying in  blocks of type (IV). Then  there is an isomorphism 
\[\cS^1(\pi)\cong \Hom_{G}^{\rm cont}\big(\pi^{\vee},H^1_{\et}(\sM_{\infty},\cO/\varpi^n)\big).\]
If $p^{\Z}$ acts trivially on $\pi$, then 
\[\cS^1(\pi)\cong \Hom_G^{\rm cont}\big(\pi^{\vee},H^1_{\et}(\sM_{\infty}^{p},\cO/\varpi^n)\big).\]
\end{theorem}
\begin{proof}
The first isomorphism  follows from the same proof as \cite[Cor.~3.13]{CDN22} (via \cite[Thm.~3.2]{CDN22}). Remark that the  irreducibility assumption on $\pi$  is not necessary in \emph{loc.~cit.} and Fust's result  \cite{Fust} applies to any smooth representation of $G$.  To deduce the second isomorphism, we argue as  Step $3$ of the proof of \cite[Lem.~4.12]{CDN22}. 
\end{proof}
\medskip
 
We will also need to extend Scholze's functor to $\mathrm{Ban}^{\rm adm}_{G}(E)$, the category of admissible unitary Banach space representations of $G.$ Following \cite[\S3.3]{Paskunas-JL}, we define a covariant homological $\delta$-functor $\{\check{\cS}^i\}_{i\geq 0}$ on $\mathfrak{C}_G(\cO)$ by 
\[\check{\cS}^i:\mathfrak{C}_{G}(\cO)\ra \mathfrak{C}_{G_{\Q_p}\times D^{\times}}(\cO),\ \ M\mapsto H^i_{\et}(\mathbb{P}^1_{\C_p},\mathcal{F}_{M^{\vee}})^{\vee}.\]
Given $\Pi\in\mathrm{Ban}^{\rm adm}_{G}(E)$, we choose an open bounded $G$-invariant lattice $\Theta$ in $\Pi$. Its Schikhof dual $\Theta^d:=\Hom_{\cO}^{\rm cont}(\Theta,\cO)$ lies in $\mathfrak{C}_G(\cO)$. We define
\[\check{\cS}^i(\Pi):=\check{\cS}^i(\Theta^d)^d\otimes_{\cO}E,\]
which   does not depend on the choice of $\Theta$. By \cite[Lem.~3.4]{Paskunas-JL}, $\check{\cS}^i(\Pi)$ is an admissible unitary  Banach space representation of $D^{\times}$ with an open bounded lattice given by $\check{\cS}^i(\Theta^d)^d$. 

\begin{proposition}\label{prop:S1-Theta}
Assume that $\cS^0(\Theta/\varpi\Theta)=\cS^2(\Theta/\varpi\Theta)=0$.  Then there is an isomorphism
\begin{equation}\label{eq:S1check}
\check{\cS}^1(\Theta^d)^d\cong \varprojlim_{n\geq 1}\cS^1(\Theta/\varpi^n\Theta)
\end{equation}
and the sequence 
\begin{equation}\label{eq:S1exact-seq}
0\ra \check{\cS}^1(\Theta^d)^d\overset{\varpi}{\lra} \check{\cS}^1(\Theta^d)^d\ra \cS^1(\Theta/\varpi\Theta)\ra0
\end{equation}
is exact. 
\end{proposition}

\begin{proof}
The isomorphism \eqref{eq:S1check} follows from \cite[Lem.~3.3]{Paskunas-JL} and the vanishing of $\cS^0(\Theta/\varpi\Theta).$ Consider the short exact sequence $0\to \Theta/\varpi^n \Theta \To{\varpi} \Theta/\varpi^{n+1} \Theta \to \Theta/\varpi \Theta \to 0. $ The vanishing of $\cS^0(\Theta/\varpi\Theta)$ and $\cS^2(\Theta/\varpi\Theta)$ implies that the sequence
\begin{equation}\label{eq:exactseq-S1-torsion}
0\to \cS^1( \Theta/\varpi^n \Theta) \To{\varpi} \cS^1(\Theta/\varpi^{n+1} \Theta) \to \cS^1(\Theta/\varpi \Theta )\to 0
\end{equation}
is exact and the inverse system $\{\cS^1( \Theta/\varpi^n \Theta)\}_{n\geq 1}$ is Mittag-Leffler. We deduce \eqref{eq:S1exact-seq} by taking limit of \eqref{eq:exactseq-S1-torsion}.
\end{proof}

Under the assumption of Proposition \ref{prop:S1-Theta}, it is more convenient to put
\[\cS^1(\Theta):=\varprojlim_{n\geq 1}\cS^1(\Theta/\varpi^n\Theta ) \cong\check{\cS}^{1}(\Theta^d)^d.\]

\begin{lemma}\label{lem:S1-Theta}
Assume that $\cS^0(\Theta/\varpi\Theta)=\cS^2(\Theta/\varpi\Theta)=0$ and $\Theta/\varpi\Theta$ does not admit any subquotient lying in blocks of type (IV). Assume further that $p^{\Z}$ acts trivially on $\Theta$. Then there is an isomorphism
\[\cS^1(\Theta)\cong \varprojlim_{n\geq 1}\Hom_G^{\rm cont}\big(\Theta^d,H^1_{\et}(\sM_{\infty}^p,\cO/\varpi^n)\big).\]
\end{lemma}

\begin{proof}
This follows from Propostion \ref{prop:S1-Theta}, Theorem \ref{thm:CDN-S1} and  the isomorphism  $\Theta^d/\varpi^n\Theta^d\cong (\Theta/\varpi^n\Theta)^{\vee}.$
\end{proof}

\subsection{The multiplicity one theorem}
Let $\brho:G_{\Q_p}\ra \GL_2(\F) $ be a continuous representation and let $\sB$ be the block of $G$ corresponding to $\brho^{\rm ss}$. We can associate to $\brho$ a set of characters of $\cO_D^{\times}$, denoted by $W_D(\brho)$ and called \emph{quaternionic Serre weights} of $\brho,$ see \cite{Gee-Geraghty} or \cite[\S6.1]{HW-JL1} for more details. 
 
The main result of this section is the following multiplicity one result.
 
\begin{theorem}\label{thm:multione}
Assume that $\sB$ is of type (I) or (II), and is generic if $\sB$ is of type (I).   
Then\footnote{Here $\brho(-1)$ denotes the  twist of $\brho$ by $\omega^{-1}$.}  
$\cS^1(\pi(\brho))\cong \brho(-1)\otimes \JL(\brho)$
 for some admissible smooth representation $\JL(\brho)$ of $D^{\times}$ satisfying 
\[\soc_{\cO_D^{\times}}\JL(\brho)\cong \bigoplus_{\chi\in W_D(\brho)}\chi.\]
 In particular, the $\cO_D^{\times}$-socle of $\JL(\brho)$ is multiplicity free.
 \end{theorem}

\begin{remark}
(i) Recently, Andrea Dotto proved an analog of \eqref{eq:multione} in a more general setting, see \cite[Thm.~1.2]{Dotto22}.  

(ii) If $\sB$ is of type (III),  the situation seems to be more complicated, and we don't expect $\soc_{\cO_D^{\times}}\JL(\brho)$ to be multiplicity free. If $\sB$ is of type (IV), then $\cS^1(\pi(\brho))$ is not always $\brho(-1)$-isotypic, and the definition of $\JL(\brho)$ is  subtler, see \cite[\S8.3]{HW-JL1} for a detailed study of this case. 
\end{remark}
\medskip

We can determine the structure of $\gr_{\fm_D}(\JL(\brho)^{\vee})$ by combining Theorem \ref{thm:multione} and Theorem \ref{thm:mu=4}.  
The following definition is a reformulation of \cite[Def.~6.13]{HW-JL1}.
\begin{definition}\label{def:fa}
For each $\chi\in W_D(\brho)$, we define a graded ideal $\mathfrak{a}(\chi)$ of $A=\F[y,z]/(yz)$ (see \S\ref{sec:CA}) as follows. 

\begin{itemize}
\item If $\chi\alpha^{-1}\in W_D(\brho)$, then $\mathfrak{a}(\chi):=\p_0=(y)$; if $\chi\alpha\in W_D(\brho)$, then $\mathfrak{a}(\chi):=\p_1=(z)$. \vspace{1mm}
\item If neither of $\chi\alpha$, $\chi\alpha^{-1}$ lies in $W_D(\brho)$, then $\mathfrak{a}(\chi):=(0)$.
\end{itemize}
 \end{definition}

The following result improves \cite[Thm.~6.14]{HW-JL1}. Remark that, contrary to Theorem \ref{thm:multione}, we need to further assume $\sB$ is \emph{generic} when it is of type (II), because this condition is needed to apply \cite[Thm.~6.14]{HW-JL1}.
\begin{theorem}\label{thm:gr-pi}
 Assume that $\sB$ is of type of (I) or (II), and  is generic. 
Then there is an isomorphism of graded modules \begin{equation}\label{eq:gr-str}\bigoplus_{\chi\in W_D(\brho)}\chi^{\vee}\otimes A/\mathfrak{a}(\chi)\cong \gr_{\fm_D}(\JL(\brho)^{\vee}).\end{equation}
 In particular,  $\mu(\JL(\brho))=4$. 
\end{theorem}
 \begin{proof}
Since $\cS^1(\pi(\brho))$ is infinite dimensional by \cite[Thm.~7.8]{Scholze},  so is $\JL(\brho)$, thus  $\mu(\JL(\brho))\geq 4$ by Theorem \ref{thm:mu=4}.  

 Remark that when $\sB$ is of type (II), we have $\JL(\brho)=\JL(\brho^{\rm ss})$ by \cite[Thm.~8.13]{HW-JL1} and also $W_D(\brho)=W_D(\brho^{\rm ss})$ by \cite[Thm.~6.3]{HW-JL1}, so it suffices to prove the result for $\brho$ non-split.  Since $\sB$ is assumed to be generic, we may further assume $\brho $ is generic, cf.~Definition \ref{def:generic}(ii).

Let $N$ denote the left-hand side of \eqref{eq:gr-str}.  It is direct to check that $\mu(N)=4$ using the explicit description of $W_D(\brho)$ (see \cite[Thm.~6.3]{HW-JL1}). By \cite[Thm.~6.14]{HW-JL1}, there is   a surjective morphism $\iota: N\twoheadrightarrow \gr_{\fm_D}(\JL(\brho)^{\vee})$; here we may take $m=1$ in \emph{loc. cit.} as $\JL(\brho)^{U_D^1}$ is multiplicity free by Theorem \ref{thm:multione}.  Hence, we get $\mu(\JL(\brho))=4$ by the above discussion.  Since $N$ is Cohen-Macaulay and   $\mu(N)=\mu(\gr_{\fm_D}(\JL(\brho)^{\vee}))$, $\iota$ must be  injective.
\end{proof} 

\medskip

Now we turn to the proof of Theorem \ref{thm:multione}. Actually we will prove a slightly stronger statement which, combined with Theorem \ref{thm:Scholze}(v), implies Theorem \ref{thm:multione}.
\begin{theorem}\label{thm:multione-bis}
Assume that $\sB$ is of type (I) or (II). If $\sB$ is of type (I), we assume that $\cS^2(\pi(\brho))=0$. Then  \begin{equation}\label{eq:rho-isotypic}\cS^1(\pi(\brho))\cong \brho(-1)\otimes \JL(\brho)\end{equation} for some admissible smooth representation $\JL(\brho)$ of $D^{\times}$ satisfying  
 \begin{equation}\label{eq:multione}\soc_{\cO_D^{\times}}\JL(\brho)\cong \bigoplus_{\chi\in W_D(\brho)}\chi.\end{equation}
\end{theorem}

The isomorphism \eqref{eq:rho-isotypic} can be proved as in  \cite[Prop.~8.35(i)]{HW-JL1}. Indeed, the case when $\sB$ is of type (II) is already treated there. When $\sB$ is of type (I) and  assuming $\cS^2(\pi(\brho))=0$, a similar argument as in \cite[Prop.~7.4]{HW-JL1}, using $P_{\pi(\brho)}$ instead of $N_{\infty}$ in \emph{loc. cit.}, shows that  $\cS^1(\pi(\brho))$ is Cohen-Macaulay of $\delta$-dimension $1$ as a representation of $D^{\times}$.  We   then conclude as in \cite[Prop.~8.35]{HW-JL1}   via the local-global compatibility \`a la Emerton, see \cite[Thm.~7.6]{HW-JL1}.

For the socle of $\JL(\brho),$ we note that 
\[\oplus_{\chi\in W_D(\brho)}\chi\subseteq \soc_{\cO_D^{\times}}\JL(\brho)=\JL(\brho)^{U_D^1}.\]
Thus, to prove \eqref{eq:multione} it suffices to prove that 
$\dim_{\F}\JL(\brho)^{U_D^1}\leq  |W_D(\brho)|$. Using \eqref{eq:rho-isotypic}, we are reduced to proving the following result.
 
\begin{proposition}\label{prop:d0<2}
Keep the assumptions in Theorem \ref{thm:multione-bis}. Then 
 \begin{equation}\label{eq:inequality} \dim_{\F}\cS^1(\pi(\brho))^{U_D^1}\leq 2  |W_D(\brho)|.\end{equation}
\end{proposition}

We will prove Proposition \ref{prop:d0<2}, hence Theorem \ref{thm:multione-bis},   in the next subsection.

\subsection{Proof of Proposition \ref{prop:d0<2}}
Arguing  as in the second paragraph of the proof of Theorem \ref{thm:gr-pi}, we may and do assume $\End_{G_{\Q_p}}(\brho)=\F$; this excludes the case $\brho$ is reducible split. Then  $\pi(\brho)$ is  either irreducible (when $\brho$ is irreducible), or a non-split extension of the   two principal series in $\sB$ (when $\brho$ is reducible non-split). We write
\[P_{\brho}:=P_{\soc_G\pi(\brho)}.\]
Since $\pi(\brho)$ is non-split, $P_{\brho}$ is also a projective envelope of $\pi(\brho)^{\vee}$ in $\mathfrak{C}_{G,\zeta}(\F).$ By Lemma \ref{lem:P/m},
\begin{equation}\label{eq:Prho/m}(P_{\brho}\otimes_{R_{\sB}^{\rm ps,\delta}}\F)^{\vee}\cong \pi(\brho). \end{equation}
Recall the element $x\in R_{\sB}^{\rm ps,\delta}$ constructed in Proposition \ref{prop:x}. Let \[\Omega_{\brho}:= (P_{\brho}/(x,\varpi)P_{\brho})^{\vee} \] 
which is admissible  and injective when restricted to $K/Z_1$. \medskip

We divide the proof of \eqref{eq:inequality}  into several steps. Precisely,  below we prove   the following relations
\[\dim_{\F}\cS^1(\pi(\brho))^{U_D^1}\leq d_1=r_1=r_2=2 |W_D(\brho)|, \]
where 
\begin{itemize}
\item[$\bullet$] $d_1:=\dim_{\F}\cS^1(\Omega_{\brho})^{U_D^1}$, 
\item[$\bullet$] $r_1:=\mathrm{rank}_{\cO}\cS^1(\Theta_{\brho})^{U_D^1}$, where $\Theta_{\brho}:=(P_{\brho}/x)^d$,
\item[$\bullet$] $r_2:=\mathrm{rank}_{\cO}\Hom_{G}^{\rm cont}\big(P_{\brho}/x, H^1_{\et}(\mathscr{M}_{1}^p,\cO)\big)$.
\end{itemize}
Remark that $\cS^1(\Theta_{\brho})$ is $\cO$-flat by Proposition \ref{prop:S1-Theta}, and $\Hom_{G}^{\rm cont}\big(P_{\brho}/x, H^1_{\et}(\mathscr{M}_{1}^p,\cO)\big)$ is also $\cO$-flat as $H^1_{\et}(\mathscr{M}_{1}^p,\cO)$ is.

Up to twist, we may and do assume that $p^{\Z}$ acts trivially on the objects of $\sB$. 

\subsubsection{Proof of $\dim_{\F}\cS^1(\pi(\brho))^{U_D^1}\leq d_1$} 

By definition of $\Omega_{\brho}$ and \eqref{eq:Prho/m}, there is a $G$-equivariant inclusion $\pi(\brho)\hookrightarrow \Omega_{\brho}$. Since $\sB$ is not of type (IV), by Theorem \ref{thm:Scholze}(ii) we have $\cS^0(\Omega_{\brho}/\pi(\brho))=0$, 
hence an embedding \[\cS^1(\pi(\brho))\hookrightarrow \cS^1(\Omega_{\brho}). \]The result follows by taking $U_D^1$-invariants.
 
\subsubsection{Proof of $d_1=r_1$}

Since $\Omega_{\brho}$ is injective as a representation of $K/Z_1$,  we deduce from Proposition \ref{prop:S1-injective} that $\cS^1(\Omega_{\brho})$ is    injective as a representation of $\cO_D^{\times}/Z_{D}^1$.  On the other hand,  we note $\Theta_{\brho}/\varpi\Theta_{\brho}\cong\Omega_{\brho}$, and by Proposition \ref{prop:S1-Theta} we have  
\[\cS^1(\Theta_{\brho})\otimes_{\cO}\F\cong \cS^1(\Omega_{\brho}).\]
Taking $U_D^1$-invariants, we obtain the assertion $d_1=r_1$ by the following lemma.

\begin{lemma}\label{lem:dim=rank}
Let $H$ be a compact $p$-adic analytic group. Let $M$ be a topological  $\cO[H]$-module. Assume that $M$ is $\cO$-flat and $M\otimes_{\cO}\F$ is admissible and injective as a smooth representation of $H$. Then for any open subgroup $H_1<H$, we have 
\begin{equation}\label{eq:dim=rank}
\dim_{E}(M\otimes_{\cO}E)^{H_1}=\mathrm{rank}_{\cO}M^{H_1}=\dim_{\F}(M\otimes_{\cO}\F)^{H_1}.\end{equation}
\end{lemma}

\begin{proof}
The admissibility of $M\otimes_{\cO}\F$ implies that the quantities in \eqref{eq:dim=rank} are all finite. The first equality in \eqref{eq:dim=rank} is clear by the $\cO$-flatness of $M$. 

The assumption implies that $H^1(H_1,M\otimes_{\cO}\F)=0$. Hence,
the sequence $0\ra M\overset{\varpi}{\ra}M\ra M\otimes_{\cO}\F\ra0$ induces an exact sequence  
\[0\ra M^{H_1} \ra M^{H_1}\ra (M\otimes_{\cO}\F)^{H_1}\ra H^1_{\rm cont}(H_1,M) \ra H^1_{\rm cont}(H_1,M)\ra 0.\]
It is easy to see that $H^1_{\rm cont}(H_1,M)$ is a finite $\cO$-module, using for example \cite[Lem.~3.13]{Scholze}. By Nakayama's lemma we deduce $H^1_{\rm cont}(H_1,M)=0$ and the second equality follows.
\end{proof}

\subsubsection{Proof of  $r_1=r_2$} 
It is proved in Step 2 of the proof of \cite[Lem.~4.12]{CDN22} that there is a  convergent spectral sequence\footnote{The result in \emph{loc.~cit.} is stated for the cohomology with $\F$-coefficients, but it clearly remains true for torsion coefficients.}
\[E_{2}^{i,j}=H^i_{\rm cont}(U_D^1, H^j_{\et}(\mathscr{M}_{\infty}^p, \cO/\varpi^n)\Rightarrow H^{i+j}_{\et}(\mathscr{M}_1^{p},\cO/\varpi^n).\]
In particular, we obtain an exact sequence
\begin{equation}\label{eq:E2-short}0\ra E_{2}^{1,0}\ra H^1_{\et}(\sM_{1}^p,\cO/\varpi^n)\ra H^1_{\et}(\sM_{\infty}^p,\cO/\varpi^n)^{U_D^1}\ra E_{2}^{2,0}.\end{equation}
Since the action of $G$ on $H^0_{\et}(\sM_{\infty}^p,\cO/\varpi^n)$ factors through the determinant, the same holds for $E_{2}^{1,0}$ and $E_{2}^{2,0}$. Since $\sB$ is not of type (IV) by assumption, applying $\Hom_G^{\rm cont}(P_{\brho}/x,-)$ to \eqref{eq:E2-short} we obtain an isomorphism 
\[\Hom_G^{\rm cont}\big(P_{\brho}/x,H^1_{\et}(\sM_1^p,\cO/\varpi^n)\big)\simto \Hom_G^{\rm cont}\big(P_{\brho}/x,H^1_{\et}\big(\sM_{\infty}^p,\cO/\varpi^n)^{U_D^1}\big).\]
Taking inverse limit and noting that  $\Hom_{G}^{\rm cont}(P_{\brho}/x,-)$  and $H^0_{\rm cont}(U_D^1,-)$ commute with taking inverse limit, we deduce by Lemma \ref{lem:S1-Theta}
\[\begin{array}{rll}\Hom_G^{\rm cont}\big(P_{\brho}/x,H^1_{\et}(\sM_1^p,\cO)\big)&\cong& \varprojlim_n \Hom_G^{\rm cont}\big(P_{\brho}/x , H^1_{\et}\big(\sM_{\infty}^p,\cO/\varpi^n)^{U_D^1}\big)\\
&\cong& \left(\varprojlim_n\Hom_G^{\rm cont}\big(P_{\brho}/x,H^1_{\et}(\sM_{\infty}^p,\cO/\varpi^n)\big)\right)^{U_D^1}\\
&\cong&\cS^1(\Theta_{\brho})^{U_D^1}\end{array}\]
 and so $r_2=r_1$.

\subsubsection{Proof of $r_2=2 |W_D(\brho)|$}

We first note that by \cite[Prop.~5.16]{CDN22} and the proof of \cite[Thm.~5.24]{CDN22}, the natural map 
\[\Hom^{\rm cont}_G\big(P_{\brho}/x,H^1_{\et}(\sM_{1,\bQp}^p,\cO)\big)\hookrightarrow \Hom_G^{\rm cont}\big(P_{\brho}/x,H^1_{\et}(\sM_1^p,\cO)\big)\]
is a bijection, where by definition
\[H^1_{\et}\big(\sM_{1,\bQp}^{p},\cO\big):=\varprojlim_n\big(\varinjlim_{[L:\Q_p]<\infty}H^i_{\et}(\sM_{1,\Q_p}^p\times_{\Q_p} L,\cO/\varpi^n)\big).\]
Therefore, we have
\[r_2=\mathrm{rank}_{\cO}\Hom^{\rm cont}_G\big(P_{\brho}/x,H^1_{\et}(\sM_{1,\bQp}^p,\cO)\big).\]
Next, letting $H^1_{\et}\big(\sM_{1,\bQp}^p,E\big):=H^1_{\et}\big(\sM_{1,\bQp}^p,\cO\big)\otimes_{\cO}E$, it is clear that
\[r_2=\dim_{E}\Hom_{G}^{\rm cont}\big(P_{\brho}/x,H^1_{\et}(\sM_{1,\bQp}^p,E)\big).\]
We may further replace $H^1_{\et}(\sM_{1,\bQp}^p,E)$ by $H^1_{\et}(\sM_{1,\bQp}^p,E(1))$ (where $(1)$ denotes the Tate twist) in the above formula as it only changes the action of $G_{\Q_p}$. The reason to do this is that we have the following structure theorem on $H^1_{\et}\big(\sM_{1,\bQp}^p,E(1)\big)$, which is a special case of \cite[Thm.~0.1]{CDN22}.

\begin{theorem}\label{thm:CDN-main}
There is a topological $E[G_{\Q_p}\times G\times D^{\times}]$-isomorphism
\[H^1_{\et}\big(\mathscr{M}_{1,\bQp}^{p},E(1)\big)\cong \bigoplus_{M}\big(\widehat{\oplus}_{\mathscr{B}}\Pi^{*}(\rho_{\mathscr{B},M})\otimes \rho_{\mathscr{B},M}\otimes\check{R}_{\mathscr{B},M}\big)\otimes \mathrm{JL}(M)\]
where $\sB$ runs over all blocks of $G$, $M$ runs over  all  types of level $\leq 1$, and\footnote{If $M$ is special, then the definitions of $R_{\sB,M}$ and $\rho_{\sB,M}$ need to be modified, but this case is excluded in our application.}
\begin{itemize}
\item[$\bullet$] $R_{\sB,M}^+$ is Kisin's ring   parametrizing pseudo-deformations of $\mathrm{tr}(\brho)$ which are potentially semi-stable of Hodge-Tate weights\footnote{Our convention is that the Hodge-Tate weight of $\varepsilon$ is $1$.} $(0,1)$ and of type $M$;
\item   $R_{\sB,M}:=R^+_{\sB,M}[1/p]$ and $\check{R}_{\mathscr{B},M}$ is the $E$-dual of $R_{\sB,M}$; 
\item $\rho_{\sB,M}$ is the universal pseudo-deformation over $R_{\sB,M}$;
\item $\rho_{y,M}$ is the specialization of $\rho_{\sB,M}$ at any maximal ideal $y\in\mathrm{Spm}(R_{\sB,M})$;
\item $\Pi^*(\rho_{\sB,M})$ is an $R_{\sB,M}[G]$-module which specializes to $\Pi(\rho_{y,M})^*$ for $y\in\mathrm{Spm}(R_{\sB,M})$, where $\Pi(\rho_{y,M})$ corresponds to $\rho_{y,M}$ via the $p$-adic Langlands correspondence for $G$, see \cite[Def.~5.14]{CDN22};
\item $\JL(M)$ is the finite dimensional irreducible  smooth $E$-representation of $D^{\times}$ associated to $M$ via the (classical) local Jacquet-Langlands correspondence.
\end{itemize}  
\end{theorem} 

Recall that we have assumed  $\End_{G_{\Q_p}}(\brho)=\F$. So the usual universal deformation ring (\`a la Mazur)  of $\brho$ with determinant $\delta$ exists, which we denote by $R_{\brho}^{\delta}$. Let $R_{\brho,M}^+$ be  the potentially semi-stable deformation ring of $\brho$ of Hodge-Tate weights $(0,1)$ and of type $M$. Passing to the trace of a deformation induces  natural morphisms \begin{equation}\label{eq:ps-def}R_{\sB}^{\rm ps,\delta}\ra R_{\brho}^{\delta},\ \ \ R_{\sB,M}^+\ra R_{\brho,M}^+.\end{equation}

\begin{lemma}\label{lem:Kisin}
 Assume that $M$ is a cuspidal type. Then the natural maps  in \eqref{eq:ps-def} are isomorphisms.
\end{lemma}
\begin{proof}
It is proved in \cite[Cor.~1.4.4]{Kisin-FM} that the map $R_{\sB}^{\rm ps,\delta}\ra R_{\brho}^{\delta}$ is an isomorphism. To obtain the second isomorphism, it  suffices to note that any deformation of $\brho$ of cuspidal type is irreducible and that an irreducible deformation of $\brho$ is uniquely determined by its trace. Compare the proof of \cite[Cor.~1.7.2]{Kisin-FM}.
\end{proof}

\begin{lemma}\label{lem:types}
Assume that $\sB$ is a block of type (I) or (II). Then the number of   cuspidal types $M$ of level $1$  such that $R_{\sB,M}^+\neq 0$ is exactly $\frac{1}{2}|W_D(\brho)|$. Moreover, for   such a type $M$, $R_{\sB,M}^+\cong \cO[\![\bar{x}]\!]$ where $\bar{x}$ denotes  the image of $x\in R_{\sB}^{\rm ps,\delta}$ in $R_{\sB,M}^+$.  
\end{lemma}

\begin{proof}
First, we recall that a cuspidal type of level $1$ always has dimension $p-1$, and its mod $\varpi$ reduction (by choosing a $K$-stable $\cO$-lattice) has semisimplification
\begin{equation}\label{eq:modp-red}(\Sym^{k}\F^2\otimes{\det}^a)\oplus (\Sym^{p-3-k}\F^2\otimes{\det}^{k+1+a}),\end{equation}
for suitable $0\leq k\leq p-2$ and $0\leq a \leq p-2$, where $\Sym^{-1}\F^2:=0$ by convention. Conversely, given a pair $(k,a)$ as above, there exists a unique cuspidal type of level $1$ whose mod $\varpi$ reduction is given by \eqref{eq:modp-red}.

Assume that $\sB$ is of type (I) and write \[\brho|_{I_{\Q_p}}=\matr{\omega_2^{r+1}}00{\omega_2^{p(r+1)}}\otimes\omega^a\] with $0\leq r\leq p-1$ and $0\leq a \leq p-2$. Then  we may replace $R_{\sB,M}^+$ by $R_{\brho,M}^+$ via  Lemma \ref{lem:Kisin}. The Breuil-M\'ezard conjecture,  proved in \cite{Kisin-FM} (see also \cite{Pa15}),  implies that $R_{\brho,M}^{+}\neq 0$ if and only if 
\[\JH(\overline{M}^{\rm ss})\cap W(\brho)\neq \emptyset,\]
and in this case  the Hilbert-Samuel multiplicity  of $R_{\brho,M}^+/\varpi$ is equal to the cardinality of $\JH(\overline{M}^{\rm ss})\cap W(\brho)$. Here, $\overline{M}^{\rm ss}$ is the semisimplication of the mod $\varpi$ reduction of an $\cO$-lattice in $M$ and $W(\brho)$ denotes the set of Serre weights of $\brho$ as in \cite{BDJ}. Explicitly, 
\[W(\brho)=\big\{\Sym^r\F^2\otimes{\det}^{a},\Sym^{p-1-r}\F^2\otimes{\det}^{r+a}\big\}.\]
Using the structure of $\overline{M}^{\rm ss}$ recalled above, one checks that there are two (resp.~one) such types $M$ with $\JH(\overline{M}^{\rm ss})\cap W(\brho)\neq\emptyset$ if $r\notin\{0,p-1\}$ (resp.~if $r\in\{0,p-1\}$). Moreover, in each case the set $\JH(\overline{M}^{\rm ss})\cap W(\brho)$ has cardinality $1$. 
Combined with the explicit description of $W_D(\brho)$ due to Khare \cite{Khare} (see also \cite[Thm.~6.3]{HW-JL1}), we easily deduce the first assertion and also that $R_{\brho,M}^+$ is   formally smooth over $\cO$ of relative dimension $1$. The last assertion follows from the property of $x$, see \cite[Thm.~3.3(iii)]{HP}.

The case when $\sB$ is of type (II) is proved in a similar way. We remark that since $\brho$ is assumed to be non-split, the set $|\JH(\overline{M}^{\rm ss})\cap W(\brho)|$ always  has cardinality $1$. 
\end{proof}
 
\begin{corollary}\label{cor:HP}
Keep the notation in Lemma \ref{lem:types}. Let $M$ be a cuspidal type of level $1$ such that $R_{\sB,M}^+\neq 0$ and let $J_{\sB,M}$ denote the kernel of $R_{\sB}^{\rm ps,\delta}\ra R_{\sB,M}^+$. Then $x\notin J_{\sB,M}$ and \[\fm_{x,M}:=(x)+J_{\sB,M}\subset R_{\sB}^{\rm ps,\delta}\] extends to  a maximal ideal of $R_{\sB}^{\rm ps,\delta}[1/p]$, i.e. defines a point of $\mathrm{Spm}(R_{\sB}^{\rm ps,\delta}[1/p])$. 
\end{corollary}
\begin{proof}
The result follows from Lemma \ref{lem:types}.  
\end{proof}

Now we prove $r_2=2 |W_D(\brho)|$.  We have seen that  \[r_2=\dim_{E}\Hom_{G}^{\rm cont}\big(P_{\brho}/x,H^1_{\et}(\sM_{1,\bQp}^p,E(1))\big).\]
  Using Theorem \ref{thm:CDN-main} and  Lemma \ref{lem:types}, it suffices to prove that for each cuspidal type $M$ of level $1$ such that $R_{\sB,M}\neq0$,  
\[\dim_E\Hom_{G}^{\rm cont}\big(P_{\brho}/x, H^1_{\et}(\sM_{1,\bQp}^p)[M]\otimes \JL(M)\big)=4,\]
where $H^1_{\et}(\sM_{1,\bQp}^p)[M]:=\Pi^{*}(\rho_{\mathscr{B},M})\otimes \rho_{\mathscr{B},M}\otimes\check{R}_{\mathscr{B},M}$. 
Since $\dim_E\mathrm{JL}(M)=2$, this again reduces to proving that
\[ \Hom_{G}^{\rm cont}\big(P_{\brho}/x, H^1_{\et}(\sM_{1,\bQp}^p)[M]\big)\]
has dimension $2$. 
By \cite[Thm.~5.24(i)]{CDN22}, this space is annihilated by $J_{\sB,M}$ (defined in Corollary \ref{cor:HP}),  hence  is equal to 
\[\Hom_{G}^{\rm cont}\big(P_{\brho}/\fm_{x,M},H^1_{\et}(\sM_{1,\bQp}^p)[M]\big)\]
where $\fm_{x,M}:=(x)+J_{\sB,M}$. By Corollary \ref{cor:HP}, $\fm_{x,M}$ defines a point in $\mathrm{Spm}(R_{\sB}^{\rm ps,\delta}[1/p])$, which actually lies in the closed subspace $\mathrm{Spm}(R_{\sB,M})$. It then  follows  from the proof of \cite[Thm.~5.24(ii)]{CDN22} that 
\[\Hom_G^{\rm cont}\big(P_{\brho}/\fm_{x,M},H^1_{\et}(\sM_{1,\bQp}^p)[M]\big)\cong \Hom_G^{\rm cont}\big(P_{\brho},\Pi(\rho_{x,M})^*\big) \otimes \rho_{x,M}\]
which has dimension $2$  because 
\[\dim_E\Hom_G^{\rm cont}\big(P_{\brho},\Pi(\rho_{x,M})^*\big)=\dim_{\F}\Hom_G^{\rm cont}\big(P_{\brho},\pi(\brho)^{\vee}\big)=1,\]
where the first equality holds by the projectivity of $P_{\brho}$.  

\medskip 
 
This finishes the proof of Proposition \ref{prop:d0<2}, hence  of Theorem \ref{thm:multione-bis}.   
 
\section{The main result} \label{sec:main}

If $\rho:G_{\Q_p}\ra \GL_2(E)$ is a $p$-adic Galois  representation, 
we let $\Pi(\rho)\in \mathrm{Ban}_G^{\rm adm}(E)$ denote the  Banach space representation of $G=\GL_2(\Q_p)$ associated to $\rho$ by the $p$-adic local Langlands correspondence (\cite{Co}),  normalized in such a way that the central character of $\Pi(\rho)$ is equal to $\det(\rho)\varepsilon^{-1}$.

Recall that  Scholze's functor extends to the category $\mathrm{Ban}^{\rm adm}_G(E)$, see \S\ref{sec:Scholze}. Although $\check{\cS}^1(\Pi(\rho))$ is residually of infinite length, we expect that $\check{\cS}^1(\Pi(\rho))$ itself is topologically of finite length. In fact, this has been proven if the difference of the Hodge-Tate-Sen weights of $\rho$ is not a non-zero integer, see \cite[Thm.~1.2]{DPS22}.

We want to apply the criterion Theorem \ref{thm:finitelength} to
\[\JL(\rho):=\Hom_{G_{\Q_p}}\big(\rho(-1),\check{\cS}^1(\Pi(\rho))\big).
\] However, it is not clear to us whether $\JL(\rho)$ is non-zero or $\check{\cS}^1(\Pi(\rho))\cong\rho(-1)\otimes\JL(\rho)$, and even this is the case, we still don't know whether $\JL(\rho)^{\rm lalg}$ is finite dimensional over $E$. Fortunately, both the questions have a positive answer \emph{if $\rho$ has a global origin}.
\medskip 

Let us fix the global setup, closely following \cite{Paskunas-JL}. 
Let $F$ be a totally real number field with a fixed place $\p$ above $p$ such that $F_{\p} \cong \Q_p$. Let $B_0$ be a quaternion algebra with centre $F$, ramified at all the infinite places of $F$ and split at $\p$. Fix a maximal order $\cO_{B_0}$ of $B_0$ and  an isomorphism $(\cO_{B_0})_v\cong M_2(\cO_{F_v})$ for each  place $v\notin \Sigma$, where $\Sigma$ denotes the set of finite places of $F$ at which $B_0$ is ramified.   Let $U=\prod_vU_v$ be a compact open subgroup contained in $\prod_v(\cO_{B_0})_v^{\times}$. Assume that $U_{\p}=\GL_2(\Z_p)$ and that $U_v$ is a pro-$p$ group at other places above $p$. We may assume that $U$ is sufficiently small by shrinking $U_v$ at any place $v\neq \p$. 

Write $U^{p}=\prod_{v\nmid p}U_v$.
Let $S(U^p,\cO)$ be the space of continuous functions 
\[f:B_0^{\times}\backslash (B_0\otimes_F\mathbb{A}_F^f)^{\times}/U^p\ra \cO.\]  
Let $\psi:(\mathbb{A}_F^f)^{\times}/F^{\times}\ra \cO^{\times}$ be a continuous character such that $\psi$ is trivial on $(\mathbb{A}_F^f)^{\times}\cap U^p$. 
Let \[S_{\psi}(U^p,\cO):=\big\{f\in S(U^p,\cO)\ |\   f(gz)=\psi(z)f(g),\ \forall z\in (\mathbb{A}_F^{f})^{\times} \big\}.\] 
If $\lambda$ is a continuous representation of $U_p^{\p}=\prod_{v|p,v\neq \p}U_v$ on a finite free $\cO$-module, denote by \[S_{\psi,\lambda}(U^{\p},\cO):=\Hom_{U_p^{\p}}(\lambda,S_{\psi}(U^p,\cO)).\]
The group $(B_0\otimes_FF_{\p})^{\times}\cong G$ acts continuously on $S_{\psi}(U^p,\cO)$ and $S_{\psi,\lambda}(U^{\p},\cO)$ by right translation.

Let $S$ be a finite set of finite places of $F$ containing $\Sigma$, all the places above $p$, all the places $v$ where $U_v$ is not maximal and all the ramified places of $\psi$. Let $\mathbb{T}_S^{\rm univ}=\cO[T_v,S_v]_{v\notin S}$ be the abstract Hecke algebra, which acts on $S_{\psi}(U^p,\cO)$ and $S_{\psi,\lambda}(U^{\p},\cO)$  in the usual way (as in \cite[\S5]{Paskunas-JL}). 

 Let $G_{F,S}$ be the absolute Galois group of the maximal extension of $F$ in $\overline{F}$ which is unramified outside $S$. Fix a continuous absolutely irreducible odd representation 
\[\overline{r}: G_{F,S}\ra \GL_2(\F).\]
This gives rise to a maximal ideal $\fm$ of $\mathbb{T}^{\rm univ}_{S}$, given as the kernel of the map $\mathbb{T}^{\rm univ}_S\ra \F$ sending $T_v$ to $\mathrm{tr}(\overline{r}(\mathrm{Frob}_v))$ and $S_v$ to $q_v^{-1}\det(\overline{r}(\mathrm{Frob}_v))$, where $\mathrm{Frob}_v \in G_{F,S}$ is a geometric Frobenius element and $q_v$ is the cardinality of the residue field at $v$. \medskip

Let $B$ the quaternion algebra over $F$, which is ramified at $\p$ and splits at a fixed infinite place $\infty_F$, and has the same ramification behaviour as $B_0$ at all the other places. Fix an isomorphism
\[B_0\otimes_F\mathbb{A}_F^{\p,\infty_{F}}\cong B\otimes_F\mathbb{A}_F^{\p,\infty_F}.\]
This allows to view $U^{\p}$ as a  subgroup  of $(B\otimes_{F}\mathbb{A}_F^f)^{\times}$. 
Then $D:=B\otimes_FF_{\p}$ is the non-split quaternion algebra over $F_{\p}\cong \Q_p$. Define
\[\widehat{H}^1(U^{p},\cO):=\varprojlim_{n\geq 1}\varinjlim_{K_{p}}H^1_{\et}(X(K_{p}U^{p})_{\overline{F}},\cO/\varpi^n),\]
where $K_{p}$ runs over compact open subgroups of $\prod_{v|p} (\cO_B)_v^{\times}$ and $X(K_{p}U^{p})$ is the corresponding Shimura curve for $B$ defined over $F$.
For $\psi$ and $\lambda$ as above, we may analogously define 
$\widehat{H}^1_{\psi}(U^{p},\cO)$ and $\widehat{H}^1_{\psi,\lambda}(U^{\p},\cO)$, see \cite[\S6]{Paskunas-JL}. 

The abstract Hecke algebra $\mathbb{T}^{\rm univ}_{S}$ naturally acts on $H^1_{\et}(X(K_{p}U^{p})_{\overline{F}},\cO)$. Let $\mathbb{T}(K_pU^p)$ denote the image of $\mathbb{T}^{\rm univ}_{S}$ in $\End\big(H^1_{\et}(X(K_{p}U^{p})_{\overline{F}},\cO)\big)$ and $\mathbb{T}(K_pU^p)_{\fm}$ be its $\fm$-adic completion. Let
\[\mathbb{T}(U^{p})_{\fm}:=\varprojlim_{K_p}\mathbb{T}(K_pU^p)_{\fm}.\]
It acts  faithfully and continuously on $\widehat{H}^1(U^p,\cO)_{\fm}$ and preserves $\widehat{H}^1_{\psi,\lambda}(U^{\p},\cO)_{\fm}$. Finally we let  $\mathbb{T}(U^{\p})_{\fm}$ be the image of $\mathbb{T}(U^{p})_{\fm}$ in $\End\big(\widehat{H}^1_{\psi,\lambda}(U^{\p},\cO)_{\fm}\big)$. 

One can prove as in \cite[Cor.~7.3]{Scholze} that the Hecke algebra  $\mathbb{T}(U^{\p})_{\fm}$ also acts continuously on $S_{\psi,\lambda}(U^{\p},\cO)_{\fm}.$  Based on \cite[Thm.~6.2]{Scholze}, \cite[Prop.~6.3]{Paskunas-JL} proves that there is an isomorphism of $\mathbb{T}(U^{\p})_{\fm}[G_{F_{\p}}\times D^{\times}]$-modules
\begin{equation}\label{eq:lg-comp}\check{\cS}^1(S_{\psi,\lambda}(U^{\p},\cO)_{\fm}^d)\cong \widehat{H}^1_{\psi,\lambda}(U^{\p},\cO)_{\fm}^d.\end{equation}
As a consequence,  $S_{\psi,\lambda}(U^{\p},\cO)_{\fm}$ is also faithful over $\mathbb{T}(U^{\p})_{\fm}.$

From now on, we assume that $\widehat{H}^1_{\psi,\lambda}(U^{\p},\cO)_{\fm}\neq0$, i.e.~$\overline{r}$ is modular. This is equivalent to $S_{\psi,\lambda}(U^{\p},\cO)_{\fm}\neq 0$ by \eqref{eq:lg-comp} and the proof of \cite[Thm.~6.5]{Paskunas-JL}. It is well-known that there is an associated $2$-dimensional Galois representation
\[r_{\fm}: G_{F,S}\ra \GL_2(\mathbb{T}(U^{\p})_{\fm})\]
deforming $\overline{r}$ (see \cite[Prop.~5.7]{Scholze}). For any $y\in \mathrm{Spm}(\mathbb{T}(U^{\p})_{\fm}[1/p])$, $r_{\fm}$ specializes to a continuous Galois representation $r_{y}:G_{F,S}\ra \GL_2(\kappa(y))$, where $\kappa(y)$ denotes the residue field of $\mathbb{T}(U^{\p})_{\fm}[1/p]$ at $y$.
 
\begin{proposition}\label{prop:isotypic}
Let $y\in \mathrm{Spm}(\mathbb{T}(U^{\p})_{\fm}[1/p])$ be such that  $\rho_y:=r_y(1)|_{G_{F_{\p}}}$ is irreducible. Then 
\begin{equation}\label{eq:n-copy}
\big(S_{\psi,\lambda}(U^{\p},\cO)_{\fm}\otimes_{\cO}E\big)[\fm_{y}]\cong \Pi(\rho_y)^{\oplus n}.\end{equation}
for some integer $n\geq 0$, where $\fm_y\subset \mathbb{T}(U^{\p})_{\fm}[1/p]$ denotes the maximal ideal associated to $y$.
\end{proposition}
\begin{proof}
The result is proved in \cite[Cor.~5.7]{Paskunas-JL}. Remark that our normalization  is different from that of \cite[\S2]{Paskunas-JL}; precisely, the Banach space representation corresponding to $\rho$ in  \emph{loc.~cit.} is equal to $\Pi(\rho(1))$. 
\end{proof}

Let $\brho:=\overline{r}(1)|_{G_{F_{\p}}}$ 
and $\sB$ be the block corresponding to $\brho^{\rm ss}$.
 
\begin{proposition}\label{prop:lalg}
Assume that $\sB$ is not of type (IV). Let $y\in \mathrm{Spm}(\mathbb{T}(U^{\p})_{\fm}[1/p])$ be such that  $\rho_y:=r_y(1)|_{G_{F_{\p}}}$ is irreducible and  $\big(S_{\psi,\lambda}(U^{\p},\cO)_{\fm}\otimes_{\cO}E\big)[\fm_{y}]\neq 0.$ Then $\check{\cS}^1(\Pi(\rho_y))^{\rm lalg}$ is finite dimensional and $\check{\cS}^1(\Pi(\rho_y))$ is $\rho_y(-1)$-isotypic, i.e. \begin{equation}\label{eq:isotypic}\check{\cS}^1(\Pi(\rho_y))\cong\rho_y(-1)\otimes \JL(\rho_y)\end{equation}
for some admissible unitary Banach space representation $\JL(\rho_y)$ of $D^{\times}$.
\end{proposition} 
 
\begin{proof}
The first assertion is proved  in \cite[Prop.~6.15]{Paskunas-JL}. The running assumption in \emph{loc.~cit.} requires $\brho$ to be reducible, but this is not necessary in the proof.

By the assumption on $\brho$ and using \eqref{eq:lg-comp}, we get an isomorphism of $G_{F_{\p}}\times D^{\times}$-representations
\[\check{\cS}^1\big((S_{\psi,\lambda}(U^{\p},\cO)_{\fm}\otimes_{\cO}L)[\fm_y]\big)\cong \big(\widehat{H}^1_{\psi,\lambda}(U^{\p},\cO)_{\fm}\otimes_{\cO}E\big)[\fm_y],\]
see the proof of \cite[Thm.~6.11]{Paskunas-JL}. On the other hand, combining Propositions 5.3, 5.4, 5.8 of \cite{Scholze}, we obtain that $\widehat{H}^1_{\psi,\lambda}(U^{\p},\cO)_{\fm}$ is $r_{\fm}$-isotypic and $\big(\widehat{H}^1_{\psi,\lambda}(U^{\p},\cO)_{\fm}\otimes_{\cO}E\big)[\fm_y]$ is $r_{y}$-isotypic.  Hence, 
 $\check{\cS}^1(\Pi(\rho_y))$ is also $r_y$-isotypic by Proposition \ref{prop:isotypic}; we need the assumption $\big(S_{\psi,\lambda}(U^{\p},\cO)_{\fm}\otimes_{\cO}E\big)[\fm_{y}]\neq 0$ to ensure that $n>0$ in \eqref{eq:n-copy}. Letting  
 \[\JL(\rho_y):=\Hom_{G_{F_{\p}}}\big(\rho_y(-1),\check{\cS}^1(\Pi(\rho_y))\big),\] 
we obtain  \eqref{eq:isotypic}. 
\end{proof}

\begin{lemma}\label{lem:flat}
Assume that $\sB$ is of type (I) or (II), and is generic if $\sB$ is of type (I). Then both  $S_{\psi,\lambda}(U^{\p},\cO)_{\fm}^d$ and $\widehat{H}^1_{\psi,\lambda}(U^{\p},\cO)_{\fm}^d$ are flat over $\mathbb{T}(U^{\p})_{\fm}.$
\end{lemma}
\begin{proof}
This is proved by the same argument of \cite[Thm.~A(1)]{Gee-Newton} by using the the miracle flatness criterion (\cite[Prop.~A.30]{Gee-Newton}). See \cite[Prop.~8.10]{HW-JL1} for the case of $\widehat{H}^1_{\psi,\lambda}(U^{\p},\cO)_{\fm}^d.$ The case of  $S_{\psi,\lambda}(U^{\p},\cO)_{\fm}^d$ is similar.
\end{proof}
 
Our main result is the following.
\begin{theorem}\label{thm:gl-main}
Assume that $\sB$ is of type (I) or (II), and is generic.
Let $y\in \mathrm{Spm}(\mathbb{T}(U^{\p})_{\fm}[1/p])$ be such that  $\rho_y:=r_y(1)|_{G_{F_{\p}}}$ is irreducible. Then $\JL(\rho_y)$ is topologically of finite length, and the quotient $\JL(\rho_y)/\JL(\rho_y)^{\rm lalg}$ is topologically irreducible. 
\end{theorem}
\begin{proof}
 First, Lemma \ref{lem:flat} implies that \[\big(S_{\psi,\lambda}(U^{\p},\cO)_{\fm}\otimes_{\cO}E\big)[\fm_{y}]\neq 0\] for any $y\in \mathrm{Spm}(\mathbb{T}(U^{\p})_{\fm}[1/p])$.
By Proposition \ref{prop:lalg},  $\check{\cS}^1(\Pi(\rho_y))^{\rm lalg}$ is finite dimensional over $E$, and so is $\JL(\rho_y)^{\rm lalg}$ using \eqref{eq:isotypic}.  To conclude, we check that $\JL(\rho_y)$ satisfies the (other) conditions of  Corollary \ref{cor:finitelength}.

Since $\overline{r}:G_F\ra \GL_2(\F)$ is irreducible, there exists a unique (up to homothety) $G_{F,S}$-invariant lattice in $r_y$, which induces a $G_{F_{\p}}$-invariant lattice in $\rho_y$ whose mod $\varpi$ reduction is just $\brho$. This lattice in $\rho_y$ corresponds to an open bounded $G$-invariant lattice in $\Pi(\rho_y)$, which we denote by $\Theta$. By the compatibility between $p$-adic and mod $p$ local Langlands correspondence for $G$ (\cite{Berger}), we have \[\Theta/\varpi\Theta\cong \pi(\brho).\] 
The assumption on $\brho$ implies that $\cS^0(\pi(\brho))=\cS^2(\pi(\brho))=0$, by Theorem \ref{thm:Scholze}. Hence,  Proposition \ref{prop:S1-Theta} implies that $\check{\cS}^1(\Theta^d)^d$ is an open bounded $D^{\times}$-invariant lattice in $\check{\cS}^1(\Pi(\rho_y))$, with mod $\varpi$ reduction $\cS^1(\pi(\brho))$. 

Again by the assumption on $\brho$, 
we have  $\cS^1(\pi(\brho))\cong \brho(-1)\otimes\JL(\brho)$  
by Theorem \ref{thm:multione} and $\cS^1(\pi(\brho))\in\mathcal{C}$ by \cite[Cor.~6.10]{HW-JL1}. Moreover, Theorem \ref{thm:gr-pi} implies that $\JL(\brho)^{\vee}$ is a Cohen-Macaulay $\F[\![U_D^1/Z_D^1]\!]$-module (as its  graded module is) and
\[ \mu(\JL(\brho))=4.\] 
Using \eqref{eq:isotypic} and Lemma \ref{lem:lattice-mu}, all the conditions of Corollary \ref{cor:finitelength} are satisfied for $\JL(\rho_y)$. This finishes the proof.
\end{proof}

\begin{remark}
We hope that Theorem \ref{thm:gl-main} remains true in other global situations. The reason to restrict to the current setup is the easy citation to \cite{Scholze} and \cite{Paskunas-JL}.
\end{remark}

\section{Appendix: Gabber filtrations} \label{sec:app}
In this appendix, we recall from \cite{Bjork} the construction of Gabber filtrations on pure modules over a filtered Auslander regular ring.
\subsection{Auslander regular rings} \label{sec:Auslander}

Let $\Lambda$ be an Auslander regular ring and $M$ be a  finitely generated $\Lambda$-module. We define the \emph{grade} of $M$ as  
\[j_{\Lambda}(M):=\mathrm{inf}\{i\in\N\ | \ \Ext^i_{\Lambda}(M,\Lambda)\neq0\},\]
with the convention $j_{\Lambda}(M):=\infty$ if $M=0$.
Let $d$ be the  global dimension of $\Lambda$. Set \[\dim_{\Lambda}(M):=d-j_{\Lambda}(M)\]
and call it the \emph{dimension} of $M$. 

\begin{definition}\label{def:CM}
We say $M$ is \emph{Cohen-Macaulay} if   $\Ext^i_{\Lambda}(M,\Lambda)=0$ for any $i\neq j_{\Lambda}(M)$, and  $M$ is \emph{pure} if any nonzero submodule has the same grade as $M$. 
\end{definition}

\begin{lemma}\label{lem:pure=CM}
Cohen-Macaulay  $\Lambda$-modules are pure. Conversely, a pure $\Lambda$-module of grade $d-1$ is Cohen-Macaulay.
\end{lemma}

\begin{proof}
The first statement is well-known. For the second, 
let $M$ be a pure finitely generated $\Lambda$-module. It is equivalent to prove that $\Ext^{d}_{\Lambda}(M,\Lambda)=0$. If not, then $\Ext^d_{\Lambda}(M,\Lambda)$ has grade $d$. But, since $M$ is pure, it follows from \cite[Chap.~III, Thm.~4.2.6]{LiO} that $\Ext^d_{\Lambda}(\Ext^d_{\Lambda}(M,\Lambda),\Lambda)=0$, a contradiction. \end{proof}

\begin{remark}\label{rem:CMquotient}
Let $M$ be a finitely generated $\Lambda$-module of grade $d-1$. By  \cite[\S3.1]{Ven}, $M$ has a unique submodule $M_d$ whose grade is $d$ (equivalently $M_d$ is finite dimensional over $\F$). Then $M/M_d$ is pure of grade $d-1$, hence  Cohen-Macaulay by Lemma \ref{lem:pure=CM}.
\end{remark}

\subsection{Gabber filtrations} 
Let $R_{\Lambda}$ denote the \emph{Rees ring} associated to $\Lambda$, which is a graded subring of $\Lambda[T,T^{-1}]$ (here $T$ is a formal variable of degree $1$). Given a  good filtration $F$ on $M$ (in the sense of \cite[\S1.5]{LiO}), let $M_F$ denote the associated \emph{Rees module} which is a graded module over $R_{\Lambda}$. It is known that $M_F$ is $T$-torsion-free and we have  a natural isomorphism
\begin{equation}\label{eq:Rees-gr}
M_{F}/TM_{F}\cong \gr_{F}(M). \end{equation}
See \cite[\S4.5]{Bjork} or \cite[Chap.~I, \S4.3]{LiO} for more details.

It is proved in \cite[Chap.~III, Thm.~3.1.7]{LiO} that if $\Lambda$ and $\gr(\Lambda)$ are Auslander regular rings, then so is $R_{\Lambda}$. From now on, we assume that $\Lambda$ and $\gr(\Lambda)$   are both Auslander regular rings. 

Recall the following definition from \cite[Def.~5.22]{Bjork}. 

\begin{definition}\label{def:bjork}
Let $M$ be a pure finitely generated $\Lambda$-module. Two good filtrations $F$ and $F'$ of $M$ are called \emph{tamely close} if $j_{R_{\Lambda}}(M_{F+F'}/M_{F\cap F'})\geq j_{\Lambda}(M)+2$.
\end{definition}
Here, the sum and the intersection of two filtrations are defined to be 
\[(F+F')^nM:=F^nM+F'^nM,\ \ (F\cap F')^nM:=F^nM\cap F'^nM\]
respectively. 
It is easy to see that being tamely close  defines an equivalence relation on the family of good filtrations on $M$. 
\begin{theorem}\label{thm:bjork}
Let $M$ be a pure finitely generated $\Lambda$-module. Then every equivalence class of good filtrations on $M$ contains a unique $F$ such that $\gr_F(M)$ is a pure $\gr(\Lambda)$-module of grade equal to $j_{\Lambda}(M)$.  Moreover, if $F'$ is another good filtration in the equivalence class of $F$, then $F'<F$ in the sense that $F'^nM\subseteq F^nM$ for all $n\in\Z$, and the kernel of the induced  morphism $\gr_{F'}(M)\ra \gr_{F}(M)$ is identified with the largest submodule of $\gr_{F'}(M)$ of grade $>j_{\Lambda}(M)$. In particular, the morphism $\gr_{F'}(M)\ra \gr_F(M)$ is non-zero.
\end{theorem}
The unique filtration in Theorem \ref{thm:bjork} is referred as a \emph{Gabber filtration}.
\begin{proof} 
The first statement and  the fact $F'<F$ are proved in \cite[Thm.~5.23]{Bjork}.  For the convenience of the reader, we recall the proof. Below we write $n:=j_{\Lambda}(M)=j_{\gr(\Lambda)}(\gr_F(M))$; here the second equality is a standard fact, see \cite[Chap.~III, Thm.~2.5.2]{LiO}. 

First,  the construction of  $F$ in a given equivalence class of good filtrations such that $\gr_F(M)$ is pure is well-known, see \cite[Lem.~5.19, Cor.~5.21]{Bjork}, or \cite[Chap.~III, Thm.~4.2.13(4)]{LiO} for a more detailed presentation. For the uniqueness of $F$, we need to show that \[M_{F}=M_{F+F'}\] for any good filtration $F'$ which is tamely close to $F$. Put $Q=M_{F+F'}/M_{F}$. Then $j_{R_{\Lambda}}(Q)\geq n+2$ by  Definition \ref{def:bjork}. 
Using the $T$-torsion-freeness of $M_{F+F'}$ and \eqref{eq:Rees-gr},  multiplication by $T$ induces an exact sequence
\[0\ra \Ker_T(Q)\ra \gr_{F}(M)\ra \gr_{F+F'}(M)\ra Q/TQ\ra0,\]
where $\Ker_T(Q)$ denotes the kernel of $T:Q\ra Q$. Note that since $R_{\Lambda}/T\cong \gr(\Lambda)$, $\Ker_T(Q)$ has a $\gr(\Lambda)$-module structure. Since $T$ is a central non-zero-divisor in $R_{\Lambda}$, it is easy to see that 
\[j_{\gr(\Lambda)}(\Ker_T(Q))=j_{R_{\Lambda}}(\Ker_T(Q))-1\geq j_{R_{\Lambda}}(Q) -1 \geq n+1,\]
where the first inequality holds as $\Ker_T(Q)\subseteq Q$.  
However, $\gr_F(M)$ is pure of grade $n$ by construction, any non-zero submodule of $\gr_F(M)$ also has grade $n$. This forces \[\Ker_T(Q)=0.\]  
On the other hand, since $F$ and $F+F'$ are both good filtrations on $M$, we have $M_{F}[T^{-1}]=M_{F+F'}[T^{-1}]$ (here $[T^{-1}]$ means taking localization at the multiplicative subset $\{T^{k}, k\geq 0\}$). This implies that $Q$ is a $T$-torsion $R_{\Lambda}$-module. Combined with the fact $\Ker_T(Q)=0$, we get $Q=0$, as required.

The last assertion essentially follows from the above argument. More precisely, since $F'<F$, we may consider the following  short exact sequence
\[0\ra M_{F'}\ra M_{F}\ra M_{F}/M_{F'}\ra0.\]
As above,  multiplication by $T$ induces an exact sequence
\[0\ra \Ker_T(M_{F}/M_{F'})\ra \gr_{F'}(M)\ra \gr_{F}(M).\]
Since  $j_{\gr(\Lambda)}(\Ker_T(M_{F}/M_{F'}))\geq n+1$ and $\gr_{F}(M)$ is pure of grade $n$,   the result easily follows.
\end{proof}



\newcommand{\etalchar}[1]{$^{#1}$}
\providecommand{\bysame}{\leavevmode\hbox to3em{\hrulefill}\thinspace}
\providecommand{\MR}{\relax\ifhmode\unskip\space\fi MR }
\providecommand{\MRhref}[2]{%
  \href{http://www.ams.org/mathscinet-getitem?mr=#1}{#2}
}
\providecommand{\href}[2]{#2}

\bigskip

\noindent Morningside Center of Mathematics, Academy of Mathematics and Systems Science,
 Chinese Academy of Sciences, University of the Chinese Academy of Sciences
Beijing, 100190,
China.\\
{\it E-mail:} {\ttfamily yhu@amss.ac.cn}\\

\noindent

\noindent  Yau Mathematical Sciences Center, Tsinghua University, Beijing, 100084\\
{\it E-mail:} {\ttfamily haoranwang@mail.tsinghua.edu.cn}\\

\end{document}